\documentclass[10pt]{amsart}
\usepackage{hyperref}

\usepackage{amsmath,amsfonts,amssymb,amsthm,epsfig,color, tikz, verbatim, amscd}
\usepackage{graphicx} 
\usepackage{scalerel,stackengine}
\usepackage{tikz-cd}

\newcommand{\Span}{\operatorname{span}}
\newcommand{\bb}{\mathbb}
\newcommand{\Aff}{\operatorname{Aff}}
\newcommand{\C}{\bb C}

\newcommand{\Z}{\bb Z}
\newcommand{\R}{\bb R}
\newcommand{\N}{\bb N}

\newcommand{\Q}{\bb Q}

\usepackage{float}

\newcommand{\SL}{\operatorname{SL}}
\newcommand{\GL}{\operatorname{GL}}

\newcommand{\FA}{\operatorname{FA}}

\newcommand{\T}{\mathcal T}
\newcommand{\TT}{\mathbb T}
\newcommand{\hh}{\mathcal H}
\newcommand{\om}{\omega}
\newcommand{\Om}{\Omega}
\newcommand{\bPsi}{\boldsymbol{\Psi}}

\newcommand{\minuszero}{\backslash\{0\}}

\newtheorem{Theorem}{Theorem}
\newtheorem{ques}[Theorem]{Question}
\newtheorem{Cor}[Theorem]{Corollary}
\newtheorem{Prop}[Theorem]{Proposition}
\newtheorem{lemma}[Theorem]{Lemma}
\newtheorem{conj}[Theorem]{Conjecture}
\newtheorem{rmk}[Theorem]{Remark}

\newcommand{\JA}{\textcolor{blue}}

\newcommand{\WPH}{\textcolor{green}}

\usepackage{amsrefs}
\usepackage{mathtools}

\newcommand{\spec}{\mathcal S}
\newcommand{\genspec}{\widehat{\mathcal S}}
\numberwithin{equation}{section}
\numberwithin{Theorem}{section}

\begin{document}
\title{Linear Flows on Translation Prisms}
\author[Athreya]{Jayadev S.~Athreya}
\address{Department of Mathematics, University of Washington, Box 354350, Seattle, WA 98195, USA}
\email{jathreya@uw.edu}
\author[B\'edaride]{Nicolas B\'edaride}
\address{Aix-Marseille Universit\'e,
Institut de Math\'ematiques de Marseille (I2M), UMR 7373, 3 place Victor Hugo, Case 19
13331 Marseille Cedex 3, France}
\email{nicolas.bedaride@univ-amu.fr}
\author[Hooper]{W.~Patrick Hooper}
\address{Department of Mathematics, The City College of New York, 160 Convent Ave, New York, NY 10031, USA}
\address{Department of Mathematics, CUNY Graduate Center, 365 5th Ave, New York, NY 10016, USA}
\email{whooper@ccny.cuny.edu}
\author[Hubert]{Pascal Hubert}
\address{Aix-Marseille Universit\'e,
Institut de Math\'ematiques de Marseille (I2M), UMR 7373, 3 place Victor Hugo, Case 19
13331 Marseille Cedex 3, France}
\email{pascal.hubert@univ-amu.fr}

\begin{abstract} Motivated by the study of billiards in polyhedra, we study linear flows in a family of singular flat $3$-manifolds which we call \emph{translation prisms}. Using ideas of Furstenberg~\cite{Furstenberg} and Veech~\cite{Veech:metrictheory}, we connect results about weak mixing properties of flows on translation surfaces to ergodic properties of linear flows on translation prisms, and use this to obtain several results about unique ergodicity of these prism flows and related billiard flows. Furthermore, we construct explicit eigenfunctions for translation flows in pseudo-Anosov directions with Pisot expansion factors, and use this construction to build explicit examples of non-ergodic prism flows, and non-ergodic billiard flows in a right prism over a regular $n$-gons for $n=7, 9, 14, 16, 18, 20, 24, 30$.
\end{abstract}

\maketitle

\textit{Dedicated to Giovanni Forni on his 60th birthday}

\section{Introduction and Statement of Results}\label{sec:intro}

\subsection{Billiards in prisms}\label{sec:billiards:prisms} 

%

Let $P \subset \C$ be a rational polygon, that is, a polygon whose angles are rational multiples of $\pi$, or, equivalently, a polygon such that the group generated by reflections in the sides of $P$ is finite modulo translations. This is a condition that emerges naturally when studying \emph{billiards} in polygons, that is, the movement of a point mass at unit speed with elastic collisions with sides. For rational polygons, the phase space $P \times S^1$ decomposes into invariant surfaces for the billiard flow, and so we will refer to the billiard flow in direction $\theta$ as $B^{\theta}_t$. Let $\mathcal P = P \times [0,1]$ be the right regular prism over $P$ with height $1$. Consider the movement of a point mass in the interior of $\mathcal P$, with elastic collisions with faces. We parameterize our initial velocities for particles by pairs $(\theta, s)$ as $S^1 \times \R$, which corresponds to the velocity vector $\mathbf{v}(\theta, s) = (\cos \theta, \sin \theta, s)$, and refer to $s$ as the vertical speed. We consider the billiard flow $B^{\theta, s}_{t}$ on $\mathcal P$ defined by letting $B^{\theta, s}_t(p)$ denote the position of the particle starting at  $p \in \mathcal P$ with initial velocity $\mathbf{v}(\theta, s)$ after time $t$. Applications of our results to the study of billiards include:


\begin{Theorem}\label{theorem:prisms} Let $\mathcal P = P \times [0, 1]$ be a right-angled prism over a rational polygon $P \subset \C$. 
\begin{enumerate}
    \item For almost every $(\theta, s) \in S^1 \times \R$, the billiard flow $\{B^{\theta, s}_t\}$ is uniquely ergodic. In particular, for every $p \in \mathcal P$ whose forward orbit under $B^{\theta, s}_t$ is defined for all $t>0$, the trajectory $\{B^{\theta, s}_t(p)\}_{t >0}$ becomes equidistributed (with respect to the Lebesgue probability measure $m_{\mathcal P}$) in $\mathcal P$, that is, for any measurable set $A \subset \mathcal P$, $$\frac{1}{T}|\{0 \le t \le T: B^{\theta, s}_t(p) \in A\}| \rightarrow m_{\mathcal P}(A).$$

 \item There are explicit polygons $P$ and non-equidistributing directions with non-zero vertical speed. For example, if $P=P_7$, the regular $7$-gon with a horizontal side, and 
 $\theta = \arctan\big( 6 \sin(\tfrac{3 \pi}{7}) + 6 \sin(\tfrac{2 \pi}{7}) + 2 \sin(\tfrac{\pi}{7})\big)$, there is a rank $3$ subgroup of real numbers $E_7$ so that for any $s \in E_7$, the flow $B^{\theta,s}$ is non-minimal on the prism $\mathcal P_7$, although the billiard flow in direction $\theta$ on the polygon $P_7$ is uniquely ergodic. 

 \end{enumerate}
\end{Theorem}

\paragraph*{\bf Almost sure results} Beck-Chen-Yang~\cite{BCY} showed, using very different methods, that for almost every initial data $(p, \theta, s) \in \mathcal P \times S^1 \times \R$, the trajectory $\{B^{\theta, s}_t(p)\}_{t >0}$ becomes equidistributed (with respect to the Lebesgue probability measure) in $\mathcal P$. Our result applies to \emph{every} initial point (with well-defined forward trajectory).


\subsection{Translation surfaces and translation prisms}\label{sec:prisms} The key tool in our study of billiards in prisms will be the study of linear flows on \emph{translation prisms}. Let $S=(X,\om)$ be a translation surface, $\{\phi_t\}_{t \in \R}$ the vertical flow on $S$, and $\{\phi_t^{\theta}\}_{t \in \R}$ the flow in direction $\theta$ from the vertical on $S$. We define the \emph{translation prism}, $M = S \times (\R/\Z),$ and given $s >0$, define the \emph{prism flow} $\{\Phi^{\theta, s}_t\}_{t \in \R}$ on $M$ by \begin{equation}\label{eq:prismflow} \Phi^{\theta, s}_t(p, \tau) = (\phi^{\theta}_t(p), \tau + ts) \end{equation} $\Phi^{\theta, s}_t$ preserves Lebesgue measure $m$, given by $dm = dp d\tau$ on $M$. We refer to this as the prism flow of speed $s$ over the flow in direction $\theta$. We are interested in the \emph{ergodic} properties of $\Phi^{\theta, v_0}_t$. We will write $\Phi^{s}_t$ for the prism flow of speed $s$ over the vertical unit speed flow $\phi_t$. The study of these flows was introduced by Forni~\cite{Forni}, who calls them \emph{twisted translation flows}. 

\paragraph*{\bf Unfolding} Just as billiard flows in rational polygons $P$ yield, via a classical unfolding procedure of Fox-Kershner~\cite{FoxKershner} and Zelmjakov-Katok~\cite{KatokZelmjakov}, linear flows on a translation surface $S_P$, the billiard flow in the prism $\mathcal P = P \times [0, 1]$ yield a prism flow on the translation prism $M_{\mathcal P}$ over $S_P$. To be completely precise, $M_{\mathcal P}$ should be the translation prism of height $2$, that is, $S \times (\R/2\Z),$ but we will, for ease of notation, use $M_{\mathcal P}$ to denote the prism of height $1$, and then billiard flows of vertical speed $s$ on $\mathcal P$ will correspond to flows with vertical speed $2s$ on the translation prism $M_{\mathcal P}$.

\subsection{Unique ergodicity and minimality}\label{sec:ue} Classical ergodic theory relates the ergodicity (respectively minimality) of $\Phi^{\theta, s}_t$ to the measurable (respectively continuous) eigenvalues of the flow $\phi_t$. We say $\kappa \in \R$ is a measurable (respectively continuous) \emph{eigenvalue} for $\phi_t$ if there is a measurable (respectively continuous) $\psi: S \rightarrow \C$ such that for almost every $p \in S$, \begin{equation}\label{eq:eigenvalue} \psi(\phi_t p) = e^{2\pi i t \kappa} \psi(p). \end{equation} We denote the set of measurable eigenvalues for $\phi_t$ by $\mathcal S^m(\om)$, the subset of continuous eigenvalues by $\mathcal S^c(\om)$ and the sets of measurable and continuous eigenvalues for $\phi_t^{\theta}$ by $\mathcal S^m_{\theta}(\om)$ and $\mathcal S^c_{\theta}(\om)$ respectively.  We define the \emph{generalized spectrum} $\genspec^m(\om)$ by \begin{equation}\label{eq:genspec} \genspec^m(\om) = \bigcup_{n \in \Z\minuszero} \frac{1}{n} \spec^m(\om),\end{equation} the \emph{generalized continuous spectrum} by \begin{equation}\label{eq:genspeccont} \genspec^c(\om) = \bigcup_{n \in \Z\minuszero} \frac{1}{n} \spec^c(\om),\end{equation} and $\genspec^m_{\theta}(\om)$ and $\genspec^c_{\theta}(\om)$ similarly. Note that for $\ast = c, m$ $\mathcal S^{\ast}_{\theta}(\om) = \mathcal S^{\ast}(e^{-i\theta} \om)$ and $\genspec^{\ast}_{\theta}(\om) = \genspec^{\ast}(e^{-i\theta}\om)$. Our first main result combines ideas of Furstenberg~\cite{Furstenberg} and Veech~\cite{Veech:metrictheory} to obtain:

\begin{Theorem}\label{theorem:ue} Let $S=(X,\om)$ be a translation surface, $M = S \times (\R/\Z)$ the associated translation prism, and let $s >0$. Then the flow $\Phi^{s}_t$ is uniquely ergodic if and only if $s \notin \genspec^m(\om)$, that is, for every non-zero integer $n \in \Z\backslash\{0\}$, $n s \notin \mathcal S^m(\om)$. The flow $\Phi^{s}_t$ is minimal if and only if $s \notin \genspec^c(\om)$, that is, for every non-zero integer $n \in \Z\backslash\{0\}$, $n s \notin \mathcal S^c(\om)$. 
\end{Theorem}

\paragraph*{\bf Separability} Using the general fact that for measure-preserving flows, the spectrum is of Lebesgue measure $0$ (indeed we have a countable number of eigenvalues since $L^2(S,\mu)$ is a separable space), we have:

\begin{Cor}\label{cor:ae:ue}  Let $S=(X,\om)$ be a translation surface, $M = S \times (\R/\Z)$ the associated translation prism of height $1$. Then for almost every $s \in \R$, the flow $\Phi^{s}_t$ is uniquely ergodic.
\end{Cor}

\subsection{Applications of weak mixing}\label{sec:application:wm} Theorem~\ref{theorem:ue} connects the ergodic properties of prism flows to the spectral properties of the associated translation flow. We recall that a flow is called \emph{weak mixing} if its spectrum is reduced to zero. That is, the vertical flow on a translation surface $S= (X,\om)$ is weak mixing if $\mathcal{S}^m(\om) = \{0\}$, and the flow in direction $\theta$ is weak mixing if $\mathcal S_{\theta}^m(\om) = \{0\}$.

\paragraph*{\bf Generic surfaces} The main result of Avila-Forni~\cite{AvilaForni} states that for almost every $\om$, the vertical flow, and indeed the flow in almost every direction $\theta$) is weak mixing, so $\mathcal S(\om)$ is reduced to zero, and for almost every $\om$ and almost every $\theta$, $\mathcal S_{\theta}(\om)$ is reduced to zero. Combining this with Theorem~\ref{theorem:ue} yields the following corollary (see \S\ref{sec:iet} for precise definitions of strata $\hh$ of moduli spaces of translation surfaces, Masur-Smillie-Veech (MSV) measure $\mu_{\hh}$):

\begin{Cor}\label{corollary:ue:generic} Let $\hh$ be a connected component of a stratum of translation surfaces.  For $\mu_{\hh}$-almost every $S=(X, \om) \in \hh$, and every $s >0$,  the prism flow $\{\Phi^{s}_t\}_{t \in \R}$ is uniquely ergodic. Moreover, for $\mu_{\hh}$-almost every $S=(X, \om) \in \hh$, Lebesgue-almost every $\theta$ and every $s >0$,  the prism flow $\{\Phi^{\theta, s}_t\}_{t \in \R}$ is uniquely ergodic. 
\end{Cor}

\paragraph*{\bf Lattice surfaces} An important class of translation surfaces are \emph{lattice surfaces}, whose \emph{Veech group} (stabilizer under the natural $SL(2,\R)$-action) is a lattice. These surfaces form a set of measure $0$, so the Avila-Forni result does not apply. The main result of Avila-Delecroix~\cite{AvilaDelecroix} is that for \emph{non-arithmetic} lattice surfaces $S = (X, \om)$, and almost every direction $\theta$, the flow in direction $\theta$ is weak mixing. Combining this with Theorem~\ref{theorem:ue}, we obtain a unique ergodicity result for prism flows built over \emph{lattice surfaces}. 

\begin{Cor}\label{corollary:ue:veech} If $S = (X, \om)$ is a non-arithmetic lattice surface, then for Lebesgue-almost every $\theta$ and every $s >0$, the prism flow $\{\Phi^{\theta, s}_t\}_{t \in \R}$ is uniquely ergodic. 
\end{Cor}

\paragraph*{\bf The tautological plane} Recently, Arana-Herrera--Chaika--Forni~\cite{AHChaikaForni}*{Theorem 1.1} have given a very general characterization of weak-mixing for translation flows, showing that on any translation surface $S$, the existence of one weak-mixing direction is equivalent to weak-mixing for almost every direction, and further, this existence can be characterized by the absence of integer vectors in the \emph{tautological plane}, the plane in real cohomology $H^1(S, \R)$ spanned by the real and imaginary parts of the differential $\omega$. We have:

\begin{Cor}\label{corollary:ue:AHChaikaForni} For any translation surface $S$ with no integer vectors in the tautological plane, for Lebesgue-almost every $\theta$ and every $s >0$,  the prism flow $\{\Phi^{\theta, s}_t\}_{t \in \R}$ is uniquely ergodic. 
\end{Cor}

\subsubsection{Billiards in right regular prisms}\label{sec:billiards} As discussed above, one our main motivations for this family of problems comes from the study of billiards in polyhedra, in particular, right prisms over rational polygons.  Veech~\cite{Veech:dichotomy} showed that the translation surfaces obtained by unfolding the billiard in the regular $n$-gon $P_n$ are lattice surfaces, which allows us to obtain the following result:

\begin{Theorem}\label{theorem:regular:prism} Let $P_n$ denote the regular $n$-gon with $n\geq 2$, $S_n =(X_n, \om_n)$ the associated translation surface, and $M_n$ the associated translation prism. Then for every $\theta$ such that $\phi_t^{\theta}$ is ergodic on $S_n$, for almost every $s \in \R^+$, the prism flow $\Phi_{t}^{\theta, s}$ is uniquely ergodic. In particular, the billiard flow in the right regular prism $\mathcal P_n$ over $P_n$ is uniquely ergodic in almost every direction.
\end{Theorem}

\noindent For $n=5$, and $n \geq 7$,  this follows from Corollary~\ref{corollary:ue:veech}, as the translation surface $S_n$ associated to unfolding the billiard flow in $P_n$ is a non-arithmetic lattice surface. For $n = 3, 4, 6$ the translation surface $S_n$ is an \emph{arithmetic} lattice surface, and this result can be obtained through relatively elementary methods. In fact, for $n=3, 4$ $S_n$ is a torus.

\paragraph*{\bf Quantitative results} Forni~\cite{Forni}*{Theorem 1.5} gives an \emph{effective version} of equidistribution results for prism flows almost every surface $S$ and almost every direction, using results on \emph{twisted Birkhoff integrals} of translation flows developed by, among others, Avila-Forni-Safaee~\cite{AFS} and Bufetov-Solomyak~\cite{BufetovSolomyak}. It is an interesting open question to see if these quantitative results can be extended to the setting of lattice surfaces, and in particular to billiards in right regular prisms.

\subsection{Non-ergodicity}\label{sec:nonergodic} Our other main result is the construction of explicit non-minimal prism flows where the direction of the flow on the base translation surface $S$ is the expanding direction for an affine pseudo-Anosov map $\rho:S \rightarrow S$ with Pisot expansion factors, allowing us to use self-similarity properties to find explicit eigenvalues and eigenfunctions for the linear flow, via linear algebra and the action of $\rho^*$ on the cohomology $H^1(S, \R)$. 

\paragraph*{\bf Eigenfunctions} Let $\phi_t^u: S \rightarrow \S$ denote the (unit-speed) flow in the expanding direction of $\rho$. A (continuous) eigenfunction with eigenvalue $\lambda$ for the flow is a (continuous) function  $F: S \rightarrow S^1$ satisfying $$F(\phi_t^u x) = e^{2\pi \lambda t} f(x).$$ We can also work \emph{additively}: if we write $F(x) = e^{2\pi i f(x)}$, then $f: S \rightarrow \R/\Z$ satisfies $$f(\phi_t^u x) = f(x) + \lambda t.$$ We note that the set of eigenvalues forms a group under addition.

\paragraph*{\bf Pseudo-Anosov directions}  For the definition and background on pseudo-Anosov maps, see, for example, Farb-Margalit~\cite{FarbMarg} or Fathi-Laudenbach-Poenaru~\cite{FLP}. Given a pseudo-Anosov map $\rho$, the derivative $D \rho \in \SL(2,\mathbb R)$ is a hyperbolic matrix. Let $\beta \in \R$ be the expanding eigenvalue of $D\rho$, which we refer to as the \emph{expansion factor} of $\rho$.  Let $w^u, w^s \in {\mathbb R}^2$ be the unstable and stable unit eigenvectors for $D\rho$ satisfying $$(D\rho)w^u = \beta w^u \quad \text{and} \quad (D\rho)w^s = \beta^{-1} w^s.$$ By replacing $\rho$ with a power if necessary, we assume that that $\beta > 1$, and that $\rho$ fixes every singular point, and moreover, every separatrix (singular leaf in the stable ($w^s$) direction) on $S$. A number is \emph{Pisot} if it is an algebraic integer of absolute value greater than $1$ and all of its Galois conjugates have absolute value strictly less than $1$. Our main result for linear flows in the unstable direction of pseudo-Anosov maps states that these flows are non-weak mixing if and only if expansion factor is Pisot, and in this setting, there is a rank $d$ group of eigenvalues, which can be computed explicitly, and moreover, a simultaneous semi-conjugacy of the pseudo-Anosov map $\rho$ and the flow in the unstable direction to a hyperbolic automorphism $A$ and linear flow on a real $d$-torus, where $d$ is the degree of $\beta$. Precisely, we have:

\begin{Theorem}\label{theorem:pisotsemiconjugacy} Let $\rho: S \rightarrow S$ be a pseudo-Anosov map with expansion factor $\beta$, and let $\phi_t^u$ denote the unit-speed flow in the unstable direction. \begin{enumerate}

\item The flow $\phi_t^u$ is not weak-mixing if and only if $\beta$ is Pisot. 

\item If $\beta$ is Pisot of degree $d$, all eigenfunctions are continuous, and the group $E(\rho) \subset \R$ of eigenvalues of $\phi_t^u$ is isomorphic to $\Z^d$. 

\item Given any generating set $\mathbf{c} = (c_1, \ldots, c_d)$ for $E(\rho)$, there is a continuous surjective map $\Psi_{\mathbf{c}}: S \rightarrow \TT^d$ which simultaneously semi-conjugates $\rho$ to a hyperbolic automorphism $A$ of $\TT^d$ with eigenvalue $\beta$, and $\phi_t^u$ to a linear flow $L_t$ in an irrational direction (in fact, the direction of the $\beta$-eigenspace of $A$) on $\TT^d$. 
\begin{center}
    
\begin{tikzcd}
S \arrow[r, "\Psi_{\mathbf c}"] \arrow[d, "\rho", "\phi_t^u"']
& \TT^d \arrow[d, "A", "L_t"'] \\
S \arrow[r, "\Psi_{\mathbf c}"]
&  \TT^d
\end{tikzcd}
\end{center}


\item Moreover, there is a bijection between the set of \emph{affine maps} $\Aff(\TT^d, \R/\Z)$ and the set of eigenfunctions. Here, affine maps are maps $f:\TT^d \to \R/\Z$ of the form
$$f(x) = a \cdot x + b' \pmod{\Z},$$
where $a \in \Z^d$ and $b' \in \R/\Z$, and the bijection between $\Aff(\TT^d, \R/\Z)$ and the set of eigenfunctions is given by $$f \longmapsto f \circ \Psi_{\mathbf c}.$$
\end{enumerate}
\end{Theorem}

\paragraph*{\bf Measurable conjugacy and the Pisot conjecture} A natural question is whether the map $\Psi_c$ is a conjugacy, that is, is it almost surely 1-to-1.  We conjecture that these maps are in fact conjugacies, a kind of \emph{Geometric Pisot Conjecture}. We recall that the \emph{Pisot conjecture} says that if a substitution (see \S\ref{sec:substitutions} for definitions) is irreducible and Pisot, then its subshift has pure discrete spectrum, and is thus conjugated to a translation on an abelian compact group. We refer to Akiyama-Barge-Berthe-Lee-Siegel~\cite{ABBLS} for an overview on the subject. The Pisot conjecture has attracted a fair amount of attention. In fact, Pisot substitution systems and the Pisot conjecture have numerous applications, for example to Diophantine approximation, equidistribution properties of toral translations and
low discrepancy sequences, beta-shifts, multidimensional continued fraction expansions, generation or recognition of arithmetic discrete planes, and the effective construction of Markov partitions for toral automorphisms whose largest eigenvalue is a Pisot number. The Pisot conjecture is supported by numerical evidence since it can be reformulated in effective terms, but it has only been proved in the case of two symbols~\cite{HS}.


\paragraph*{\bf A semi-algorithm} In Appendix \ref{sect: example semiconjugacy}, we describe a procedure to analyze this question in any particular case of the flow in the unstable direction for a pseudo-Anosov map with a Pisot expansion factor, using a semi-algorithm of Mercat~\cite{Mercat:semialgorithm}, based on ideas of Rauzy~\cite{Rauzy}. We give a concrete example of a pseudo-Anosov map $\rho_7$ (with a degree $3$ Pisot expansion factor) of the genus $3$ surface $S_7$ arising from the double regular heptagon for which the map $\Psi_c$ defined above is indeed a conjugacy. Choosing a favorite eigenfunction, we explain that the return map of straight-line flow is conjugate to an interval exchange. We show this interval exchange is conjugate to a translation of $\TT^2$, and draw an approximation of this measurable conjugacy, see Figure~\ref{fig:torus}. By work of Clark-Sadun~\cite{ClarkSadun}, see remark in the paragraph below, the interval exchange has pure discrete spectrum if and only if the associated flow does.

\paragraph*{\bf A symbolic perspective} Solomyak~\cite{Solomyak97ETDS}*{page 96} studied \emph{self-similar substitution flows} and showed that they are weak mixing if and only if $\beta$ is not Pisot (see also the Bufetov-Solomyak survey~\cite{BufetovSolomyakSurvey}*{Theorem 5.3}), and that all eigenfunctions are continuous, following ideas of Host~\cite{Host}. This implies the first part of Theorem~\ref{theorem:pisotsemiconjugacy}, although the topology used by Solomyak is the Cantor topology. We will show directly in \S\ref{sect:Pisot} that the eigenfunctions we construct in our setting are continuous. Barge-Kwapisz~\cite{BargeKwapisz} study primitive unimodular Pisot substitutions, and associated tiling flows, and provide a geometric model space, the \emph{strand space} of the substitution, and their work also suggests a possible semi-algorithm for checking for discrete spectrum. It seems likely that the flow on the strand space should be related to a flow on a (perhaps infinite) translation surface, via the construction of Lindsey-Trevino~\cite{LindseyTrevino}. We also note there is a combinatorial semi-algorithm for checking discrete spectrum for Pisot substitutions based on ideas of Livshits~\cites{Livshits1, Livshits2}, explained by Sirvent-Solomyak~\cite{SirventSolomyak}. We also note that Clark-Sadun~\cite{ClarkSadun} showed that in the Pisot self-similar substitution flow setting, the flow has pure discrete spectrum if and only if the shift map does.

\paragraph*{\bf Degree $2$ expansion factors and square-tiled surfaces}  If $\beta>1$ is the expansion factor of a pseudo-Anosov map on a translation surface $S$, and $[\Q(\beta):\Q] = 2$, Franks-Rykken~\cite{FranksRykken} showed that $S$ must be square-tiled, that is, a branched cover of the flat torus branched at $0$, and so the flow is not weak-mixing (each coordinate of the covering map to the torus yields an eigenfunction). In this case $\beta$ is Pisot, since $\beta^{-1}$ is its only conjugate.

\paragraph*{\bf Non-minimality and holonomy fields} We will state a corollary to Theorem~\ref{theorem:pisotsemiconjugacy} yielding explicit non-minimal prism flow directions. Let $K \subset \R$ be a subfield. We say a translation surface $S$ is defined over $K$ if $S$ has a decomposition into polygons (obtained by cutting along saddle connections) whose edge vectors lie in $K^2$. The {\em holonomy field} is the smallest subfield $K_S \subset \R$ such that there is an $A \in \GL(2,\R)$ such that $A \cdot S$ is defined over $K_S$. Given a pseudo-Ansosov map $\rho:S \to S$, it is known that $K_S$ equals $\mathbb Q$ adjoin the trace of $D \rho$ \cite{KenyonSmillie}*{Theorem 28}, so this field is also called the {\em trace field}. We have $\mathrm{tr}(D \rho)=\beta + \beta^{-1}$. When $\beta$ is a Pisot eigenvalue of $D \rho$ with degree at least three, $\beta^{-1}$ cannot be an algebraic conjugate of $\beta$, and so $\beta \in K_S$ and it follows that $K_S=\Q(\beta).$

\begin{Cor}\label{cor:pisot} Let $S$ be as in Theorem~\ref{theorem:pisotsemiconjugacy} and use the same hypotheses and notation here. Suppose in addition that $S$ is defined over its own holonomy field, $K_S = \Q(\beta)$. Let $M=S \times (\R/\Z)$ denote the associated translation prism, and consider the prism flow in a direction $\mathbf v \in \R^3$ whose projection to $S$ lies in the unstable direction of $\rho$. This prism flow is non-minimal if and only if $\mathbf v$ has a real scalar multiple in $\Q(\beta)^3$.
\end{Cor}

\paragraph*{\bf Holonomy fields} We remark that there is a more technical result that holds when $S$ is not defined over its holonomy field which comes from an understanding of how eigenvalues change under affine maps;
see Proposition \ref{prop:affine change}.

\paragraph*{\bf Eigenfunctions and non-minimality} The link between (continuous) eigenfunctions for the base flow and (continuous) invariant functions for the prism flow is explained in more detail \S\ref{sec:furstenberg:flow}. For now, we note that if we have a continuous eigenfunction $f:S \rightarrow \R/\Z$ with eigenvalue $-\lambda$, that is, $$f(\phi_t x) = e^{2\pi i(-\lambda )t} f(x) \mbox{ for all } t \in \R$$ then the function $F(x, r) = e^{2\pi i r} f(x)$ is a continuous invariant function for the prism flow with vertical speed $\lambda$.

\paragraph*{\bf The heptagonal prism} Part (2) of Theorem~\ref{theorem:prisms}  follows from Corollary~\ref{cor:pisot} applied to the translation surface $S_7$ associated to the double heptagon, which arises from unfolding billiards in the $(\pi/7, 5\pi/14, \pi/2)$-right triangle. In \S\ref{sec:genus3}, we describe a pseudo-Anosov transformation $\rho_7$ of this surface with a degree $3$ Pisot expansion factor, and explicitly compute the group $E(\rho_7)$ of eigenvalues, which yield non-minimal directions for the prism flow on $M_7 = S_7 \times S^1$, and thus for the associated billiard flow. As we mentioned above, we explore this example even further in Appendix~\ref{sect: example semiconjugacy}, where we show that for $\rho_7$ the associated map $\Psi_c$ defined in Theorem~\ref{theorem:pisotsemiconjugacy} is in fact a conjugacy.

\paragraph*{\bf Other regular prisms} In Appendix~\ref{appendix:ngons}, we use Corollary~\ref{cor:pisot} to find explicitly non weak-mixing directions and eigenvalues for the straight-line flow on the translation surfaces $S_n$ associated to regular polygons. We consider examples when $n=7, 9, 14, 18, 20, 24$ drawn from the work of Arnoux-Schmidt~\cite{ArnouxSchmidt}*{Theorem 2}, who were searching for directions with vanishing SAF invariant. We also give examples when $n=16$ and $n=30$, which were found by our own search. Pictures of eigenfunctions for each of these $n$ are provided in Figure \ref{fig:heptagon eigenfunction} and in the ancillary files on the arXiv.

\paragraph*{\bf Arnoux-Yoccoz example} Theorem~\ref{theorem:pisotsemiconjugacy} can be viewed as a generalization of a result of Arnoux~\cite{ArnouxBSMF88}, who introduced a genus $3$ example and built a semi-conjugacy (later shown to be a conjugacy by Arnoux-Cassaigne-Ferenczi-Hubert~\cite{ArnouxCassaigneFerencziHubert}) for a pseudo-Anosov map, associated linear flow, and interval exchange transformation related to the Tribonacci substitution and the Rauzy fractal. Here, additionally, the interval exchange map given by first return map of the unstable flow on the surface to a stable segment is conjugated to a toral translation, the first return map of the linear flow on the torus to a codimension 1 sub-torus. In fact this example is part of a larger family of examples growing in genus known now as the Arnoux-Yoccoz examples. This example does not arise from a classical billiard, and is not a lattice surface.

\paragraph*{\bf Do-Schmidt examples and beyond} There is a large family of examples of genus $3$, found by Do-Schmidt~\cite{DoSchmidt}*{Theorem 2}. Motivated by the study of vanishing \emph{Sah-Arnoux-Fathi (SAF) invariant} (see \S\ref{sec:SAF}), they showed that for each integer $k \geq 2$, there exists at least four orientable pseudo-Anosov
maps in the hyperelliptic component of the stratum $\hh(2,2)$ with expansion factor $\beta_k$ whose minimal polynomial is given by $$P_k(x) = x^3
- (2k+ 4)x^2 + (k+ 4)x -1.$$ It is easy to verify directly that the largest root $\beta_k$ of this polynomial is Pisot (the other eigenvalues of this polynomial are in fact real and have absolute value less than $1$). Other interesting explicit examples of pseudo-Anosovs with vanishing SAF invariant can be found in McMullen~\cite{McMullen:cascade}*{\S5} (who explains an example of Lanneau) and a larger class of examples in Winsor~\cite{Winsor}. We will discuss the Do-Schmidt series of examples, and the connections between our constructions and the SAF invariant in \S\ref{sec:genus3}.

\subsection{Minimality and non-unique ergodicity}\label{sec:minimal} Recall that a dynamical system is \emph{minimal} if every (well-defined) orbit is dense, and uniquely ergodic it there is a unique invariant probability measure. Corollary~\ref{cor:pisot} gives examples of non-minimal directions for prism flows, since the associated eigenfunctions for the translation flow are continuous. Following a construction of Avila-Delecroix~\cite{AvilaDelecroix}*{Theorem 32} (which uses a criterion of Bressaud-Durand-Maass~\cite{BDM}*{Theorem 1}, applied as~\cite{AvilaDelecroix}*{Theorem 33}) of a large set of directions admitting measurable but non-continuous eigenfunctions for linear flows on lattice surfaces, we can use our dictionary to deduce the existence of (a relatively large) set of minimal, non-uniquely ergodic directions on lattice prisms- we note these do not correspond to expanding directions for pseudo-Anosov maps on the associated lattice surface. We say a lattice surface $S = (X, \om)$ is \emph{Salem} if its Veech group is non-arithmetic and contains a pseudo-Anosov map with a Salem expansion factor, where we recall that an algebraic integer is \emph{Salem} if it has absolute value greater than 1, and all of its algebraic conjugates have modulus at most $1$, and there is at least one Galois conjugate of modulus $1$. We note that for this result we need both the lattice assumption, and the existence of a Salem pseudo-Anosov, so the existence of a Salem pseudo-Anosov is not sufficient. To summarize:

\begin{Cor}\label{cor:minimalnue} Let $S$ be a Salem lattice surface and $M$ be the associated translation prism. There exist \begin{itemize}
    \item A positive Hausdorff dimension set of angles $\theta$, such that there are non-zero $s$ so that the prism flow in direction $(\theta, s)$ on $M$ is minimal and non-uniquely ergodic.
    \item A positive Hausdorff dimension set of angles $\theta$, such that there are non-zero $s$ so that the prism flow in direction $(\theta, s)$ on $M$ is non-minimal.
\end{itemize}
\end{Cor}

\paragraph*{\bf Minimal non-uniquely ergodic directions on regular prisms} Avila-Delecroix~\cite{AvilaDelecroix}*{Proposition 35} show that the Veech groups for the surfaces associated to billiards regular polygons $P_n$ with $n \le 15$, $n$ odd, and $n \le 30$, $n$ even, contain Salem elements, so Corollary~\ref{cor:minimalnue} applies in these cases. A natural question would be to identify for which regular polygons $P_n$ does the associated translation surface admit a pseudo-Anosov map with a \emph{Pisot} expansion factor.

\paragraph*{\bf Organization} We review the basics of linear flows on translation surfaces and interval exchange transformations in \S\ref{sec:background}, including a criterion of Furstenberg~\cite{Furstenberg} for unique ergodicity of skew-products in \S\ref{sec:furstenberg}. We use this criterion to derive our unique (and non-unique) ergodicity results (Theorem~\ref{theorem:ue}, Corollaries~\ref{corollary:ue:generic}, \ref{corollary:ue:veech}, \ref{corollary:ue:AHChaikaForni}, and Corollary~\ref{cor:minimalnue}) from weak-mixing results in \S\ref{sec:weakmixing}. In \S\ref{sect:Pisot} we show how to generate eigenvalues and eigenfunctions for the unstable flow associated to pseudo-Anosov maps with Pisot expansion factor, proving the first two parts of Theorem~\ref{theorem:pisotsemiconjugacy}, and in \S\ref{sec:tori} we show how to simultaneously semi-conjugate these pseudo-Anosov maps and straight line flows to hyperbolic toral automorphisms and linear flows, proving the remaining two parts of Theorem~\ref{theorem:pisotsemiconjugacy}. We discuss the relationship between the degree of expansion factors and genus and then discuss the double heptagon, Arnoux's genus $3$ example, and the Do-Schmidt examples in \S\ref{sec:genus3}, including constructing our concrete example of a non-uniquely ergodic prism flow in the right regular heptagonal prism in \S\ref{sec:nonergodicdirections}. We also use a combinatorial construction of Mercat's~\cite{Mercat} to show that for the double heptagon, the natural associated interval exchange map is weak mixing. In Appendix~\ref{appendix:ngons}, following Arnoux-Schmidt~\cite{ArnouxSchmidt} we give further examples of regular $N$-gons where the translation surface associated to the billiard flow has pseudo-Anosov maps with Pisot expansion factors, and thus, non weak-mixing directions. In Appendix~\ref{sect:computing eigenfunctions}, we explain how we plotted the eigenfunctions. In Appendix~\ref{sect: example semiconjugacy}, we work out some details of the semi-conjugacy that arises from Theorem~\ref{theorem:pisotsemiconjugacy} in the context of the regular heptagon: the return map to a level set of an eigenfunction is semi-conjugate to a toral rotation. 

\paragraph*{\bf Acknowledgments} This work was started as part of research in residence at Centre International de Recontres Mathématiques (CIRM) in Luminy, as part of the Chaire Jean Morlet program in Autumn 2023. The Chaire Jean Morlet program was supported by the National Science Foundation, CNRS, Aix-Marseille Université, Ville de Marseille, the Clay Mathematics Institute, and the Compositio Foundation. We thank the wonderful staff at CIRM for their hospitality and the ideal working conditions. JSA was partially supported by the United States National Science Foundation (DMS 2404705), the Pacific Institute for the Mathematical Sciences, and the Victor Klee Faculty Fellowship at the University of Washington; JSA also thanks Aix-Marseille Université and the City University of New York (CUNY) for their hospitality in Autumn 2024. WPH was partially supported by the Simons Foundation and by a PSC-CUNY Award, jointly funded by The Professional Staff Congress and CUNY; NB and PH thank ANR IZES. Thanks to Pierre Arnoux, Julien Cassaigne, Stefano Marmi, Paul Mercat and Maurice Reichert for useful discussions.

\section{Translation flows, interval exchange maps, and Furstenberg's criterion}\label{sec:background} We first (\S\ref{sec:translation}) recall some background on translation surfaces, their moduli spaces, associated flows (\S\ref{sec:linearflows}), interval exchange maps (\S\ref{sec:iet}) and renormalization schemes (\S\ref{sec:substitutions}). These sections will be very brief and mainly serve as a pointer to appropriate references. Then (\S\ref{sec:furstenberg}) we discuss a beautiful criterion of Furstenberg~\cite{Furstenberg} for unique ergodicity of skew product transformations where the skewing factor is a circle, and show how to modify it for flows (\S\ref{sec:furstenberg:flow}), and then show how to use it to prove  our unique (and non-unique) ergodicity results (Theorem~\ref{theorem:ue}, Corollaries~\ref{corollary:ue:generic}, \ref{corollary:ue:veech}, \ref{corollary:ue:AHChaikaForni}, and Corollary~\ref{cor:minimalnue}) from weak-mixing results in \S\ref{sec:weakmixing}.

\subsection{Translation surfaces}\label{sec:translation} A comprehensive reference for the basic definitions of translation surfaces and their moduli spaces is, for example,~\cite{AthreyaMasur}. A (compact) \emph{translation surface} $S$ is a pair $(X,\omega)$ of a compact genus $g$ Riemann surface $X$ and a holomorphic one-form $\omega$. Integrating $\om$ away from its finite zero set $\Sigma$ gives an atlas of charts to $\mathbb C$ whose transition maps are translations, justifying the name translation surface. Decomposing the surface into disks along a basis for the relative homology $H_1(S, \Sigma; \Z)$ and integrating $\om$ along this basis gives a presentation of the surface as a polygon in $\C$ with parallel sides identified by translations, and indeed, given such a polygonal presentation pulling back $dz$ from $\C$ recovers the original one-form $\om$. The orders of the zeros of $\omega$ must add up to $2g-2$, and a zero of order $k$ corresponds to a cone point of angle $2\pi(k+1)$ for the flat metric it determines. Two translation surfaces $S_1 = (X_1, \om_1)$ and $S_2 = (X_2, \om_2)$ are equivalent if there is a biholomorphism $f:X_1 \rightarrow X_2$ between the underlying Riemann surfaces $X_1$ and $X_2$ with $f_* \om_2 = \om_1$. 

\paragraph*{\bf Strata} The moduli space of genus $g$ translation surfaces is denoted by $\Omega_g$, and this space is stratified into spaces $\Omega_g(\alpha)$ by integer partitions $\alpha = (\alpha_1, \ldots, \alpha_n)$ of $2g-2$, where the $\alpha_i$ are the orders of the zeros of one-forms $\omega \in \Omega_g(\alpha)$. Kontsevich-Zorich~\cite{KZ} showed that these strata have at most $3$ connected components. The group $GL^+(2,\R)$ of invertible orientation preserving two-by-two $\R$-linear maps acts on $\Omega_g(\alpha)$ by $\R$-linear postcomposition with translation charts. We will work with the \emph{area 1} loci $\hh(\alpha) \subset \Omega_g(\alpha)$, the set of surfaces which have area $1$, where $\mbox{Area}(\om) = \frac i 2 \int_X \omega \wedge \overline{\om}$ is the area of the translation surface. These loci are preserved by the action of the subgroup of determinant $1$ matrices $SL(2,\R)$, and carry an ergodic absolutely continuous invariant probability measure known as Masur-Smillie-Veech (MSV) measure, which we denote $\mu_{\mathcal H}$. As discussed above, associated to a rational billiard table $P$, there is a associated translation surface $S_P$. We note that the set of surfaces arising from billiards is of measure $0$ in each stratum. The dynamics of the $SL(2,\R)$-action serves as a renormalization dynamics for natural dynamics on the surface, as we explain further below.

\paragraph*{\bf Lattice surfaces} Given a translation surface $S = (X, \om) \in \hh(\alpha)$, we denote its stabilizer under the $SL(2,\R)$ action by $SL(S)$. Generically, a surface $S \in \hh(\alpha)$ has \emph{trivial} stabilizer, but for a dense set of surfaces, the stabilizer, known as the \emph{Veech group}, is a lattice, and these surfaces are known as \emph{Veech surfaces}, or \emph{lattice surfaces}. Surveys on lattice surfaces include Hubert-Schmidt~\cite{HubertSchmidt} and a more recent survey of McMullen~\cite{McMullen:survey}.

\subsubsection{Linear flows}\label{sec:linearflows} Associated to each translation surface, we have a $S^1$-indexed family of directional flows $\phi^t_{\theta}$, which gives the unit-speed flow in direction $\theta$ on the surface, which preserve the natural Lebesgue measure on the surface.  Motivated by billiards, the ergodic properties of these flows have been studied from many different perspectives, including renormalization dynamics and symbolic dynamics. They are a prime example of what has come to be known as \emph{parabolic dynamics}. These are generalizations of linear flows on the torus, which, in the absence of any singularities, have elliptic dynamics (and are never weak mixing). 

\paragraph*{\bf Renormalization} Using renormalization ideas, and in particular the action of the positive diagonal subgroup of $SL(2,\R)$, known as the \emph{Teichm\"uller geodesic flow}, Masur~\cite{Masur} and Veech~\cite{Veech:ergodic} independently showed that the \emph{vertical} flow on almost every translation surfaces is uniquely ergodic, and Kerckhoff-Masur-Smillie~\cite{KMS} proved that for every surface, and almost every direction, the flow is uniquely ergodic. For lattice surfaces, Veech~\cite{Veech:dichotomy} showed that the dynamics at the level of unique ergodicity is similar to that of the torus: in every direction, the flow is either periodic or uniquely ergodic. We will be particularly interested in some low genus lattice surfaces, which arise from billiard constructions, where we can produce explicit examples of non-weak mixing directional flows.

\paragraph*{\bf Affine diffeomorphisms} Elements of $SL(S)$ correspond to (derivatives of) \emph{affine automorphisms} of the translation surface $S$. We will be particularly interested in \emph{pseudo-Anosov} maps $\rho$, which correspond to hyperbolic elements of $SL(S)$. Associated to such a $\rho$, there are a pair of eigendirections for the action of the matrix on $\R^2$, one expanding and one contracting, which induce a pair (expanding and contracting) of invariant transverse measured foliations on $S$. 

\subsubsection{Interval Exchange Transformations}\label{sec:iet} Choosing a transverse segment to a linear flow on a translation surface, the first return map is an \emph{interval exchange transformation}, an orientation-preserving pieceweise isometry of the interval with finitely many discontinuities. Excellent surveys about interval exchange transformations include those of Viana~\cite{Viana}, Yoccoz~\cite{YoccozSurvey} and Zorich~\cite{Zorich:survey}. See Figure~\ref{fig-iet} for an illustrative example of a translation surface, a transverse interval, and associated interval exchange map.

\subsubsection{Rauzy Induction and pseudo-Anosov maps}\label{sec:rauzyPA} There is a discrete analog of the Teichm\"uller geodesic flow which serves as a renormalization dynamics for interval exchange transformations, known as \emph{Rauzy induction}. This involves carefully choosing a subinterval of the original domain of definition, and considering the first return map to this interval, which, when chosen properly, will have the same number of intervals of continuity as the original transformation. When the IET arises as the first return map for the linear flow in the expanding direction of a pseudo-Anosov diffeomorphism, Rauzy induction shows how to view the interval exchange transformation as a \emph{substitution system}, a particularly well-studied family of symbolic dynamical systems.

\subsubsection{Zippered rectangles}\label{sec:zippered} An important construction related to interval exchange transformations, originally studied by Veech~\cite{Veech:ergodic} and further developed by Zorich~\cite{Zorich:gauss} is the notion of \emph{zippered rectangles}. Given a translation surface $S= (X,\om)$ and a (minimal) directional flow $\phi_t$, and a transverse interval $I \subset S$, we can express the flow $\phi_t$ as a suspension flow over the interval exchange map $T: I \rightarrow I$ which is the first return map of $\phi_t$ to $I$. The return time function $h: I \rightarrow \R^+$ is in fact \emph{constant} on intervals $I_j$ of continuity of $T$, and this allows us to decompose the surface $S$ into a union of rectangles, with bases $I_j$ and heights $h_j$, together with gluing data. In the particular situation when the flow is in the unstable direction of a pseudo-Anosov map $\rho$, the zippered rectangle decomposition in fact corresponds to a \emph{Markov partition} for $\rho$ (see, for example, Fathi-Laudenbach-Po\'enaru~\cite{FLP}*{Expos\'e 10}). We refer to Section \ref{sec:nonergodicdirections} for an explicit example.

\subsection{Substitution systems}\label{sec:substitutions} 
For an introduction to subshifts and substitutions, we refer to \cite{Durand-Perrin-22}, \cite{Foggbook} and \cite{Queffelec}. We recall some important definitions and constructions.

\paragraph*{\bf Alphabets and languages} Let $\mathcal{A}$ be a finite set called the alphabet with cardinality $d\ge 2$. Elements of $\mathcal{A}$ are called \emph{letters} or \emph{digits}.
A \emph{word} is a finite  or infinite string of digits. If $u=u_{0}\ldots u_{n-1}$ is a word, a prefix of $u$ is any word $u_{0}\ldots u_{j}$ with $j\le n-1$. A suffix of $u$ is any word of the form $u_{j}\ldots u_{n-1}$ with $0\le j\le n-1$.  If $v$ is the  finite word $v=v_0\dots v_{n-1}$ then $n$ is called the length of the word $v$ and is denoted by $|v|$. The set of all finite words over $\mathcal{A}$ is denoted by $\mathcal{A}^*$. 


\paragraph*{\bf The shift map} The shift map is the map defined on $\mathcal{A}^\mathbb{N}$ by $\sigma(u)=v$ with $v_n=u_{n+1}$ for every integer $n$. We endow $\mathcal A$ with the discrete topology and consider the product topology on $\mathcal A^\N$. This topology is compatible with the distance $d$ on $\mathcal A^{\mathbb N}$ defined by
$$d(x,y)=\frac{1}{2^n}\quad \text{if}\quad n=\min\{i\geq 0,  x_i\neq y_i\}.$$ A {\bf subshift} is a closed subset of $\mathcal{A}^{\mathbb N}$ which is invariant by the shift map.

\paragraph*{\bf Substitutions} A \emph{substitution} $\sigma$ is a map from an alphabet $\mathcal{A}$ to the set $\mathcal{A}^*\setminus\{\epsilon\}$ of nonempty finite words on $\mathcal{A}$. It extends to a morphism of $\mathcal{A}^*$ by concatenation, that is $\sigma(uv)=\sigma(u)\sigma(v)$. The matrix associated to a substitution is an element $M$ of $M_d(\mathbb N)$ such that $M_ {i,j}$ is equal to the number of ocurences of the letter $i$ in the word $\sigma(j)$. The susbtitution is said to be \emph{primitive} if there exists a positive integer $k$ such that $M^k$ is a strictly positive matrix.

\paragraph*{\bf Subshifts} Let $\sigma$ be a substitution over the alphabet $\mathcal{A}$, the subshift associated to $\sigma$ is a subset $X_\sigma$ of $\mathcal A^{\mathbb N}$ such that $x\in X_\sigma$ if  and only if for every positive integers $i, j$ the word $x_i\dots x_{j+i}$ appears in some $\sigma^n(a)$ for a letter $a$. It is called the {\it subshift} associated to the substitution. By a classical result \cite{Foggbook}, if the substitution is primitive, then the subshift is uniquely ergodic and minimal.


\paragraph*{\bf Suspensions} If $H$ is an element of $\mathbb R_+^d$, the suspension flow of the subshift $X$ by $H$ is the vertical flow $\phi_t:Y \rightarrow Y$ on the space $$Y = \bigsqcup_{a \in \mathcal A} (X_a \times [0, H_a])/\sim.$$ Here $(x, H_{a}) \sim (\sigma(x), 0),$ if $x\in X_a$ and $X_a$ is the cylinder associated to the letter $a$. The flow is given by $\phi_t(x,s)=(x,s+t)$.

\paragraph*{\bf Subshifts and interval exchange maps}
Given an interval exchange map, we can define an associated subshift, where the alphabet is given by the intervals, and by coding the orbit of a point under the interval exchange by the sequence of intervals visited.  The collection of all such infinite words defines a subshift. This class of subshifts has been extensively studied, see \cite{F} and a characterization of such subshifts exists~\cite{FZ}. In some cases (for example, if the orbit of the interval exchange transformation is periodic under Rauzy induction), these subshifts can be associated to a substitution, see for example \cite{BC} where it is proven that if the lengths of the exchanged intervals all belong to a quadratic number field then the subshift is substitutive, or \cite{J}.

\subsection{Furstenberg's criterion}\label{sec:furstenberg} 

\subsubsection{Skew-products}\label{sec:furstenberg:skew} We review a beautiful result of Furstenberg~\cite{Furstenberg}*{Lemma 2.1} (inspired, according to discussions with Bryna Kra, by work of Auslander on nilflows, see, for example~\cite{AGH}) connecting the ergodic properties of skew-products with a circle factor to the weak-mixing properties of the base transformation. Let $T_0: X_0 \rightarrow X_0$ be a continuous uniquely ergodic transformation of a compact metric space $\Om_0$, with $\mu_0$ the unique $T_0$-invariant measure on $X_0$, and let $g: X_0 \rightarrow S^1 = \{z \in \C: |z| =1\}$ be continuous. We define the \emph{skew-product} $T: X \rightarrow X$, where $X = X_0 \times S^1$ by \begin{equation}\label{eq:skewproduct} T(x, z) = (T_0 x, g(x) z). \end{equation} By definition, $T$ preserves the measure $\mu = \mu_0 \times m$, where $m$ is the Lebesgue probability measure on $S^1$. Furstenberg's criterion states that $T$ is \emph{not uniquely ergodic} if and only if there is a $0 \neq n \in \Z$ and a measurable function $R: X_0 \rightarrow S^1$ such that \begin{equation}\label{eq:furstenberg} \frac{R(T_0 x)}{R(x)} = g^n(x). \end{equation}

\paragraph*{\bf Fourier expansions} We sketch Furstenberg's argument from~\cite{Furstenberg}*{Proof of Lemma 2.1}. Observe first that the unique ergodicity of $T$ is equivalent to the \emph{ergodicity} of the measure $\mu$. If $T$ is not ergodic, then there is a non-constant $T$-invariant function $F: \Om \rightarrow \C$, $F \in L^2(\Om, \mu)$. Then for $\om \in \Om$, we can consider the Fourier expansion of the function $F(\om, z)$ in the z variable, where \begin{equation}\label{eq:fourier} F(\om, z) = \sum_{n \in \Z} c_n(\om) z^n.\end{equation} Since $F$ is non-constant, there is an $n_0 \in \Z\minuszero$ such that $c_n(\om)$ is not identically zero. By $T$-invariance, we have $$\sum_{n \in \Z} c_n(T_0 \om) g^n(\om) z^n = F(T(\om, z)) = F(\om, z) = \sum c_n(\om) z^n.$$ By the uniqueness of Fourier coefficients, and since $c_{n_0} (\om) \neq 0$, we have $$c_{n_0}(\om) = c_{n_0}(T_0\om) g^{n_0}(\om).$$ So putting $R(\om) = \frac{1}{c_{n_0}(\om)}$ is a solution to \eqref{eq:furstenberg}, with $n=n_0$. Conversely, if we have a solution $R$ to \eqref{eq:furstenberg} with $n=n_0$, the function $$F(\om, z) = \frac{z^{n_0}}{R(\om)}$$ is a $T$-invariant non-constant $L^2$-function on $\Om$. Note that if $n_0 = 1$ we see that for any $c \in \C$, the level sets $$F^{-1}(c) = \{(\om, z) \in \Om: F(\om, z) = c\}  = \{(\om, z): z = cR(\om)\}$$ are copies of the graphs of the functions $cR$ (note that if $R$ is a solution to \eqref{eq:furstenberg}, so is $cR$ for any $c \in \C$). For each $c$, the pushforward (under the map $cR$) of $\mu_0$ on $F^{-1}(c)$ is the unique ergodic invariant probability measure  on $F^{-1}(c)$. In particular, there are infinitely many distinct ergodic invariant probability measures (and this conclusion holds even if $n_0 \neq 1$).

\subsubsection{Furstenberg's criterion for flows}\label{sec:furstenberg:flow} We now show how Furstenberg's criterion can be modified to give conditions for unique ergodicity (and minimality) for product flows with a circle factor. 
\begin{lemma}\label{lemma:furstenberg:flow} Let $\phi_t: (X_0, \mu_0) \rightarrow (X_0, \mu_0)$ be a flow on a probability (topological) measure space $(X_0, \mu_0)$, and for $s \in \R$, define the product flow $\Phi_t^s$ on $(X, \mu)$ with $X = X_0 \times \R/\Z$, $\mu =\mu_0 \otimes m$, where $m$ is the Lebesgue measure on $\R/\Z$ by $$\Phi_t^s(x, r) = (\phi_t(x), r+st).$$ Then \begin{enumerate}
    \item If there is a (continuous) non-trivial eigenfunction $f \in L^2(X_0, \mu_0)$ for $\phi_t$ with eigenvalue $\kappa$, there is a (continuous) non-constant invariant function $F_{n} \in L^2(X, \mu)$ for $\Phi_t^{-\frac{\kappa}{n}}$ for all $0 \neq n \in \Z$.
    \item If there is a (continuous) non-constant invariant function $F \in L^2(X, \mu)$ for $\Phi_t^s$ there exists a $0 \neq n \in \Z$ and a (continuous) eigenfunction for $\phi_t$ with eigenvalue $-ns$.
\end{enumerate}
    
\end{lemma}

\begin{proof} To prove the first assertion, suppose $f(x)$ is an eigenfunction for $\phi_t$ with eigenvalues $\kappa$, that is $$ f(\phi_t x) = e^{2\pi i \kappa t} f(x),$$ then $$F_n(x, r) = f(x)e^{2\pi i n r}$$ is $\Phi_t^{-\frac{\kappa}{n}}$ invariant. Indeed, 
\begin{align*} F_n\left(\Phi_t^{-\frac{\kappa}{n}}(x,r)\right) &= F_n\left(\phi_t x, r-\frac{\kappa}{n}\right)\\ &= f(\phi_t x) e^{2\pi i n\left(r- \frac{\kappa}{n} t\right)} \\ &= f(x) e^{2\pi i \kappa t} e^{2\pi i nr}e^{-2\pi i \kappa t} \\ &= F_n(x,r).\end{align*}
Note that by construction, if $f \in L^2(X_0, \mu_0)$, $F_n \in L^2(X, \mu)$, and if $f$ is continuous, so is $F_n$. To prove the second assertion, suppose $F(x, r) \in L^2(X,\mu)$ is $\Phi_t^s$-invariant. We have the Fourier expansion $$F(x, r) = \sum_{n \in \Z} c_n(x) e^{2\pi i n r},$$ where $$c_n(x) = \int_{\R/\Z} F(x, r) e^{-2\pi i n r} dr.$$ Note that if $F(x,r)$ is continuous, so is $c_n(x)$. By $\Phi_t$-invariance, we have $$F(\varphi_t x, r+st)= F(x, r),$$ thus $$\sum_{n \in \Z} c_n(\phi_t x) e^{2\pi i n (r+st)} = \sum_{n \in \Z} c_n(x) e^{2\pi i nr},$$ so by uniqueness of Fourier coefficients, $$c_n(\phi_t x) = e^{-2\pi i n st} c_n(x).$$ If $F$ is non-constant, there is an $0 \neq n \in \Z$ such that $c_n(x)$ is non-zero, proving the assertion.
    
\end{proof}

\subsection{Ergodicity, Weak mixing, and coboundaries}\label{sec:weakmixing} We explain how to use Furstenberg's criteria and how to relate them to a criterion of Veech linking weak mixing and the cohomological equation.

\paragraph*{\bf Applying Furstenberg's criterion} We now turn to using the Furstenberg criteria we explained in \S\ref{sec:furstenberg} to show Theorem~\ref{theorem:ue}. We give two proofs, using two different interpretations of the prism flow. For the first proof, we consider the first return map to the product of an interval in the base translation surface with the vertical direction, which we, by slight abuse of notation, call a \emph{face}. This first return map is a skew-product over an interval exchange map with a circle factor, and the skewing function is given by the heights of the associated zippered rectangles for the linear flow on the translation surface, normalized by the vertical speed, and applying the classical Furstenberg criterion and a result of Veech~\cite{Veech:metrictheory}*{Theorem 6.3} gives us our result. For our second proof, we can use directly the Furstenberg criterion modified for flows, and the expression of our prism flow as a product flow over a linear flow on a translation surface.

\paragraph*{\bf Coboundaries and zippered rectangles} A crucial ingredient in our results is Veech's~\cite{Veech:metrictheory} observation linking the existence of eigenfunctions for translation flows to the cohomological equation for interval exchange maps which arises from the \emph{zippered rectangles}~\cite{Veech:ergodic} construction. Let $\phi_t$ be the vertical flow  on a translation surface $S$, $I \subset S$ be a transverse segment, $T_0$ the first return map to $I$, and $h$ the associated return time function, \emph{constant} on intervals of continuity of $T$. We write $I = \bigsqcup_{j=1}^k I_j$, where the $I_j$ are these intervals of continuity, and let $h_j \in \R^+$ denote the value of $h$ on $I_j$. If $\psi_c: S \rightarrow \R/\Z$ is an (additive) eigenfunction for the flow with eigenvalue $c$, we have $$\psi_c(\phi_t p) - \psi_c(p) = ct \mod \Z,$$ for all $p \in S$ with well-defined $\phi_t$-orbit. We remark that we will primarily work additively in the sequel, and use $\R/\Z$ as the target space for our eigenfunctions. In particular, starting with $x \in I_j$, putting $t=h_j$, and letting $f_c = \psi_c|_{I}$, we have $$f_c(T_0x) - f_c(x) = c h_j \mod \Z.$$ That is, an eigenfunction for the flow $\phi_t$ yields a \emph{coboundary} for the interval exchange map $T_0$.

\paragraph*{\bf Time changes} The IET $T_0$ is weak-mixing if no constant function is a coboundary, and has an eigenvalue $c$ if the constant function $c$ is a coboundary for the IET. The flow $\phi_t$ has an eigenvalue $c$ if the constant function $c$ is a coboundary for the flow, which means that the function $ch$ is a coboundary for the IET. We see that the weak-mixing properties of the flow and the IET, while related, are by no means equivalent, and indeed, even different height functions will lead to potentially different weak-mixing properties. This is part of a general principle, well-known to ergodic theorists, but crucial for us, so that we emphasize it: {\em Weak-mixing is sensitive to time-changes}. We refer to Section \ref{sec:nonergodicdirections} for an explicit example.

\paragraph*{\bf Skew-products, first returns, and suspensions} With notation still as above, let $M = S \times (\R/\Z)$ the associated translation prism. We define the \emph{face} $F = I \times (\R/\Z)$ over $I$, and then the first return map $T$ of the flow $\Phi_t^s$ to $F$ is a skew-product over $T_0$, with piecewise constant skewing function (the constants are scalar multiples of the height function $h$ which governs the returns of $\phi_t$ to $I$), given by \begin{equation}\label{eq:iet:skew}T(x, \theta) = (T_0 x, \theta + s h(x)).\end{equation} Rewriting \eqref{eq:iet:skew} by identifying $\R/\Z$ with $S^1$ in multiplicative notation, we have $$T(x, z) = (T_0 z, e^{2\pi is h(x)} z).$$ Applying Furstenberg's criterion with $g(x) = e^{2\pi is h(x)}$, we have that this map is uniquely ergodic if and only if there are no non-trivial solutions to \eqref{eq:furstenberg}, that is, if for all $n \in \Z\minuszero$ there is no measurable $R: I \rightarrow S^1$ satisfying \begin{equation}\label{eq:veechcriterion}R(T_0 x) = e^{2\pi n s h(x)} R(x).\end{equation} By an observation of Veech~\cite{Veech:metrictheory}, the existence of a solution $R$ to \eqref{eq:veechcriterion} is equivalent to $n s$ being an eigenvalue for the flow $\phi_t$, that is $n s \in \mathcal S(\om)$. Thus, the flow $\Phi^{s}_t$ is uniquely ergodic if and only if $ns \notin \mathcal S(\om)$ for all $n \in \Z\minuszero$, as desired, proving Theorem~\ref{theorem:ue}.


\paragraph*{\bf Floor returns} With notation as above, we can directly apply the Furstenberg criterion for product flows as explained in \S\ref{sec:furstenberg:flow} to conclude that the flow $\Phi_t^{s}$ is uniquely ergodic if and only if $ns \notin S(\om)$ for all $n \in \Z\minuszero$, as desired. We note that this also says that the first return map of the flow $\Phi_t^{s}$ to the floor $S \times \{0\} \subset M$, that is, the time $s$ map for the flow $\phi_t$ is uniquely ergodic if and only if $ns \notin S(\om)$ for all $n \in \Z\minuszero$, once again proving Theorem~\ref{theorem:ue}.

\paragraph*{\bf Relationships} Before moving on to proving our corollaries, we recall the relationships between the various flows and maps we have introduced, which induce relationships between ergodicity, weak-mixing, and coboundaries of these systems. The face return map $T$ is a first return map for the prism flow $\Phi_t^s$ and a skew-product over the IET $T_0$, which is the first return map for the translation flow $\phi_t$ to a transverse interval. The return time (or roof) function $h$, for $T_0$ with respect to $\phi_t$ is piecewise constant, and the associated return time function for the face return with respect to $\Phi_t^s$ is the scaled function $sh$, and is constant on vertical fibers. The prism flow $\Phi_t^s$ is a  skew-product flow over $\phi_t$ with constant skewing function $s$, and a suspension flow over the floor return map $\phi_s$ with constant roof function $1$, and this map is also the time $s$-map of $\phi_t$.

\subsubsection{Spectra of translation flows}\label{sec:allsurfaces} As discussed in \S\ref{sec:prisms}, $\mathcal S(\om)$ is of Lebesgue measure $0$ for all $\omega$, and therefore, so is $\genspec(\om)$. Then since for all $s \notin \genspec(\om)$, we have unique ergodicity, we have the first conclusions in Theorem~\ref{theorem:ue}.

\paragraph*{\bf Generic surfaces} Applying the main result of Avila-Forni~\cite{AvilaForni}, which states that for almost all $S= (X, \om)$, $\mathcal S(\om) = \emptyset$ and in fact, for almost all $\om$ and almost all $\theta$, $\mathcal S_{\theta}(\om) = \emptyset$, we have that the associated prism flows are in fact uniquely ergodic for \emph{every} choice of vertical speed $s$, giving us the remainder of Theorem~\ref{theorem:ue}.

\paragraph*{\bf Lattice prisms} Applying the main result of Avila-Delecroix~\cite{AvilaDelecroix}*{Theorem 1} almost every $\theta$, $\mathcal S_{\theta}(\om) = \emptyset$, we obtain Corollary~\ref{corollary:ue:veech}.

\paragraph*{\bf Tautological planes} Applying the main result of Arana-Herrera--Chaika--Forni~\cite{AHChaikaForni}*{Theorem 1.1}, we obtain Corollary~\ref{corollary:ue:AHChaikaForni}.

\paragraph*{\bf Minimal non-uniquely ergodic directions} Corollary~\ref{cor:minimalnue} follows from Avila-Delecroix~\cite{AvilaDelecroix}*{Theorem 32}, which guarantees a non-zero Hausdorff dimension set of directions on lattice surfaces which admit measurable (but not continuous) eigenfunctions. Applying our dictionary we obtain Corollary~\ref{cor:minimalnue}.
\section{Pisot expansion factors in pseudo-Anosov homeomorphisms}\label{sect:Pisot} 
\paragraph*{\bf Proof strategy and outline} We now turn to proving the first two parts of Theorem~\ref{theorem:pisotsemiconjugacy}, namely, constructing eigenfunctions for the directional flows in the unstable direction of a pseudo-Anosov $\rho$ with Pisot expansion factor of degree $d$, and showing that the group $E(\rho) \subset \R$ of eigenvalues is isomorphic to $\Z^d$. We recall that if we have a pseudo-Anosov map $\rho$ whose expansion factor $\beta$ is not Pisot, results of Solomyak~\cite{Solomyak97ETDS}*{page 96} show that the linear flow in the unstable direction is in fact weak mixing, and thus the group $E(\rho)$ of eigenvalues is the trivial additive group $\{0\}$. We now turn to the case when $\beta$ is Pisot. We will show that the group of eigenvalues $E(\rho)$ is equal to the group $\FA(\rho) \subset \R$ of numbers $c \in \R$ such that the multiples $c\eta^u$ of a real cohomology class $\eta^u \in H^1(S,\R)$ associated to $\rho$ are (forward) asymptotic under the action of $\rho^*$ on cohomology to an integer class $\varphi^{\infty}_c \in H^1(S, \Z)$. 

\paragraph*{\bf Periodic strong Veech criterion} The inclusion of $E(\rho)$ in $\FA(\rho)$ is the \emph{strong Veech criterion} (see, for example~\cite{AvilaDelecroix}*{Theorem 33}, which is~\cite{BDM}*{Theorem 1} in a geometric context) applied to the setting of periodic Teichm\"uller geodesics, and the reverse inclusion is a \emph{converse} to this criterion.

\paragraph*{\bf Markov partitions and winding numbers} We will identify $\FA(\rho)$ with the subgroup of integer cohomology consisting of the classes $\varphi^{\infty}_c$, and using the Pisot property of $\beta$, show that this subgroup has rank $d$, and then show how to construct eigenfunctions $\psi_c$ with eigenvalue $c$ associated to each $c \in \FA(\rho)$ using the zippered rectangle decomposition of the surface corresponding to a Markov partition for the pseudo-Anosov map $\rho$. Vice-versa, to show $E(\rho) \subset FA(\rho)$, we will show that if we have an eigenfunction $\psi_c:S \rightarrow \R/\Z$ of eigenvalue $c$, there is a natural integer cohomology class $[\psi_c]$ (given by taking the winding number of the function $\psi_c$ around any closed curve) so that $c\eta^u$ is forward asymptotic to $[\psi_c]$.

\paragraph*{\bf Pseudo-Anosov maps} We recall our notation. Let $S = (X, \om)$ be a translation surface, and $\rho:S \to S$ denote a pseudo-Anosov affine automorphism, that is, the derivative $D \rho \in \SL(2,\mathbb R)$ is a hyperbolic matrix with expanding eigenvalue $\beta \in \R, |\beta| >1$, and unit (unstable and stable) eigenvectors $w^u, w^s \in {\mathbb R}^2$ satisfying $$(D\rho)w^u = \beta w^u \quad \text{and} \quad (D\rho)w^s = \beta^{-1} w^s.$$ We recall that we assume (by replacing $\rho$ with a power if necessary) that $\beta > 1$, and that $\rho$ fixes every singular point, and moreover, every separatrix (singular leaf in the stable ($w^s$) direction) on $S$.

\paragraph*{\bf Action on cohomology} Consider the action $\rho^\ast$ of $\rho$ on absolute cohomology of the surface, $H^1(S; \mathbb R)$, which is naturally isomorphic to the space of linear maps $H_1(S; \Z) \to \R$. There is a unique pair of cohomology classes $\eta^u, \eta^s \in H^1(S; \R)$ such that for every $\gamma \in H_1(S; \Z)$, the vector
$$\eta^u(\gamma) w^u + \eta^s(\gamma) w^s$$
is the holonomy vector of $\gamma$. By construction, $$\rho^\ast(\eta^u)=\beta \eta^u \quad \text{and} \quad \rho^\ast(\eta^s)=\beta^{-1} \eta^s.$$

\paragraph*{\bf Forward asymptotics} We say a class $\varphi \in H^1(S; \mathbb R)$ is {\em forward asymptotic} to an integral class $\varphi^\infty \in H^1(S; \mathbb Z)$ if
$$\lim_{n \to +\infty} \|(\rho^\ast)^n(\varphi) - (\rho^\ast)^n(\varphi^\infty)\| = 0.$$ This notion does not change if we replace $\rho$ by a power. Motivated by Veech's criterion~\cite{Veech:metrictheory}*{\S 7} (see also ~\cite{AvilaForni}*{Theorem 6.1}) for the existence of eigenvalues for translation flows, we define the set $\FA(\rho)$ by
\begin{equation}
	\label{eq:E}
	\FA(\rho) = \{c \in \R:~\text{$c \eta^u$ is forward asymptotic to an integral class}\}
\end{equation}
This is precisely Veech's criterion in the setting of periodic renormalization, which was also studied by Host~\cite{Host}. We define $$\alpha:\FA(\rho) \to H^1(S; \mathbb Z)$$ to be the map $$c \xmapsto{\alpha} \varphi^{\infty}_c$$ that sends $c \in \FA(\rho)$ to the integer class $\varphi^{\infty}_c \in H^1(S, \Z)$ to which $c \eta^u$ is forward asymptotic.  We let $A(\rho)$ denote the image $\alpha\big(\FA(\rho)\big)$, which is the collection of integer classes such that there is a $c \in \R$ for which $c \eta^u$ is forward asymptotic to that class. We note that $\FA(\rho) \subset \R$ and $A(\rho) \subset H^1(S; \mathbb Z)$ are (additive) subgroups, and that the natural map $\alpha$ is a group isomorphism from $\FA(\rho)$ to $A(\rho)$: since $0 \eta^u$ is forward asymptotic to only the zero class, so $\alpha(0)$ is the zero class. By linearity of $\rho^\ast$, if $c_i \eta^u$ is asymptotic to $\varphi^\infty_i \in H^1(S; \mathbb Z)$ for $i=1,2$, then $(c_1+c_2) \eta^u$ is forward asymptotic to $\varphi^\infty_1+\varphi^\infty_2$. It follows that $\FA(\rho) \subset \R$ and $A(\rho) \subset H^1(S; \mathbb Z)$ are subgroups and that $\alpha$ is a group homomorphism. From the first sentence, we see that $\alpha$ is one-to-one, and it is surjective onto $A(\rho)$ by definition.

\paragraph*{\bf Eigenfunctions} Let $\phi_t:S \to S$ denote the straight-line flow in direction $w^u$. The goal of this section is to show that the set of (continuous) eigenvalues for $\phi_t$ coincides with $\FA(\rho)$. Recall that a function $\Psi_c: S \rightarrow \R/\Z$ is an (additive) \emph{eigenfunction} for $\phi_t$ with eigenvalue $c$ if $$\Psi_c (\phi_t x) - \Psi_c(x) = ct.$$ We say the eigenvalue $c$ is \emph{continuous} if the associated eigenfunction $\Psi_c$ is continuous. We let $$E(\rho) = \{c \in \R: c \mbox{ is an eigenvalue for } \phi_t\}$$ be the group of (continuous) eigenvalues of $\phi_t$ (which in fact are all eigenvalues by Solomyak~\cite{Solomyak97ETDS}). We note that since the flow $\phi_t$ is (uniquely) ergodic, any two eigenfunctions with eigenvalue $c$ differ by a constant, since their difference is $\phi_t$-invariant.

\paragraph*{\bf Cohomology classes} Note that each continuous function $f: S \to \mathbb R/\mathbb Z$ has a naturally associated cohomology class $[f] \in H^1(S; \mathbb Z)$ obtained as the pullback of cohomology class of the closed $1$-form $d\theta$, where $\theta \in \R$ is the local coordinate on $\R/\Z$. We interpret $[f]$ as a linear map $H_1(S; \mathbb Z) \to \Z$. If $\gamma: \R/\Z\to S$ is a closed curve, we define $[f]([\gamma])$ to be the winding number of the composition $f\circ \gamma: \R/\Z \rightarrow \R/\Z$. That is, identifying the domain $\R/\Z$ of $\gamma$ with $[0, 1]/0 \sim 1$, lift the composition $f \circ \gamma:[0, 1] \to \R / \Z$ to a continuous map $\widetilde{f \circ \gamma}:[0, 1] \to \mathbb R$, and
$$[f](\gamma) = \widetilde{f \circ \gamma}(1) - \widetilde{f \circ \gamma}(0).$$ By construction, the cohomology class $[f]$ is invariant under composing $f$ with a rotation.

\paragraph*{\bf Eigenvalues and forward asymptotics} The goal of this section is to prove the following result showing that the sets $\FA(\rho)$ and $E(\rho)$ coincide.

\begin{Theorem}
	\label{thm:asymptotic}
	Suppose $\beta$ is Pisot and has degree $d$. Then $\FA(\rho) \subset \R$ is a subgroup of rank $d$, and $$\FA(\rho) = E(\rho).$$ Moreover, if $c \in \FA(\rho)$ there is a continuous eigenfunction $\Psi_c: S \rightarrow \R/\Z$ for $\phi_t$, such that $[\Psi_c] = \varphi_c^{\infty}$, and if $c \in E(\rho)$, with continuous eigenfunction $\Psi_c$, $c\eta_u$ is forward asymptotic to $[\Psi_c]$.
\end{Theorem}

\paragraph*{\bf Organization of the section} We will prove this theorem using renormalization ideas from Avila-Delecroix~\cite{AvilaDelecroix}, Forni~\cite{Forni}, Marmi-Moussa-Yoccoz~\cite{MarmiMoussaYoccoz}, Veech~\cite{Veech:ergodic}, and Zorich~\cite{Zorich:gauss}, many of which simplify in a very nice way since we are in the pseudo-Anosov setting, so the renormalization is periodic. We have tried to keep our proofs as self-contained as possible since the constructions are very concrete in our setting. We start (\S\ref{sec:computing}) by proving our claim on the rank of the subgroup $\FA(\rho)$. In the proof of this result, we introduce several subspaces of homology and cohomology that will be useful for us in the rest of our construction. Next, in \S\ref{sec:cohom:iet}, we use ideas from Forni~\cite{Forni} and Marmi-Moussa-Yoccoz~\cite{MarmiMoussaYoccoz} to build continuous solutions to the cohomological equation for the (self-renormalizable) interval exchange transformation (IET) obtained by taking a first return map of $\phi_t$ to an appropriately chosen segment of the \emph{stable} foliation for $\rho$. We use these solutions in \S\ref{sec:constructingeigenfunctions} to build eigenfunctions for the flow $\phi_t$, using ideas of Avila-Delecroix~\cite{AvilaDelecroix}, Veech~\cite{Veech:ergodic} and the description of the surface as a \emph{zippered rectangle} over the IET. Finally, in \S\ref{sec:thmasymptotic} we combine these constructions to prove Theorem~\ref{thm:asymptotic}.



\subsection{Computing the eigenvalues}\label{sec:computing}
\label{sect:Pisot algebraic setup} Some of the arguments in this subsection are reminiscent of arguments in the beautiful note of Kwapisz~\cite{Kwapisz}. We first prove that in the Pisot setting, the degree of $\beta$ and the rank of $\FA(\rho)$ coincide.
\begin{Prop}
	\label{prop:E intro}
	If $\beta$ is Pisot and has degree $d$, then $\FA(\rho)$ is an additive subgroup of $\R$ of rank $d$.
\end{Prop}

\paragraph*{\bf An algorithm} We will prove this as well as give an algorithm for computing $\FA(\rho)$. We do this by first explaining how to compute $A(\rho)$ and then explaining how to compute the preimage $\alpha^{-1}(v)$ for $v \in A(\rho)$. Note that the preimage under $\alpha$ of a generating set for $A(\rho)$ will give a generating set for $\FA(\rho)$.

\paragraph*{\bf Linear algebra and Perron-Frobenius} Let $g$ be the genus of $S$, and choose a basis for $H^1(S; \C)$ that is contained in $H^1(S; \mathbb Z)$, and use it to identify $H^1(S; \C)$ with $\C^{2g}$. In these coordinates $\rho^\ast:H^1(S; \mathbb C) \to H^1(S; \mathbb C)$ is given by multiplication by a $2g \times 2g$ integer matrix, and $\eta^u$ is an eigenvector with eigenvalue $\beta$. There is a real number $r \neq 0$ such that as a vector in these coordinates $r \eta^u$ lies in $\Q(\beta)^{2g}$, since by Perron-Frobenius $\beta$ is the unique eigenvalue of $\rho^\ast$ of greatest absolute value and that this eigenspace is one-dimensional~\cite{Fried}. This eigenspace is the span of $\eta^u$, so $(\rho^\ast-\beta I) v = 0$ can then be solved over $\Q(\beta)$ giving the solution $r \eta^u$ required.

\paragraph*{\bf Algebraic conjugates} Let $V \subset H^1(S; \mathbb C)$ be the $\C$-span of the vector $r \eta^u$ and its algebraic conjugates. Let $V_\Z = V \cap H^1(S; \mathbb Z)$. We will show \begin{equation}\label{eq:asrhovz} A(\rho) = V_\Z \end{equation}
Let $W \subset V$ be the $\C$-span of all algebraic conjugates of $r \eta^u$ excluding $r \eta^u$, and let $L_{\C} = \{c \eta^u:~c \in \mathbb C\}$, and $L_{\R} = \{c \eta^u:~c \in \mathbb R\}$. Then $W$ has (complex) dimension $d-1$, and $$V = W \oplus L_{\C}.$$ Note that if we write, for $v \in V$, $$v = w+l \quad \text{with} \quad l = c \eta_u \in L_{\C} \quad \text{and} \quad w \in W,$$ then $$\|(\rho^*)^n(c\eta_u) - (\rho^*)^n(v)\| = \|(\rho^*)^n(w)\| \xrightarrow{n\rightarrow\infty} 0.$$ Note that for each class $v \in A(\rho)$, there is a unique $c \in \R$ such that $c \eta^u$ is the unique point in the intersection $L_{\mathbb R} \cap (v+W)$. We claim that \begin{equation}\label{eq:alphainverse} \alpha^{-1}(v)=c.\end{equation}

\begin{proof}[Proof of Proposition \ref{prop:E intro} and equations \eqref{eq:asrhovz} and \eqref{eq:alphainverse}.]

	We now prove that $V_\Z \subset A(\rho)$ (one inclusion in \eqref{eq:asrhovz}) and that \eqref{eq:alphainverse} holds for $v \in V_\Z$. Fix $v \in V_\Z$. 
	Observe that both $L_\C$ and the translated subspace $v+W$ are contained in $V$. Furthermore, $L \not \subset W$. As $L$ and $v+W$ are complementary dimensions in $V$, they must intersect in a single point, $c \eta^u$ with $c \in \C$. In fact, we have that $c \in \R$ because both $L$ and $v+W$ are preserved by complex conjugation. Since $\beta$ is Pisot, $W$ is contained in the stable space for the linear action of $\rho^\ast$. As $c \eta^u - v$ is in $W$, we see that $c \eta^u$ is forward asymptotic to $v$ and thus $\alpha(c)=v$.
	
	It remains to show that $A(\rho) \subset V_\Z$. Fix $c \in \R$. We need to show that if $c \eta^u$ is forward asymptotic to an element $v \in H^1(S; \Z)$ then $v \in V_\Z$. Let $V_\R=V \cap H^1(S, \R)$. Observe that both $V_\R$ and $H^1(S; \Z) \setminus V_\Z$ are invariant by translations by elements $V_\Z$. Fixing a norm on $H^1(S; \R)$, the function sending $v \in V_\R$ to the minimal distance to an element of $H^1(S; \Z) \setminus V_\Z$ is positive and continuous and so by compactness of $V_\R/V_\Z$ is always greater than $\epsilon$ for some $\epsilon>0$. As $(\rho^\ast)^n(c \eta^u) \in V_\R$ for all $n$, it follows that $c \eta^u$ cannot be forward asymptotic to an element of $H^1(S; \Z) \setminus V_\Z$. Thus, $A(\rho) \subset V_\Z$ as claimed, so we have \eqref{eq:asrhovz}.
	
	Because $\beta$ has algebraic degree $d$, the subspace $V \cap H^1(S; \Q)$ is a $d$-dimensional rational subspace of $H^1(S; \Q)$ and $V_\Z=A(\rho)$ is isomorphic to $\Z^d$, and $\FA(\rho)$ is isomorphic to $A(\rho)$.
\end{proof}

\subsection{Continuous solutions to the cohomological equation for interval exchange maps}\label{sec:cohom:iet}
To start our proof of Theorem~\ref{thm:asymptotic} we follow a standard reduction to considering self-renormalizable IETs. We refer to Section \ref{sec:nonergodicdirections} for an explicit example. Fix a segment $s:[0, L) \to S$ such that $s(0)$ is a singularity, $s'(x)=w^s$ for all $x \in (0,L)$, and such that the negative $\phi_t$-trajectory of $s(L)$ hits a singularity before returning to $s\big([0, L)\big)$. Let $T_0:[0, L) \to [0,L)$ be the IET conjugate via $s$ to the first return map to $s\big([0,L)\big)$ with trajectories that hit singularities resolved by tracking points to the right. Because of the existence of the pseudo-Anosov, there is a smallest integer $p>0$ such that 
$$\rho^p \circ s(x) = s(\beta^{-p} x) \quad \text{for all $x \in [0,L)$}.$$
Then $T_0$ is periodic under Rauzy induction (up to scale) and the scaling map $x \mapsto \beta^{-p} x$ conjugates $T_0$ to the first return map $\hat T_0:[0, \beta^{-p} L) \to [0, \beta^{-p} L)$ to $T_0$, i.e.,
\begin{equation}
	\label{eq:renormalize T_0}
	\hat T_0\big(\beta^{-p} x) = \beta^{-p} T_0(x) \quad \text{for all $x \in [0, L)$}.
\end{equation}

\paragraph*{\bf Intervals of continuity} Let $I_1, \ldots, I_k$ denote the partition of $[0, L)$ into half-open intervals of continuity for $T_0$. Each interval $I_i$ is naturally associated to a class $\gamma_i \in H_1(S \setminus \Sigma; \mathbb Z)$, the homology class of a trajectory of $\phi_t$ of a point $s(x)$ with $x \in I_i$ defined until the first return to $s\big([0, L)\big)$ and closed by a segment in $s\big([0, L)\big)$. These integral classes are known to form a basis of $H_1(S \setminus \Sigma; \mathbb Q)$ \cite{YoccozSurvey}*{\S 4.5}.
We associate each vector $v \in \mathbb C^k$ to a cohomology class 
$\zeta_v \in H^1(S \setminus \Sigma; \mathbb C)$ (which we view as the space of linear maps 
$H_1(S \setminus \Sigma; \mathbb Z) \to \mathbb C$) defined by the property that
$$\zeta_v(\gamma_i) = v_i  \quad \text{for $i=1, \ldots, k$}.$$
Since the $\gamma_i$ form a basis, the map $v \mapsto \zeta_v$ is an isomorphism $\mathbb C^k \to H^1(S \setminus \Sigma; \mathbb C)$. The homomorphism $\rho$ acts linearly on $H^1(S \setminus \Sigma; \mathbb C)$ via
$$\big(\rho^\ast(\zeta)\big)(\gamma)=\zeta\big(\rho(\gamma)\big) \quad \text{for 
	$\zeta \in H^1(S \setminus \Sigma; \mathbb C)$ and
	$\gamma \in H_1(S \setminus \Sigma; \mathbb C)$.}$$
\begin{lemma}
    	\label{thm:curve case 1}
	If $v \in \C^k$ is such that $\zeta_v$ lies in the span of contracting eigenvectors for the action $(\rho^\ast)^p$ on $H^1(S \setminus \Sigma; \mathbb C)$, then there is a continuous function $\psi_v:[0, L] \to \mathbb C$ such that $\psi_v(0)=0$ and
	\begin{equation}\label{eq:cohom eq curve case 1}
		\psi_v \circ T_0(x) - \psi_v(x) = v_i \quad \text{for all $x \in I_i$}.
	\end{equation}	
\end{lemma}

\paragraph*{\bf Loops around singularities} We remark that if $\zeta_v \in H^1(S \setminus \Sigma; \mathbb C)$ lies in the span of contracting eigenvectors, then it is the image of an element of $H^1(S; \mathbb C)$ because $\zeta_v$ must send loops around the singularities to zero.

\paragraph*{\bf Cantor topology} The proof uses ideas of Gottschalk-Hedlund from, \cite{GottschalkHedlund}*{\S 14}, applied in this context by Marmi-Moussa-Yoccoz \cite{MarmiMoussaYoccoz} who use it to deduce the existence of a solution that is continuous in the Cantor topology for a broader collection of interval exchange maps (in particular beyond the periodic renormalization setting). The conclusion of Lemma~\ref{thm:curve case 1} is stronger: we obtain continuity of the function $\psi_v$ as a function of the interval with its usual topology.

\paragraph*{\bf Contracting subspaces} First observe that the set of vectors $v$ such that there is a continuous $\psi_v:[0, L] \to \mathbb C$ satisfying both $\psi_v(0)=0$ and \eqref{eq:cohom eq curve case 1} forms a subspace of $\C^k$: if $\psi_v$ satisfies both statements and $\alpha \in \C$, then $\alpha \psi_v$ satisfies both statements for the vector $\alpha v$. If $\psi_v$ and $\psi_{v'}$ are solutions for $v$ and $v'$, respectively, then $\psi_v+\psi_{v'}$ gives a solution for $v+v'$. This reduces Lemma~\ref{thm:curve case 1} to the following special case:

\begin{lemma}
	\label{lem:curve case 2}
	The conclusion of Lemma~\ref{thm:curve case 1} holds if $(\rho^\ast)^p \zeta_v= \mu \zeta_v$ for some $\mu \in \C$ with $|\mu|<1$.
\end{lemma}

\paragraph*{\bf $\C$-skew-products} We understand this result by looking at a family of skew-products over $T_0$, with skewing factor $\mathbb C$. Fix any $v \in \mathbb C^k$, and define the skew-product:
\begin{equation}
	\label{eq:tilde T_v}
 T_v:[0, L) \times \mathbb C \to [0, L) \times \mathbb C; 
	\quad 
	(x,z) \mapsto \big(T_0(x), z + v_i\big) \quad \text{when $x \in I_i$.}
\end{equation}
\paragraph*{\bf Graphs} If we had a continuous solution $\psi_v$ to \eqref{eq:cohom eq curve case 1}, note that the portion of the graph 
\begin{equation}
	\label{eq:Gamma graph}
	\Gamma=\big\{\big(x, \psi_v(x)\big):~ x \in [0, L] \big\}
\end{equation}
with $x < L$ would be $T_v$-invariant. 

\paragraph*{\bf Renormalizing skew-products} We lift the renormalization of $T_0$ to a renormalization of the skew-product $T_v$. We have:
\begin{Prop}
	\label{prop:strip renormalization}
	Fix any $v \in \C^k$.
	Let $\hat{T}_v$ denote the first return map of $T_v$ to $[0, \beta^{-p} L) \times \C$. Let $R:[0,L) \times \C \to [0, \beta^{-p} L) \times \C$ be $R(x,z)=(\beta^{-p} x, z)$. Then if $\hat v \in \C^k$ is defined such that $\zeta_{\hat v} = (\rho^\ast)^p \zeta_v$, we have 
	$$\hat{T}_v \circ R = R \circ T_{\hat v}.$$
\end{Prop}
\begin{proof}
	We introduce a geometric way to think about the skew product $T_v$. Let $\pi_2:[0,L) \times \C \to \C$ be projection to the second coordinate. Fix an $x \in [0, L)$ and an integer $n \geq 1$. We associate $x$ and $n$ to a homology class $\eta_{x,n} \in H_1(S \setminus \Sigma; \Z)$. A representative of $\eta_{x,n}$ starts at $s(x)$ flows under $\phi_t$ for $t>0$ until the $n$-th return to $s\big([0,L)\big)$ and then closes by a segment in $s\big([0,L)\big)$. As our convention with the interval exchange, we resolve a trajectory hitting a singularity by tracking points to the right. Observe that if for $m=0, \ldots, n-1$, the point $T_0^m(x)$ lies in the interval $I_{j_m}$, then in $H_1(S \setminus \Sigma; \Z)$ we have
	$$\eta_{x,n} = \sum_{m=0}^{n-1} \gamma_{j_m}.$$
	From this, our definition of $\zeta_v$, and the definition of the skew product, we conclude that for any $v \in \C^k$, we have
	\begin{equation}
		\label{eq:vertical displacement}
		\pi_2 \circ T_v^n(x,z) - \pi_2(x,z) = \zeta_v(\eta_{x,n}).
	\end{equation}
	
	Fix an $(x,z) \in [0, L) \times \C$. Let $j$ be such that $x \in I_j$. Recall the homology class $\gamma_j$. A representative is obtained by the loop which follows the flow $\phi_t$ of $s(x)$ up until the return to $s\big([0,L)\big)$ and then closes up by a segment in $s\big([0,L)\big)$. Let $\hat \gamma_j = \rho^p(\gamma_j)$. This is a loop that starts at $s(\beta^{-p}x)$ and flows under $\phi_t$ until the first return to $s\big([0,\beta^{-p} L)\big)$ and then closes by a segment in $s\big([0,\beta^{-p} L)\big)$. Let $n \geq 1$ be the number of times the $\phi_t$ trajectory of $s(\beta^{-p}x)$ crosses $s\big([0,L)\big)$ up to and including the first return to $s\big([0,\beta^{-p} L)\big)$. This is also the return time of $\beta^{-p} x$ to $[0,\beta^{-p} L)$ under $T_0$. We have $\hat \gamma_j=\zeta_{\beta^{-p} x, n}$. Then from
	\eqref{eq:renormalize T_0} and \eqref{eq:vertical displacement}, we see 
	\begin{align*}
		\hat{T}_v \circ R(x,z) &= \hat{T}_v(\beta^{-p}x, z) \\
		&= \big(\hat T_0(\beta^{-p}x), z + \zeta_v(\hat \gamma_j)\big) \\
		&= \big(\beta^{-p} T_0(x), z + \zeta_v \circ \rho^p(\gamma_j)\big) \\
		&= \Big(\beta^{-p} T_0(x), z  + \big((\rho^\ast)^p\zeta_v\big)(\gamma_j)\Big) \\
		&= R \circ T_{\hat v}(x,z).
	\end{align*}
\end{proof}

\begin{Cor}
	If  $(\rho^\ast)^p \zeta_v= \mu \zeta_v$ for some $\mu \in \C$, then 
	\begin{equation}
		\label{eq:renormalization eigenvector case}
		\hat{T}_v \circ R_\mu = R_\mu \circ T_{v}.
	\end{equation}
	where we define $R_\mu:[0,L) \times \C \to [0, \beta^{-p} L) \times \C$ by $R_\mu(x,z)=(\beta^{-p} x, \mu z)$.
\end{Cor}
\begin{proof}
	Since $v \mapsto \zeta_v$ is linear, we have $\mu \zeta_v=\zeta_{\mu v}$. Therefore in the context of Proposition \ref{prop:strip renormalization} we have $\hat v=\mu v$, and so $\hat{T}_v \circ R = R \circ T_{\mu v}$.
	Let $D_{\mu}(x,z) = (x, \mu z)$ and observe that
	$T_{\mu v} \circ D_{\mu} = D_{\mu} \circ T_v$. 
	Observing that $R_\mu=R \circ D$, we see:
	$$\hat{T}_v \circ R_\mu=\hat{T}_v \circ R \circ D_{\mu} =R \circ  T_{\mu v} \circ D_{\mu} = R \circ D_{\mu}\circ T_v= R_\mu \circ T_v.$$
\end{proof}

\paragraph*{\bf A dense orbit} We return to considering the hypothetical graph $\Gamma$ from \eqref{eq:Gamma graph}. From the hypothesis that $\psi_v(0)=0$, we already know one point that must be on such a graph: $(0,0)$ must be in $\Gamma$. Now let ${\mathcal O}_v$ denote the forward orbit of $(0,0)$ under $T_v$:
$${\mathcal O}_v = \{T_v^n(0,0):~n \geq 0\}.$$ Any continuous solution $\psi_v$ to \eqref{eq:cohom eq curve case 1} must have a graph $\Gamma$ such that ${\mathcal O} \subset \Gamma$, and this determines $\Gamma$ because $\psi_v$ is continuous and by minimality of $T_0$, the forward orbit of $0$ under $T_0$ is dense in $[0, L]$. We now show if $\zeta_v$ is an eigenvector for $(\rho^\ast)^p$ with eigenvalue less than $1$, then $\mathcal O_v$ is bounded.

\begin{Prop}
	\label{prop:bounded orbit}
	If $(\rho^\ast)^p \zeta_v= \mu \zeta_v$ for some $\mu \in \C$ with $|\mu|<1$, then the orbit ${\mathcal O}_v$ is bounded.
\end{Prop}
\begin{proof}
	We make use of the renormalization. Note that $(0,0)$ is fixed by $R_\mu$ and so \eqref{eq:renormalization eigenvector case} implies that $R_\mu({\mathcal O}_v) = {\mathcal O}_v \cap [0, \beta^{-p}L) \times \C$.
	
	For $j=1, \ldots, k$, let 
	$\hat I_j = \beta^{-p} I_j$. These are the intervals of continuity for the return map $\hat T_0$ of $T_0$ to $[0, \beta^{-p} L)$. Return times are constant on each interval $\hat I_j$; let $r_j$ denote this return time. For $n=0, \ldots, r_j-1$, the restriction of $T_0^n$ to the interval $\hat I_j$ is continuous and $T_0^n(\hat I_j)$ is contained entirely in one of the original intervals $I_i$. It follows that there is a list of displacements
	$$D_j=\{(x_{j,0}, z_{j,0}), \ldots, (x_{j,r_j-1}, z_{j,r_j-1})\} \in \R \times \C$$
	such that for $x \in \hat I_j$, $z \in \C$, $n \in \{1, \ldots, r_j-1\}$
	$$T_v^n(x,z)=(x + x_{j,n}, z+ z_{j,n}).$$
	Set $D=\{(0,0)\} \cup \bigcup_j D_j$, which is a finite collection of points in $\R \times \C$.
	As we vary $j$, the forward orbits of $\hat I_j$ up to iterate $r_j-1$ cover $[0,L)$, so we've shown that for each $(x_0,z_0) \in {\mathcal O}_v$, there is an $n_0 \geq 0$ such that $(x_1,z_1)=T_v^{-n_0}(x,z)$
	satisfies
	$$(x_1, z_1) \in {\mathcal O}_v, \quad 
	x_1 \in [0, \beta^{-p}L)
	\quad \text{and} \quad
	(x_0,z_0)-(x_1,z_1) \in D.$$
	Note that in the special case that we have $x \in [0, \beta^{-p}L)$, we take $n_0=0$ and the displacement $(x_0,z_0)-(x_1,z_1)$ is zero.
	Applying the renormalization \eqref{eq:renormalization eigenvector case} inductively, we see that for any $k \geq 0$, if $(x_k,z_k) \in {\mathcal O}_v$ and $x_k \in [0, \beta^{-pk}L)$, then there is an $n_k \geq 0$ such that $(x_{k+1}, z_{k+1}) = T_v^{-n_k}(x_k,z_k)$ satisfies
	$$(x_{k+1}, z_{k+1}) \in {\mathcal O}_v, \quad 
	x_{k+1} \in [0, \beta^{-p(k+1)}L)
	\quad \text{and} \quad
	(x_k,z_k)-(x_{k+1},z_{k+1}) \in R_\mu^k(D).$$
	This produces a sequence of points $(x_k, z_k) \in {\mathcal O}_v$ such that $x_k \in [0, \beta^{-p(k+1)}L)$ for all $k$. As we move backward along the forward orbit of $(0,0)$, there is a $K$ such that $(x_k, z_k)=(0,0)$ when $k \geq K$. Then we have
	$$(x_0,y_0)= \sum_{k=0}^{K-1} \big[(x_k,z_k)-(x_{k+1},z_{k+1})\big] \in D+R_\mu(D) + \ldots + R_\mu^{K-1}(D),$$
	where $+$ in rightmost equation denotes Minkowski sums of the sets. Since $D$ is finite and $R_\mu$ is a contraction fixing $(0,0)$, boundedness follows.
\end{proof}

\paragraph*{\bf Oscillation} To complete the proof of Lemma \ref{lem:curve case 2}, we will use the following standard criterion on oscillation for the closure $\overline{\mathcal O}_v \subset \R \times \C$ to be a graph of a continuous function.
\begin{Prop}
	\label{prop:continuity criterion}
	Let $I \subset \R$ be a closed interval and $G \subset I \times \C$ be a bounded set. Then the closure $\bar G$ is the graph of a continuous function $I \to G$ if and only if the projection of $G$ to $I$ is dense in $I$ and at each point $x \in I$, the oscillation of $G$ at $x$ is zero. Here, the {\em oscillation} of $G$ at $x \in I$
	$$\mathrm{Osc}(x) = \lim_{\epsilon \to 0} \sup \{|y'-y''|:~y',y'' \in N_\epsilon(x)\}$$
	where $N_\epsilon(x)$ is the set of all $y' \in \C$ such that there exists an $x'\in I$ such that $|x-x'|<\epsilon$ and $(x',y') \in G$.
\end{Prop}

\begin{proof}[Proof of Lemma \ref{lem:curve case 2}]
	Assume $(\rho^\ast)^p \zeta_v= \mu \zeta_v$ with $|\mu|<1$. We apply the criterion from Proposition \ref{prop:continuity criterion} to show that the closure of ${\mathcal O}_v$ is the graph of a continuous function $\psi_v:[0,L] \to \C$. We know ${\mathcal O}_v$ is bounded by Proposition \ref{prop:bounded orbit}. The projection of $\overline{\mathcal O}_v$ to $[0,L]$ is dense because it is an infinite orbit of the minimal interval exchange $T_0$. 
	
	To apply Proposition \ref{prop:continuity criterion}, it remains to check that $\mathrm{Osc}(x)=0$ for $x \in [0,L]$, where oscillation is defined as in the Proposition using the set ${\mathcal O}_v$. We begin my making some observations about the renormalization. Since ${\mathcal O}_v$ is bounded, there is an $M>0$ such that $(x',z'), (x'', z'') \in {\mathcal O}_v$  implies $|z'-z''|<M$. Using the renormalization from \ref{prop:strip renormalization}, we see that for any $k \geq 1$, 
	\begin{equation}
		x', x'' \in [0, \beta^{-kp} L ) \quad \text{implies} \quad |z'-z''| < |\mu|^k M.
	\end{equation}
	In particular, $x \in [0, \beta^{-kp} L )$ implies that $\mathrm{Osc}(x) \leq |\mu|^k M$.
	
	Because $T_v$ acts on the vertical strips $I_j \times \C$ by translation, we have 
	$$\mathrm{Osc}\big(T_0(x)\big)=\mathrm{Osc}(x) \quad \text{whenever $T_0$ is continuous at $x$}.$$
	The same statement holds for $T_0^{-1}$ replacing $T_0$. Fix any $x \in (0, L)$. The map $T_0$ has no saddle connections. Therefore, in either the positive or negative direction, the orbit of $x$ is free from discontinuities. Let $s \in \{-1, 1\}$ be positive or negative respectively. Since $T_0$ is minimal, for any $k \geq 1$, there is an $n$ such that $T_0^{sn}(x) \in [0, \beta^{-kp}L).$ We conclude that $\mathrm{Osc}(x) \leq |\mu|^k M$ for all $k \geq 1$ and so $\mathrm{Osc}(x) = 0$. It remains to check $x \in \{0, L\}$. Observe from the fact about strips that
	$$\mathrm{Osc}(0) \leq \mathrm{Osc}\big(T_0(0)\big) \quad \text{and} \quad \mathrm{Osc}(L) \leq \mathrm{Osc}\big(\lim_{x \to L^-} T_0(x)\big),$$
	so we also have $\mathrm{Osc}(0)=\mathrm{Osc}(L)=0,$ completing the verification of the oscillation condition. Thus we've shown that there is a continuous function $\psi_v:[0,L] \to \C$ whose graph is the closure of ${\mathcal O}_v$. It remains to show \eqref{eq:cohom eq curve case 1} holds. Observe that for $(x,z) \in {\mathcal O}_v$, we have $y=\psi_v(x)$. Since $(x,z) \in {\mathcal O}_v$ implies $\tilde T_v(x,z) \in {\mathcal O}_v$ and the $x$-coordinate of $\tilde T_v(x,z)$ is $T_0(x)$, we have
	$$\tilde T_v\big(x, \psi_v(x)\big) = \big(T_0(x), \psi_v \circ T_0(x)\big).$$
	Then assuming $x \in I_i$, the displacement in the $z$-coordinate by applying $\tilde T_v$ is $v_i$, so we have
	$$\psi_v \circ T_0(x)-\psi_v(x)=v_i$$
	as desired.
\end{proof}

\subsection{Constructing eigenfunctions}\label{sec:constructingeigenfunctions} We now finish showing that $$\FA(\rho) \subset E(\rho).$$ That is, we show how to build eigenfunctions using forward asymptoticity. Let $c \in \FA(\rho)$ be non-zero, and recall that $\alpha(c) = \varphi_c^\infty \in H^1(S; \mathbb Z)$ is defined to be the integral class to which $c \eta^u$ is forward asymptotic. There is a $v\in \R^k$ (depending on $c$) such that $$\zeta_{v} = \varphi_c^\infty - c\eta_u.$$ Since $c \in \FA(\rho),$ $$(\rho^\ast)^n(\zeta_{v}) \xrightarrow{n \rightarrow \infty} 0,$$  and so we can apply Lemma~\ref{thm:curve case 1} to produce a continuous function $\psi_v:[0,L] \to \R$ satisfying \eqref{eq:cohom eq curve case 1}, that is, 
		$$\psi_v \circ T_0(x) - \psi_v(x) = v_i \quad \text{for all $x \in I_i$}.$$ We will use this function and the zippered rectangle construction of Veech~\cite{Veech:ergodic} to produce an eigenfunction for the flow.

\begin{lemma}
	\label{thm:explicit eigenfunction} Let $c \in \FA(\rho)$ be non-zero, and $v, \psi_v$ as above. Then there is a continuous function $\Psi_c: S \to \mathbb R / \mathbb Z$ such that:
\begin{equation}\label{eq:explicit eigenfunction 1} \Psi_c \circ \phi_t(p) =  \Psi_c(p) + ct \pmod{\Z} \end{equation} for all $t \in \mathbb R$ and $p \in S$ such that $\phi_t(p)$ is well defined.
	We have \begin{equation}\label{eq:explicit eigenfunction 2} \Psi_c \circ s(x)=\psi_v(x)\end{equation} for all $x \in [0,L]$.

\end{lemma}

\paragraph*{\bf Zippered rectangles} To prove Lemma~\ref{thm:explicit eigenfunction}, we partition the surface $S$ into parallelograms. Let $h \in {\mathbb R}^k$ denote the height vector, that is, $h_i=\eta^u(\gamma_i)$ is the return time of $s(x)$ to $s\big([0, L)\big)$ under the flow $\phi_t$ when $x$ is in the interior of $I_i$. Recall that the domain $[0, L)$ of $T_0$ is partitioned into intervals of continuity $I_1, \ldots, I_k$. For $i=1, \ldots, k$, define the map
$$f_i:I_i^\circ \times (0, h_i)\to S; \quad (x, t) \mapsto \phi_t \circ s(x),$$
where $I_i^\circ$ is the interior of the interval $I_i$. Since $f_i$ is Lipschitz, the map has a continuous extension 
$$\bar f_i: R_i \to S \quad \text{where $R_i$ is the closed rectangle $\bar I_i \times [0, h_i]$}.$$
Define $Z=\bigsqcup R_i$ and define $\bar f:Z \to S$ such that $\bar f|_{R_i}=\bar f_i$.
The collection of images $\{\bar f(R_i):~i=1, \ldots, k\}$ form a Markov partition for $\rho$. We will define a continuous function
$$\tilde \Psi_c: Z \to \mathbb R / \mathbb Z$$
and then argue it descends to a well-defined continuous function $S \to \R/\Z$.
We define
\begin{equation}
\label{eq:tilde Psi definition}
\tilde \Psi_c|_{R_i}(x, t)=\psi_v(x)+ct \pmod{\Z}.
\end{equation}

\noindent If $\tilde \Psi_c$ descends to a function $\Psi_c:S \to \R/\Z$ in the sense that \begin{equation}\label{eq:psicexists} \tilde \Psi_c=\Psi_c \circ \bar f,\end{equation} then both \eqref{eq:explicit eigenfunction 1} and \eqref{eq:explicit eigenfunction 2} are satisfied. To see \eqref{eq:explicit eigenfunction 1}, observe that the equation is satisfied by the restriction to each parallelograms, and a finite $\phi_t$-trajectory will pass through finitely many parallelograms counting multiplicity. To see \eqref{eq:explicit eigenfunction 2} observe that $\tilde \Psi_c|_{R_i}(x, 0)=\psi_v(x)$ and $\bar f_i(x,0)=s(x)$ when $x \in \bar I_i \subset [0,L]$.
So, the existence of $\Psi_c$ is all that is left to prove.

To prove there is a $\Psi_c$ satisfying \eqref{eq:psicexists}, we need to show that if $p_1, p_2 \in Z$ satisfy $\bar f(p_1)=\bar f(p_2)$, then $\tilde \Psi_c(p_1)=\tilde \Psi_c(p_2)$. Our partition is a standard Markov partition, so the map $\bar f$ is one-to-one on each image $\bar f_i(R_i^\circ)$. Therefore, we focus on the boundaries.
Say that the {\em bottom edge} of $R_i$ is $\bar I_i \times \{0\} \subset R_i$, and the {\em top edge} is $\bar I_i \times \{h_i\}$. 
We call both bottom and top edges {\em horizontal edges}. 

\begin{Prop}
	\label{prop: construction part 1}
	If $p_1, p_2 \in Z$	both belong to horizontal edges and $\bar f(p_1)=\bar f(p_2)$, then $\tilde \Psi_c(p_1)=\tilde \Psi_c(p_2)$.
\end{Prop}

\begin{proof}
	On bottom edges, we have $\tilde \Psi_c(x,0)=\psi_v(x) \pmod{\Z}$ by definition of $\tilde \Psi_c$. In particular, $\tilde \Psi_c$ descends to a well defined function on the union of images of bottom edges.
	
	Now let $(x, h_i)$ be a point on the top edge of $R_i$. We have 
	$$\tilde \Psi_c(x, h_i) = \psi_v(x)+c h_i \pmod{\Z}.$$
	Assuming $x$ lies in the half-open interval $I_i$, we have that 
	$$\bar f_i(x, h_i)=\phi_{h_i} \circ s(x)=s \circ T_0(x)=\bar f\big(T_0(x), 0).$$
	So, we need to show that 
	$$\psi_v \circ T_0(x) = \psi_v(x)+c h_i \pmod{\Z}.$$
	From \eqref{eq:cohom eq curve case 1}, we have 
	$\psi_v \circ T_0(x) - \psi_v(x) = v_i$, so it suffices to show that $v_i = c h_i \pmod{\Z}$. Now recall that $c \eta^u = \zeta_v + \varphi^\infty$. Applying this equation to $\gamma_i$, we obtain $c h_i = v_i + \varphi^\infty(\gamma_i)$. Since $\varphi^\infty$ is an integer class and $\gamma_i \in H^1(S \setminus \Sigma; \Z)$, we have $\varphi^\infty(\gamma_i) \in \Z$ and $c h_i = v_i \pmod{\Z}$ as desired.
	
	The previous paragraph holds for all points $(x, h_i)$ on the top edge of $R_i$ unless $x$ is the right-endpoint of $\bar I_i$. Suppose $x_i$ is the right endpoint of $\bar I_i$.
	Since $T_0$ has no saddle connections, $y_i = \lim_{x \to x_i^-} T_0(x)$ must lie in the interior of some $I_j$, and we have $\bar f_i(x_i, h_i) = \bar f_j(y_i, 0)$, so we need to show that
	$$\psi_v(x_i)+c h_i = \psi_v(y_i) \pmod{\Z}.$$
	But, this follows from continuity and the previous paragraph:
	$$\psi_v(x_i)+c h_i = \lim_{x \to x_i^-} \psi_v(x)+c h_i = \lim_{x \to x_i^-} \psi_v \circ T_0(x) = \psi_v(y_i).$$
\end{proof}

\subsection{Proving Lemma \ref{thm:explicit eigenfunction}}\label{sec:expliciteigenfunction} 
Let $\hat Z = Z/\sim$, where we $p_1 \sim p_2$ if both $p_1$ and $p_2$ belong to horizontal edges and $\bar f(p_1)=\bar f(p_2)$. Proposition \ref{prop: construction part 1} shows that $\tilde \Psi_c$ descends to a well defined function $\hat \Psi_c: \hat Z \to \R/\Z$, and by construction, the map $f:Z \to S$ induces a map $\hat f: \hat Z \to S$. We need to prove that $\hat \Psi_c$ further descends to a map $\Psi_c$ on $S$, that is, for $\hat p_1, \hat p_2 \in \hat Z$ such that $\hat f(\hat p_1) = \hat f(\hat p_2) = p \in S$, that $$\hat \Psi_c(\hat p_1) = \hat \Psi_c(\hat p_2).$$ We use an idea from Avila and Delecroix \cite{AvilaDelecroix}*{pp.1201}. Since $\bar f(p_1)=\bar f(p_2) = p$, the image point $p \in S$ belongs to a singular leaf of the vertical foliation of $\rho$. By the construction of $\Psi_c$, in particular the fact that it satisfies the eigenfunction equation, we can recover the $\Psi_c$ at any point on an orbit from the value at one point on the orbit. Since there are no saddle connections, singular leaves either coincide at a non-singular point in either the past or the future.

\subsection{Proving Theorem~\ref{thm:asymptotic}}\label{sec:thmasymptotic} We now put all our arguments together to prove Theorem~\ref{thm:asymptotic}.

\begin{proof}[Proof of Theorem \ref{thm:asymptotic}] First note that for $c=0$, $c\eta_u$ is forward asymptotic to the zero class and that we can take $\Psi_c$ identically zero, which shows that $0$ is in both $\FA(\rho)$ and $E(\rho)$. Now fix a non-zero $c \in \FA(\rho)$, which exists by Proposition~\ref{prop:E intro}. The existence of $\Psi_c:S \to \mathbb R/\mathbb Z$ satisfying the eigenfunction equation is given by Lemma~\ref{thm:explicit eigenfunction}. It remains to show that we have $[\Psi_c]=\varphi_c^\infty$, where $\varphi_c^\infty$ is the integral class to which $c \eta^u$ is forward asymptotic.
	
	Recall that $w \mapsto \zeta_w$ is a group isomorphism from $\mathbb R^k$ to $H^1(S \setminus \Sigma; \mathbb R)$.
	If $h =(h_1, h_2, \ldots, h_k) \in \mathbb R^k$ is the height vector (that is, the vector of return times), then $\zeta_h=\eta^u$. As in Lemma \ref{thm:explicit eigenfunction}, let $v \in \mathbb R^k$ be such that $c \eta^u = \zeta_v + \varphi_c^\infty$. Let $v' \in \Z^k$ be such that $\zeta_{v'} = \varphi_c^\infty$. Then $c h = v + v'$.
	
	Let $\gamma_i$ be one of the basis elements for $H_1(S \setminus \Sigma; \mathbb R)$ given before. We will compute
	$[\Psi_c](\gamma_i)$. Pick a point $x$ from the interior of $I_i$. Then a realization of the class $\gamma_i$ is given by flowing $s(x)$ for time $h_i$ under $\phi_t$ until $s(T_0(x))$ is reached, then returning to $s(x)$ within the segment $s([0,L])$. We know that $\Psi_c$ increases by $ch_i$ as we move along the flow line $\phi_t$ from $x$ to $T_0(x)$. As we return along $s([0,L])$, $\Psi_c$ increases by 
	$$\psi_v(x)-\psi_v \circ T_0(x)=-v_i,$$
	by \eqref{eq:explicit eigenfunction 2} and \eqref{eq:cohom eq curve case 1}. It then follows that
	$[\Psi_c](\gamma_i)=ch_i - v_i = v'_i.$ As $\zeta_{v'}=\varphi_c^\infty$, we have $[\Psi_c]=\varphi_c^\infty$.

Finally, we prove the reverse inclusion, that $E(\rho) \subset \FA(\rho)$. We will show that if $c \in E(\rho)$, then $c\eta_u$ is forward asymptotic to the cohomology class $[\Psi_c]$. As we argued above, $[\Psi_c](\gamma_i) = ch_i -v_i$ where $v_i$ is given by the change in $\Psi_c$ along the stable segment $s([0,L])$. Since by definition $c\eta_u(\gamma_i) = ch_i$  the difference $$[\Psi_c](\gamma_i) - c\eta_u(\gamma_i)=-v_i$$
is also a change in $\Psi_c$ within $s([0,L])$. Similarly, for $n>0$, 
$$(\rho^\ast)^n\big([\Psi_c]-c \eta_u\big)(\gamma_i)=([\Psi_c]-c \eta_u)\big(\rho^n(\gamma_i)\big)$$
is a change in $\Psi$ over $\rho^m\big(s([0,L])\big)=s([0, \beta^{-n} L])$,
which tends to zero as $n \to \infty.$ This holds for all $i$, so $[\Psi_c]$ and $c \eta_u$ are forward asymptotic. 
 \end{proof}

\paragraph*{\bf Proving Corollary~\ref{cor:pisot}.} Finally, we note that our dictionary between weak mixing for the base and the ergodicity of the prism flow, as explained in \S\ref{sec:furstenberg} yields Corollary~\ref{cor:pisot}.

\section{Semi-conjugacies and torus-filling surfaces}\label{sec:tori} 

\paragraph*{\bf The invariant subspace} We now turn to proving the last two parts of Theorem~\ref{theorem:pisotsemiconjugacy}, namely, the simultaneous semi-conjugacy of the flow $\phi_t$ and the pseudo-Anosov $\rho$ to a irrational linear flow $F_t$ and a hyperbolic toral automorphism $M$ on $\TT^d$. In fact, this semi-conjugacy can be seen via the action of $\rho^{\ast}$ on the invariant subspace $V_{\R}$. Recall that $V \subset H^1(S; \mathbb C)$ is the $\C$-span of the vector $r \eta^u$ and its algebraic conjugates,  and $V_\R = V \cap H^1(S; \mathbb R)$, $V_\Z = V \cap H^1(S; \mathbb Z)$. By construction $V_{\R}$ (and $V_{\Z}$) are \emph{invariant} under the action of $\rho^{\ast}$ on $H^1(S, \C)$, and $\rho^{\ast}$ acts as a hyperbolic toral automorphism on the torus $V_{\R}/V_{\Z}$, with top eigenvalue $\beta$. 

\paragraph*{\bf Eigenvalues, eigenfunctions, and cohomology classes} We now fix $$\mathbf{c} = (c_1, \ldots c_d)$$ to be a generating set for the group of eigenvalues $E(\rho) \cong V_{\Z},$ where we recall that we can associate to each eigenvalue $c \in E(\rho)$ the integer cohomology class $[\Psi_c]$ associated to an eigenfunction $\Psi_c$. This does not depend on the choice of eigenfunction, since any two eigenfunctions differ by a rotation. We will write $v_i = [\Psi_{c_i}]$ for the generating set of $V_{\Z}$ (and basis for $V_{\R}$) associated to $\mathbf{c}$. Let $x_0 \in S$ be a fixed point of $\rho$ (which can always be found by replacing $\rho$ by a power if necessary), and let $\Psi_{c_i}: S \rightarrow \R/\Z$ be the eigenfunction with eigenvalue $c_i$ normalized so $\Psi_{c_i}(x_0)= 0.$  We define the map $\Psi_{\mathbf{c}}: S \rightarrow V_{\R}/V_{\Z}$ by $$\Psi_{\mathbf{c}}(x) = \sum_{i=1}^d \Psi_{c_i}(x) v_i.$$ 

\begin{Theorem}
    \label{thm: torus1} Fix notation as in the paragraph above. The map $\Psi_{\mathbf c}$ is continuous and surjective, and simultaneously semi-conjugates the action of $\rho$ on $S$ to the action of $(\rho^{\ast})^T$ on $V_{\R}/V_{\Z}$ and $\phi_t^u$ to the linear flow $F_t$, where $$F_t\left(\sum_{i=1}^d a_i v_i\right) = \sum_{i=1}^d (a_i+c_it) v_i.$$ That is, the following diagram commutes. 
\begin{center}
    
\begin{tikzcd}
S \arrow[r, "\Psi_{\mathbf c}"] \arrow[d, "\rho", "\phi_t^u"']
& V_{\R}/V_{\Z} \arrow[d, "(\rho^{\ast})^T", "F_t"'] \\
S \arrow[r, "\Psi_{\mathbf c}"]
&  V_{\R}/V_{\Z}
\end{tikzcd}
\end{center}

    \end{Theorem}

\paragraph*{\bf Proving Theorem~\ref{theorem:pisotsemiconjugacy}} To prove Theorem~\ref{theorem:pisotsemiconjugacy}, identify $V_{\R}/V_{\Z}$ via $\TT^d$ by mapping $v_i \mapsto e_i$, the $i$th standard basis vector. The hyperbolic automorphism $A$ is the map $(\rho^{\ast})^T$ with this identification. By abuse of notation, we will write $\Psi_{\mathbf{c}}$ to denote the induced map from $S$ to $\TT^d$. Since $(c_1, \ldots, c_d)$ generate the set of eigenvalues, and eigenfunctions are unique up to rotation, the composition $f \circ \Psi_{\mathbf c}$ of any map $f \in \Aff(\TT^d, \R/\Z)$, $f(x) = a \cdot x + b$, $a \in \Z^d, b \in \TT^d$ yields an eigenfunction with eigenvalue $\sum_{i=1}^d a_i c_i$, and all eigenfunctions arise in this way.

\paragraph*{\bf Coordinates} We note that different bases will of course yield different coordinates, and thus different hyperbolic matrices. We note that if $A_1, A_2 \in \GL(d, \Z)$ are linear maps $\R^d \to \R^d$ with eigenvalue $\beta$, and $w_1, w_2 \in \Z^d$ be primitive elements (parts of generating sets for the group $\Z^d$), there is a unique $C \in \GL(d, \Z)$ such that $C A_1 = A_2 C$ and $C w_1 = w_2$. To see this it suffices to show that for any fixed $A_2 \in \GL(d, \Z)$ with eigenvalue $\beta$ and a primitive $w_2 \in \Z^d$ there is a unique $C$ such that $C A_\beta = A_2 C$ and $C e_1 = w_2$, where $e_1$ is the standard basis vector and $A_\beta$ is the \emph{companion matrix} of $\beta$, given by
\begin{equation}
    \label{eq:M beta}
A_{\beta} = \begin{pmatrix}
    0      & 0 & 0      & \ldots & -a_0 \\
    1      & 0 & 0      & \ldots & -a_1 \\
    0      & 1 & 0      & \ldots & -a_2 \\
    \vdots &   & \ddots &        & \vdots \\
    0      & 0 & 0      & 1      & -a_{d-1} \\    
\end{pmatrix},
\end{equation} where $P_{\beta}(x) = x^d + a_{d-1} x^{d-1} + \ldots + a_1 x + a_0$ is the minimial polynomial of $\beta$. Since $\beta$ is the eigenvalue of a derivative of a pseudo-Anosov, $\beta^{-1}$ is also an algebraic integer and therefore $A_{\beta} \in \GL(d, \Z).$ The conclusions $C A_\beta = A_2 C$ and $C e_1 = w_2$ together with the form of $A_\beta$ imply that the matrix $C$ must have its $i+1$-st column given by $A_2^i w_2$. This proves the uniqueness of $C$. With this definition of $C$ we have $C e_1 = w_2$, and the proof that $C A_\beta = A_2 C$ is standard.

\paragraph*{\bf A virtual model} One can simultaneously semi-conjugate $\rho$ and $\phi_t$ to the action of the companion matrix $M_{\beta}$ in a simple form. Fix a \emph{single} eigenvalue $c_1$ and associated eigenfunction $\Psi_{c_1}$, and consider the map $\Psi_{\mathbf{c_1}}: S \rightarrow \TT^d$ given by $$\Psi_{\mathbf{c_1}}(x) = (\Psi_{c_1}(x), (\Psi_{c_1} \circ \rho)(x), (\Psi_{c_1} \circ \rho^2)(x), \ldots, (\Psi_{c_1} \circ \rho^{d-1}(x)),$$ where $$\mathbf{c_1} = (c_1, c_1\beta, c_1 \beta^2, \ldots, c_1 \beta^{d-1}).$$ Indeed $\Psi_{c_1} \circ \rho^{j}$ is an eigenfunction with eigenvalue $c_1 \beta^j$, since $$(\Psi_{c_1} \circ \rho_j) (\phi_t^u x) = \Psi_{c_1}( \phi_{\beta^j t}^u \rho_j x) = (\Psi_{c_1} \circ \rho_j)(x) - \beta^j c_1 t.$$ Direct computation shows that $\Psi_{\mathbf{c_1}}$ conjugates $\rho$ to $M_{\beta}$, and $\phi_t^u$ to the flow $F_{\mathbf{c_1}t}$ on $\TT^d$. However, although the group generated by $(c_1, c_1\beta, c_1 \beta^2, \ldots, c_1 \beta^{d-1})$ is abstractly $\Z^d$, it need not generate all of $E(\rho)$ - it could generate a finite index subgroup, so we could miss eigenfunctions even when composing with affine maps. As pointed out to us by Bianca Viray, the question of whether we can find $c_1$ so that $(c_1, c_1\beta, c_1 \beta^2, \ldots, c_1 \beta^{d-1})$ generates $E(\rho)$ is closely related to understanding the relationship between $\Z(\beta)$ and the ring of integers in $\Q(\beta)$, which is a subtle problem in algebraic number theory. See Question~\ref{ques:ring} for further questions in this direction.

\paragraph*{\bf Proving Theorem~\ref{thm: torus1}} We now turn to proving Theorem~\ref{thm: torus1}. The action of $\rho$ on $S$ maps eigenfunctions with eigenvalue $c$ to those with eigenvalue $\beta c$, and a direct computation shows that $(\rho^*)^T \Psi_{\bf c}$ and $\Psi_{\mathbf c} \circ \rho$ are both maps from $S$ to $V_{\R}/V_{\Z}$ where the $i$th coordinate is an eigenfunction for $c_i\beta$. By definition they agree at the point $x_0$. This proves that $\rho$ is semi-conjugated to $(\rho^{\ast})^T$. The semi-conjugation of the flow $\phi_t^u$ to the linear flow $F_t$ follows immediately from the eigenfunction equation, completing the proof of Theorem~\ref{thm: torus1}.\qed

\paragraph*{\bf Action on $\Q(\beta)$}  If we denote $\alpha: E(\rho) \rightarrow V_{\Z}$ the map which associates an eigenvalue $c$ to its integer cohomology class $\alpha(c) = [\Psi_c] \in V_{\Z}$, we have $$\alpha(\beta c) = \rho^* \alpha(c),$$ by the characterization via Theorem~\ref{thm:asymptotic} of $[\Psi_c]$ as the class $\varphi^{\infty}_c$ to which $c \eta^u$ is forward asymptotic to under $(\rho^{\ast})$. Indeed, we have $$(\rho^{\ast})^n(c\eta^u - \varphi_c^{\infty}) \rightarrow 0,$$ and since $\rho^{\ast} \eta^u = \beta \eta^u$, we have $$(\rho^{\ast})^n(\beta c\eta^u - \rho^{\ast}\varphi_c^{\infty}) \rightarrow 0.$$ Thus, we can see the action of $\rho^*$ on $V_{\Z}$ as the action of an integer matrix on $E(\rho) \subset \Q(\beta)$ (fixing a $\Z$-basis for $E(\rho)$), whose action corresponds to multiplication by $\beta$. This approach is worked out by Arnoux~\cite{ArnouxBSMF88}*{section 4.2. page 499} for his genus $3$ example.

\paragraph*{\bf Normalizations}
We describe a normalizations of some surfaces that makes our eigenvalues particularly easy to compute, by applying an affine map the stable and unstable directions vertical and horizontal respectively. We normalize using $\GL(2,\R)$. First observe:
\begin{Prop}
    \label{prop:affine change}
    Let $\phi_t^{\mathbf u}:S \to S$ be the unit speed on $S$ in direction $\mathbf u$, and suppose
    $\psi:S \to \R/\Z$ is an eigenfunction for $\phi_t^{\mathbf u}$ with eigenvalue $c$. Let $A \in \GL(2,\R)$ and suppose $h:S \to S'$ is an affine homeomorphism with derivative $A$. Let $\mathbf u' = \frac{A \mathbf u}{\|A \mathbf u\|}$. Then $\psi'=\psi \circ h^{-1}$ is an eigenfunction for the unit speed flow $\phi_t^{\mathbf u'}:S' \to S'$ with eigenvalue $c'=\frac{c}{\|A \mathbf u\|}$.
\end{Prop}

We leave the proof to the reader. Applying this in the setting of a pseudo-Anosov yields:

\begin{Prop}\label{prop:normalization:zbeta} Given a translation surface $S$ with an affine pseudo-Anosov $\rho$ with Pisot eigenvalue $\beta$, there is an $g \in GL(2, \R)$ such that  $$ D(\rho_g) = g(D \rho)g^{-1} = \begin{pmatrix} \beta & 0 \\ 0 & \beta^{-1} \end{pmatrix},$$ where $$\rho_g = g \rho g^{-1}: g\cdot S \rightarrow g \cdot S$$ is a pseudo-Anosov map on the surface $g\cdot S$ and that the collections of eigenvalues in the horizontal and vertical directions, $E(\rho_g)$ and $E(\rho_g^{-1})$,
are both additive subgroups of the ring of integers of the holonomy field $K_S$.
Furthermore, if $S$ is defined over its holonomy field $K_S=\Q(\beta)$, we can assume $g \in \GL(2, K_S)$.
\end{Prop}
\begin{proof}
    Applying an element of $\GL(2,\R)$ to $S$, we can assume that $S$ is defined over $K_S$. Since $\beta$ is Pisot, $D \rho$ is diagonalizable over $K_S$, so we may assume by applying the conjugating matrix that $D \rho$ has the diagonal form of $D(\rho_g)$ above.

    Now consider $FA(\rho)$ from \eqref{eq:E}, which equals $E(\rho)$ by Theorem \ref{thm:asymptotic}. Recall that $\eta^u$ measures holonomy in the unstable direction. From the normalizations of the first paragraph, $\eta^u(\gamma)$ is the
    $x$-component of holonomy, which lies in $\Q(\beta)$ when $\gamma \in H_1(S; \Z)$. Observe
    $$(\rho^\ast)^n(c\eta^u)(\gamma) = c \beta^n \eta^u(\gamma).$$
    For $x \in \R$, let $\|x\|$ denote the minimal distance from an integer. If $c \eta^u$ is
    forward asymptotic to an integer class, then for any $\gamma \in H_1(S; \Z)$ we have
    $\|c \beta^n \eta^u(\gamma)\| \to 0$ as $n \to +\infty$. It then follows from work of Pisot and Vijayaraghavan \cite{Cassels}*{Ch. 8} that $c \eta^u(\gamma) \in \Q(\beta)$.
    Since $\eta^u(\gamma) \in \Q(\beta)$, we can show (by choosing $\gamma$ such that $\eta^u(\gamma) \neq 0$) that $c \in \Q(\beta)$. We have shown $E(\rho) \subset \Q(\beta)$. Since $E(\rho)$ is a finite-rank abelian group in $\Q(\beta)$, there is an integer $m_u$ such that $m_u E(\rho)$ is contained in the ring of integers.

    Applying the same argument, we see that there is an integer $m_s$ such that $m_s E(\rho^{-1})$ is contained in the ring of integers. Then applying the diagonal matrix with entries $1/m_u$ and $1/m_s$, we get a surface whose horizontal and vertical eigenvalues are contained in $\Q(\beta)$.
\end{proof}

\begin{proof}[Proof of Corollary \ref{cor:pisot}]
    Assume $S$ is defined over its holonomy field $K_S=\Q(\beta)$. Let $\mathbf v = (v_1, v_2, v_3)$ be the flow direction on the prism $M = S \times (\R/\Z)$. By assumption $(v_1,v_2)$ is an unstable eigenvector for $\rho$. Let $g \in \GL(2, K_S)$ be as in the previous proposition.
    Then the flow in direction $\mathbf v$ in $M$ is minimal if and only if the flow in direction $\mathbf v' = \big(g(v_1, v_2); v_3\big)$ in $M' = (g \cdot S) \times (\R/\Z)$ is minimal.
    Since $g \in \GL(2, K_S)$, we have that a scalar multiple of $\mathbf v$ is in $\Q(\beta)^3$
    if and only if $\mathbf v'$ is in $\Q(\beta)^3$. From the definition of $g$, we have $v'=(x', 0, z')$ for some $x',z' \in \R$ and so these conditions are equivalent to the condition that $\frac{z'}{x'} \in \Q(\beta)$. 

    Observe that $\mathbf v'$ is proportional to $(1, 0, \frac{z'}{x'})$. Assume $\frac{z'}{x'} \in \Q(\beta)$. Using the previous proposition, we find an integer $k \geq 1$ such that $\frac{kz'}{x'} \in E(\rho)$. Let $\psi:S \to \R/\Z$ be the associated eigenfunction which satisfies $\psi \circ \phi_t = \frac{kz't}{x'} + \psi$.
    Then the function $F: M' \to \R/\Z$ sending $(p', w)$ to $w - \frac{1}{k} \psi(p')$ is invariant under flow in direction $\mathbf v'$. Thus, this flow is not minimal.

    For the converse, observe that if the flow in direction $\mathbf v'$ is not minimal, then it is not ergodic. Furstenberg's criterion for flows (Lemma~\ref{lemma:furstenberg:flow}) tells us that there is an integer $k \neq 0$ such that $k \frac{z'}{x'}$ is an eigenvalue for the flow in the horizontal direction on $S'$. But then $k \frac{z'}{x'} \in \Q(\beta)$ by the previous proposition again.
\end{proof}

\paragraph*{\bf The ring of integers} 
    A natural question is to what extent the normalization of Proposition \ref{prop:normalization:zbeta} can be chosen so that the groups of eigenfunctions for the horizontal and vertical flows on $g \cdot S$ can be arranged to be equal to the full ring of integers in the holonomy field $K_S=\Q(\beta).$ This can be arranged for all the examples in the appendix (regular $n$-gon surfaces), but we will explain why this might fail for regular $n$-gons with $n$ large.

    \paragraph*{\bf Regular $n$-gons} We focus on translation surfaces associated to the
    regular $n$-gon with $n \geq 3$. Let $a=2 \cos \frac{\pi}{n}$ if $n$ is odd or $a=2 \cos \frac{2\pi}{n}$ if $n$ is even. Then the holonomy field of $S_n$ is $\Q(a)$. It is well known that the ring of integers in $\Q(a)$ is $\Z[a]$. Using ideas related to real multiplication following McMullen \cite{McMullen03}, we have:

    \begin{Prop}\label{prob:normalization:a}
    With the notation of \ref{prop:normalization:zbeta} with $S=S_n$, the groups of eigenvalues for the $E(\rho_g)$ and $E(\rho_g^{-1})$ of eigenfunctions for the horizontal and vertical flows on $g \cdot S$ are ideals in the ring of algebraic integers $\Z[a]$.
    \end{Prop}
    \begin{proof}
        Since $E(\rho_g)$ is a subgroup of $\Z[a]$, it suffices to check that multiplication by an element $x \in \Z[a]$ satisfies $x E(\rho_g) \subset E(\rho_g)$. Since $\Z[a]$
        is generated by $a$, it suffices to check that $a E(\rho_g) \subset E(\rho_g)$.
        
        Let $\eta:H_1(S; \Z) \to \R^2$ be the holonomy map. We observe that there is an endomorphism $L:H_1(S; \Z) \to H_1(S; \Z)$ such that $\eta \circ L=a \eta$. Such an $L$ can be produced by an integer-weighted sum of actions on $H_1(S; \Z)$ coming from affine automorphisms. To see this, observe that
        $$a I = \begin{cases}
            AB + BA - 2B & \quad \text{when $n$ is odd}, \\ 
            -2 I - (A - B)^2 & \quad \text{when $n$ is even}, \\ 
        \end{cases}$$
        where $A$ and $B$ are the Veech group generators from \eqref{eq:Veech odd} or \eqref{eq:Veech even}, respectively.
    
        We now claim that $c \in E(\rho)$ implies $ac \in E(\rho)$. To see this suppose $c \in E(\rho)$. Then $c \eta^u$ is forward asymptotic to some integer class $\varphi$. Observe that
        $ac \eta^u =c \eta^u \circ L$ is forward asymptotic to the integer class $\varphi \circ L$, so $ac \in E(\rho)$.
    \end{proof}

\paragraph*{\bf Principal ideals}    Recall that an ideal is {\em principal} if it is generated by a single element. If both     $E(\rho_g)$ and $E(\rho_g^{-1})$ are principal, then we can further multiply $g \cdot S$ by a diagonal matrix $D$ whose diagonal entries are these generators to ensure that $D \cdot g \cdot S$ has horizontal and vertical eigenvalues given precisely by $\Z[a]$. If $\Z[a]$ is a \emph{principal ideal domain} (equivalently has class number one), then all ideals are principal and it follows that we can arrange that the normalized eigenvalues are precisely the ring of integers $\Z[a]$. This holds for $n$ small \cite{derLindenFranciscus} (at least $n \leq 35$), but not in general. 

    \begin{ques}\label{ques:ring}
        Is it true that for every pseudo-Anosov $\rho:S_n \to S_n$ with Pisot expansion factor, the group of eigenvalues $E(\rho)$ is a constant multiple of the ring of integers of the holonomy field, $\Z[a]$?
    \end{ques}

\section{Genus 3 examples}\label{sec:genus3} 
\paragraph*{\bf Arnoux-Yoccoz surfaces} As we have mentioned, Arnoux~\cite{ArnouxBSMF88} constructed an explicit genus $3$ example of a pseudo-Anosov map on a genus $3$ surface with expansion factor $1/\alpha$, where $\alpha^3+\alpha^2+\alpha= 1$, and showed how to simultaneous semi-conjugate this (along with its unstable flow) to an explicit hyperbolic automorphism of the $3$-torus and a flow in an irrational direction on this $3$-torus, which shows that this flow is not weak-mixing (and in fact, builds an at least rank $3$ group of eigenvalues). This example is not a lattice surface, and does not arise from billiards. It has many extraordinary properties, including the fact that the semi-conjugacy in fact conjugates the first return map to an appropriately chosen interval of the stable foliation to a toral translation, in particular, the associated IET is also non weak-mixing. As we will see below, this is different from our situation, as the interval exchange transformation we study in our example on the double heptagon surface is indeed weak-mixing. We also note (as we mentioned above) that Arnoux's construction generalizes to higher-genus surfaces, to pseudo-Anosov maps with expansion factor $\frac{1}{\alpha_n}$, where  $\alpha^n + \alpha^{n-1} + \ldots +\alpha^2+\alpha= 1$.

\paragraph*{\bf Do-Schmidt examples} As we discussed in the introduction, there is in fact a large family of examples of pseudo-Anosov map on translation surface of genus $3$, found by Do-Schmidt~\cite{DoSchmidt}*{Theorem 2}. They showed that for each integer $k \geq 2$, there exist genus $3$ translation surfaces (in the hyperelliptic component of the stratum $\hh(2,2)$) which admit orientable pseudo-Anosov
maps with expansion factor $\beta_k$ whose minimal polynomial is given by $P_k(x) = x^3
- (2k+ 4)x^2 + (k+ 4)x -1$. Since the other eigenvalues of the polynomial $P_k$ are real and have absolute value less than $1$, $\beta_k$ is Pisot, and so we can produce explicit non-weak mixing directions for the linear flows in the unstable direction for these surfaces.

\paragraph*{\bf Degrees of field extensions} We now show a bound of the degree of potential Pisot expansion factors for pseudo-Anosov maps on genus $g$ surfaces. The interactions between the degrees of expansion factors, traces, and genus have been studied by many authors, see for example the work of Kenyon-Smillie~\cite{KenyonSmillie}*{Appendix 7} and \cite{GutkinJudge}*{\S 7}.

\begin{lemma}\label{lemma:pisotbound} Let $\beta$ be the expansion factor for a pseudo-Anosov map $\rho$ of a genus $g>1$ surface and suppose $\beta$ is Pisot. Then $$\deg_{\Q}(\beta) \le g.$$
\end{lemma}
\begin{proof} If $\beta$ is an expansion factor of a pseudo-Anosov map on $S$, write $K = \Q(\beta)$. The action of $\rho^*$ on cohomology $H^1(S, \Z)$ gives a $2g \times 2g$ symplectic integer matrix for which $\beta$ (and $1/\beta$) are eigenvalues, so $[K:\Q] \le 2g$. The zeros of the characteristic polynomial $P_{\rho}$ of $\rho^*$  occur in reciprocal pairs, since $\rho^*$ is symplectic. Write $$P_{\rho}(x) = (x-\beta)(x- \beta^{-1}) \prod_{i=1}^{g-1} (x-\alpha_i) (x- \alpha_i^{-1}),$$ where $|\alpha_i| \geq 1$. If $P_{\beta}$ is the minimal polynomial of $\beta$,  $$P_{\beta}(x) = (x-\beta)(x- \beta^{-1})^{\epsilon_0} \prod_{i=1}^{g-1} (x- \alpha_i^{-1})^{\epsilon_i},$$ where $\epsilon_i$ can take on the values $0$ or $1$, which yields $[K:\Q] \le g+1$. If $\epsilon_0 = 1$, the polynomial $P_{\beta}$ is palindromic, and all roots occur in reciprocal pairs, so if $\alpha_i^{-1}$ is a root, so is $\alpha_i$, which would mean $\beta$ is not Pisot. Thus, we have our bound $[K:\Q] \le g$.
    
\end{proof}

\paragraph*{\bf Trace fields} We note that for $K = \Q(\beta)$, the bound $[K:\Q] \le 2g$ can also be seen by a bound on the degree of the \emph{trace field} $K'$ of $S$, which is given by $\Q(\tau)$, $\tau = \beta + \frac{1}{\beta}$. We have either $$[K:K'] = 1 \mbox{ or } 2.$$ Kenyon-Smillie~\cite{KenyonSmillie}*{Lemma 27} shows that $[K': \Q] \le g$, so $[K:\Q] \le 2g$. 

\begin{Cor}\label{cor:pisotbound:genus3} Let $\rho$ be a pseudo-Anosov map of a genus $3$ surface $S$ with expansion factor $\beta$. Then $\beta$ is Pisot (and the linear flow in the unstable direction is not weak mixing) if and only if $[\Q(\beta):\Q] \le 3$.   
\end{Cor}

\begin{proof} As noted in \S\ref{sec:nonergodic} if $[\Q(\beta):\Q] = 2$, Franks-Rykken~\cite{FranksRykken} showed that $S$ (of any genus) must be square-tiled, and so the flow is not weak-mixing (there is a covering map to the torus), and $\beta$ is Pisot. If $[Q(\beta):\Q] = 3$, we will show $\beta$ is in fact Pisot. Again, we write $K = \Q(\beta)$, $K' = \Q(\tau)$. If $[K:\Q]=3$, we have $[K:K'] =1$, that is, $\beta \in K'$. Thus, $\beta$ and $\beta^{-1}$ are not conjugates, since if they were, and $\beta_2$ was the third root of the common minimal polynomial of $\beta$ and $\beta^{-1}$, we would have $$\beta \beta^{-1} \beta _ 2 = 1,$$ and thus $\beta_2$ =1, which is impossible. Thus $P_{\rho}(x) = P_{\beta}(x) P_{\beta^{-1}}(x).$  The six roots of $P_{\rho}$ split into two groups, $(\beta, \alpha_1^{-1}, \alpha_2^{-1})$, $(\beta^{-1}, \alpha_1, \alpha_2)$, where \begin{equation}\label{eq:beta} \beta \alpha_1^{-1} \alpha_ 2^{-1} = \beta^{-1} \alpha_1\alpha_2 = 1 \end{equation} and 
$$\beta > \max\{ |\alpha_1|, |\alpha_2|, |\alpha_1^{-1}|, |\alpha_2^{-1}|\} > \beta^{-1}.$$ To see that $\beta$ is Pisot, note that if $|\alpha_1^{-1}| \geq 1$, then $\beta \alpha_ 2^{-1} = \alpha_1$ would imply $\beta|\alpha_2^{-1}| \le 1$, so $\beta \le |\alpha_2|$, a contradiction. The same argument of course applies to $\alpha_2^{-1}$. So we have shown $\beta$ is indeed a Pisot number. Finally, if $\beta$ is Pisot, we can apply Lemma~\ref{lemma:pisotbound} to see $[\Q(\beta):\Q]\le 3$.
    
\end{proof}

\subsubsection{Vanishing Sah-Arnoux-Fathi invariant and special pseudo-Anosovs}\label{sec:SAF} The work of Do-Schmidt~\cite{DoSchmidt}*{Theorem 1} relates the minimial polynomial of the expansion factor of (orientable) pseudo-Anosov maps to the \emph{Sah-Arnoux-Fathi}, or \emph{SAF}-invariant of the associated interval exchange map. A consequence of their result is that if $\beta$ is a Pisot expansion factor of a pseudo-Anosov map, and the degree of $\beta$ is greater than $2$ (that is, the underlying surface is not square-tiled) then the SAF invariant vanishes. So in particular, if a pseudo-Anosov direction is not weak mixing, it must have vanishing SAF invariant. Further results in this direction (with many explicit examples) have been obtained by Winsor~\cite{Winsor}.

\subsection{Non-ergodic directions}\label{sec:nonergodicdirections} 
We now turn to proving the last part of Theorem~\ref{theorem:prisms}, namely producing explicit non-ergodic directions for the prism flow on the prism over $\omega_7$. Let $a=2\cos(\pi/7)$. The minimal polynomial for $a$ is $$p_a(x) = x^3-(x^2+2x-1).$$ The matrix $$\begin{pmatrix}
1&a\\a&a^2+1
\end{pmatrix}$$ has maximal eigenvalue $\beta=a^2+a$, with eigenvector $(1, a^2-1)$. The other eigenvalue is $1/\beta = 2-a$. The number $\beta$ is a Pisot number with minimal polynomial $$p_{\beta}(x) = x^3-6x^2+5x-1.$$ 
It is the expansion factor of a pseudo-Anosov map defined on this surface by Boulanger~\cite{Boulanger}.
We note that this is the same polynomial as the $k=1$ Do-Schmidt polynomial, however, their construction yields a pseudo-Anosov map of a surface in (the hyperelliptic component of) $\hh(2, 2)$, whereas our example is for a surface in $\hh(4)$.

%


\paragraph*{\bf Staircase model} Consider $u=a/(a^2-1)=a-1$
and $y=1/(a^2-1)=a^2-a-1$.
It defines a surface, see Figure \ref{fig:double-hepta}, which is a staircase surface in the $GL^+(2,\R)$-orbit of the double heptagon, with the matrix $$\begin{pmatrix}
    1+a/2&1+a/2\\ -\sin(\pi/7)&\sin(\pi/7)
\end{pmatrix}$$ mapping the staircase to the double regular heptagon, see for example Boulanger~\cite{Boulanger} (or the PhD thesis of T.Monteil~\cite{Monteil}*{Figure 10}).
\begin{figure}[h!]
\begin{tikzpicture}[scale=.75]
\draw (-.5,1)--(-.5,1.8)--(0,1.8)--(1,1.8)--(1,1)--(1.8,1)--(1.8,-.5)--(1,-.5)--(1,0)--(0,0)--(0,1)--cycle;
\draw[dashed] (0,1)--(1,1);
\draw[dashed] (0,1)--(0,1.8);
\draw[dashed] (1,0)--(1,1);
\draw[dashed] (1,0)--(1.8,0);
\draw (-.3,.9) node{\tiny $y$};
\draw (1.1,1.5) node{\tiny $u$};
\draw (.5,-.2) node{\tiny $1$};

\draw (2.1,-.2) node{\tiny $y$};
\draw (1.4,-.8) node{\tiny $u$};
\draw (2.1,.5) node{\tiny $1$};
\end{tikzpicture}
\begin{tikzpicture}[scale=.75]
\draw (-.5,1)--(-.5,1.8)--(0,1.8)--(1,1.8)--(1,1)--(1.8,1)--(1.8,-.5)--(1,-.5)--(1,0)--(0,0)--(0,1)--cycle;
\draw[red] (-.5,1.8)--(0,1.8)--(1,1)--(1.8,0)--(1.8,-.5)--(1,0)--(0,1)--cycle;
\end{tikzpicture}
\begin{tikzpicture}[scale=.75]
\draw[red] (-.5,1.8)--(0,1.8)--(1,1)--(1.8,0)--(1.8,-.5)--(1,0)--(0,1)--cycle;
\draw[blue] (0,0)--(-.8,1);
\draw[blue, dashed] (1,0)--(0,0);
\draw[blue, dashed] (1,0)--(1,-.8);
\draw[blue] (1,-.8)--(0,0);
\draw[blue, dashed] (0,1)--(-.8,1);
\draw[blue] (-.8,1)--(-.8,1.5);
\draw[blue] (-.8,1.5)--(0,1);
\draw[blue, dashed] (1,-.8)--(1,0);
\draw[blue] (1,-.8)--(1.8,-.8)--(1,0);
\draw[blue, dashed] (0,0)--(0,1);
\end{tikzpicture}
\begin{tikzpicture}[scale=.75]
\draw (0,0)--(1,0)--++(51.4:1)--++(102.8:1)--++(154.2:1)--++(205.6:1)--++(257:1)--cycle;
\draw (0,0)--(1,0)--++(-51.4:1)--++(-102.8:1)--++(-154.2:1)--++(-205.6:1)--++(-257:1)--cycle;
\end{tikzpicture}
\caption{Surface of the double heptagon}\label{fig:double-hepta}
\end{figure}
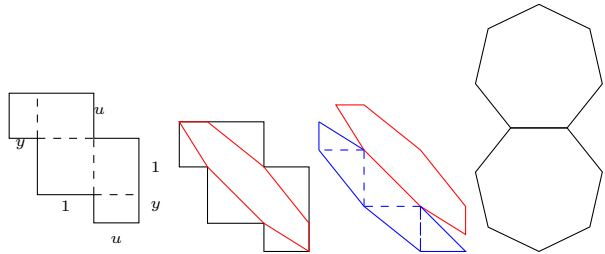


\begin{figure}[h!]
\begin{center}

\begin{tikzpicture}[scale=3]
\draw (0,0)--(0,1)--(-.44,1)--(-.44,1.81)--(0,1.81)--(1,1.81)--(1,1);
\draw (0,0)--(1,0)--(1,-.44)--(1.81,-.44)--(1.81,0)--(1.81,1);
\draw[dashed] (0,0)--(0,1)--(1,1)--(1,0);
\draw[] (1,1)--(1.81,1);
\draw[dashed] (0,1)--(0,1.81);
\draw[dashed](1,0)--(1.81,0);

\draw[red] (0,1)--++(2.41*7.5/10,-7.5/10);

\draw[dotted] (0,.25)--++(1.81,0);
\draw[red] (0,.25)--++(2.41/4,-1/4);

\draw[dotted] (.6,0)--++(0,1.9);
\draw[red] (.6,1.81)--++(2.41/6,-1/6);

\draw[dotted] (1,1.65)--++(-1.44,0);
\draw[red] (-.44,1.65)--++(2.41/10,-1/10);

\draw[dotted] (-.44,1.18)--++(1.44,0);

\draw[green] (0,1.81)--++(-1/3,-2.41/3);
\draw[green] (-.32,1.81)--++(-1/8.3,-2.41/8.3);
\draw[dotted] (-.32,1)--++(0,.81);

\draw[green] (1,1.81)--++(-1/16,-2.41/16);
\draw[green] (1,1)--++(-1/6.5,-2.41/6.5);
\draw[green] (1.81,1)--++(-1/3.8,-2.41/3.8);
\draw[green] (1.81,0)--++(-1/5.5,-2.41/5.5);
\draw[dotted] (1.63,-0.44)--++(0,1.44);
\draw[green] (1.63,1)--++(-1/4,-2.41/4);

\draw[blue] (-.44,1)--++(1/3,2.41/3);
\draw[blue, dashed] (0,1)--++(1/3,2.41/3);
\draw[dotted] (-.1,1.81)--++(0,-.81);

\draw[blue] (-.1,1)--++(1/3,2.41/3);

\draw[dotted] (.33,1.81)--++(0,-2);

\draw[blue] (.33,0)--++(1/23,2.41/23);
\draw (0.23,0)--++(1/17,2.41/17);
\draw[dotted] (0.23,0)--++(0,2);
\draw[blue] (0,0)--++(1/11,2.41/11);
\draw[blue] (1,0)--++(1/4.9,2.41/4.9);
\draw[blue] (1,-.44)--++(1/2.72,2.41/2.72);

\fill (0,0) circle(.05);
\fill (0,1) circle(.05);
\fill (1,0) circle(.05);
\fill (1.81,0) circle(.05);
\fill (1.81,-.44) circle(.05);
\fill (1,-.44) circle(.05);
\fill (-.44,1) circle(.05);
\fill (-.44,1.81) circle(.05);
\fill (0,1.81) circle(.05);
\fill (1,1.81) circle(.05);
\fill (1,1) circle(.05);
\fill (1.81,1) circle(.05);

\fill[pink, opacity=.5] (1.63,-.44)--(1.81,-.44)--(1.81,0)--cycle;
\fill[pink, opacity=.5] (1,0)--(1,-.44)--(1.367,0.446)--(1.205,.491)--cycle;
\fill[pink, opacity=.5] (1.63,1)--(1.81,1)--(1.546,.365)--(1.38,.42)--cycle;

\fill[blue, opacity=.5] (1,-.44)--(1.63,-.44)--(1.81,0)--(1.807,.25)--(1.36,.43)--cycle;

\fill[blue, opacity=.5] (1,1)--(1.63,1)--(1.38,.42)--(.846,.64)--cycle;
\fill[blue,opacity=.5] (0,0)--(0,.25)--(.092,.22)--cycle;

\fill[red, opacity=.5] (.23,0)--(.33,0)--(.37,.09)--(.28,.15)--cycle;
\fill[red, opacity=.5] (-.4,1.64)--(-.44,1.65)--(-.44,1.81)--(-.32,1.81)--cycle;
\fill[red, opacity=.5] (-.44,1)--(-.34,1)--(0,1.81)--(-.1,1.81)--cycle;
\fill[red, opacity=.5] (-.1,1)--(0,1)--(.34,1.81)--(.22,1.81)--cycle;
\fill[red, opacity=.5] (1,1.81)--(1,1.63)--(.937,1.67)--cycle;
\fill[yellow, opacity=.5] (0,1)--(0,.25)--(.092,.22)--(.363,.841)--cycle;
\fill[yellow, opacity=.5] (1.81,.25)--(1.81,1)--(1.546,.365)--cycle;

\fill[green, opacity=.5] (1,1.65)--(1,1)--(.846,.64)--(0,1)--(.34,1.81)--(.6,1.81)--cycle;
\fill[green, opacity=.5] (.33,0)--(.6,0)--(.37,.09)--cycle;
\fill[green, opacity=.5] (-.44,1.65)--(-.44,1)--(-.21,1.55)--cycle;

\fill[yellow, opacity=.5] (.6,1.81)--(1,1.81)--(.937,1.67)--cycle;
\fill[yellow, opacity=.5] (.092,.22)--(.6,0)--(1,0)--(1.2,.49)--(.36,.85)--cycle;

\draw (0,0) node[left]{$(0,0)$};
\end{tikzpicture}
\end{center}
\caption{First return of the flow of slope $-y = 1+a-a^2$ to a transverse (indeed, orthogonal) interval on the staircase surface associated to the double heptagon.}\label{fig-iet}
\end{figure}
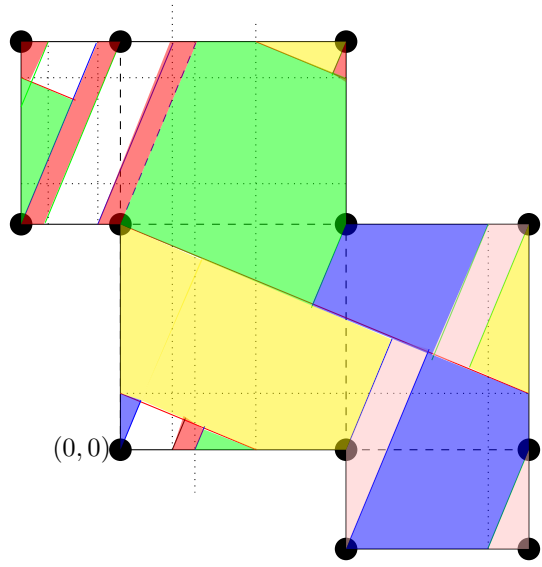


\begin{lemma}\label{lemma:7gonIET}
Consider the line starting at $(0,1)$ of slope $-y$, and the first return map $T_{\lambda, \pi}$ of this linear flow in the direction of slope $y$ on the staircase surface. The map $T_{\lambda, \pi}$ is an interval exchange map of six intervals with permutation 
$$\pi = \begin{pmatrix}
1&2&3&4&5&6\\
4&3&2&6&5&1 
\end{pmatrix}$$
\noindent and lengths of intervals given in the proof. The interval exchange $T_{\lambda, \pi}$ is self-similar with substitution (where we code the intervals by letters in the alphabet $\{1,\dots, 7\}$)
$$\sigma_7: \begin{cases}1\mapsto 14164\\ 2\mapsto 142324\\ 3\mapsto 1423324\\ 4\mapsto 1424\\ 5\mapsto 154154164\\ 6\mapsto 154164  \end{cases}$$

\end{lemma}

\begin{rmk}
The linear flow on the surface on the considered direction is thus the suspension flow over the subshift generated by the substitution $\sigma_7$ with heights given by $H$. Moreover  $L$ and $H$ are respectively, right and left eigenvectors of the incidence matrix of the substitution. 
\end{rmk}

\paragraph*{\bf Weak mixing} The main result of this subsection is that the interval exchange map defined in Lemma~\ref{lemma:7gonIET} is weakly mixing:
\begin{Prop}\label{prop:eigenvalues-subst}
The interval exchange map defined in Lemma~\ref{lemma:7gonIET} and the associated subshift $(X_{\sigma_7}, S)$ are weakly mixing (no (additive) eigenvalue except $0$).
\end{Prop}

\subsection{Proof of Lemma \ref{lemma:7gonIET}}
We refer to Figure \ref{fig-iet}. On Figure \ref{fig-iet} the first interval corresponds to green intervals, the second one to the blue, the third to the pink, the fourth to the yellow, the fifth to the red and the sixth to the white.
The vertices of the polygon are, in one cyclic order, the following points
$$\begin{cases} (-a^2+a+1,a), (0,a), (1,a), (1,1),(a,1)\\(a,0), (a,-a^2+a+1), (1, -a^2+a+1), (1,0), (0,0),(0,1), (-a^2+a+1,1)
\end{cases}$$

\paragraph*{\bf Lengths and heights} The interval exchange given in Lemma~\ref{lemma:7gonIET} has lengths and heights given by the vectors $L, H$: 
$$\begin{array}{|c|c|c|c|c|c|c|}
\hline
&I_1&I_2&I_3&I_4&I_5&I_6\\
\hline
L&1&2a-3&2-a&a^2-a&-3a-1+2a^2&-3a^2+5a+1\\
\hline
H&1&a^2-a&a&a-1&a^2-1&a^2-a\\
\hline
\end{array}$$
In order to construct the interval exchanges, we need to start from the line and move by the flow along the direction with slope $1/y$  until we come back to the initial line. We deduce that it suffices to project all the vertices of the polygon on this line, and compute the lengths.

\paragraph*{\bf Permutations} To find the permutation associated to the interval exchange, we start from the first line and flow in the direction of slope $1/y$. We see, on Figure \ref{fig-iet} that, for example, the first intersection of the image of $I_1$  by the flow with the interval is at the end. 

\paragraph*{\bf Lengths} We sketch the computations of the lengths and leave the details to the reader. The line which supports the interval has slope $-y$,  We recall that the projection of the vector $(u,v)$ on the line of slope $\alpha$ gives an interval of length $\frac{|u+\alpha v|}{1+\alpha^2}$.

\paragraph*{\bf Heights} To compute the heights, we first compute the coordinates of the points at the intersection of the line of slope $1/y$ and the polygon:

$$\begin{pmatrix}
0\\1
\end{pmatrix},
\begin{pmatrix}
a\\1-u
\end{pmatrix},
\begin{pmatrix}
0\\1-u
\end{pmatrix},
\begin{pmatrix}
\frac{1-u}{y}\\0
\end{pmatrix},
\begin{pmatrix}
\frac{1-u}{y}\\a
\end{pmatrix},
\begin{pmatrix}
1\\2-y
\end{pmatrix},
\begin{pmatrix}
-y\\2-y
\end{pmatrix}.
$$
Now we project the vectors on the direction and obtain the result.

\subsection{Mercat's construction}\label{sec:mercat}

Consider a primitive substitution $\sigma$ on a finite alphabet $\mathcal A$. Mercat~\cite{Mercat} gives an algorithm to compute eigenvalues of the subshift $X_\sigma$ associated to the substitution and also for the suspension flow constructed from $X_\sigma$ and a vector $H$. In order to describe it, we first define a coboundary:

\paragraph*{\bf Coboundaries} For a substitution $\sigma$, the associated language is the set of words in $\mathcal A^*$ which appear in some element of $X_\sigma$. If $w$ is a word, then $ab(w)$ is an element of $\R^d$ which counts the occurrences of each letter in $w$. If $a$ is a letter of $\mathcal{A}$, then a \emph{return word} of $a$ is a word $w$ in the language such that $awa$ is also in the language and contains only two ocurences of $a$. A \emph{coboundary} is a morphism $F: \mathcal A^*\to S^1$ such that $F(w)=1$ if $w$ is a return word of a letter $a$ in the subshift defined by the substitution. This map can be lifted to a linear map $f$ from $A^*$ to $\mathbb R$ such that $f(ab(w))=0$ where $w$ is a return word of a letter $a$. 

\paragraph*{\bf Hosts's criterion} The existence of coboundaries can be reformulated via the following proposition due to Host \cite{Host}:
\begin{Prop}
    A morphism $h: \mathcal A^*\to S^1$ is a coboundary if and only if there exists a function $F:\mathcal A\to S^1$ such that for every word $ab$ in the language of the substitution one has $F(b)=F(a)h(a)$
\end{Prop}

\paragraph*{\bf Mercat's algorithm: coboundaries} Mercat~\cite{Mercat}*{\S 5.2} gives an algorithm to compute coboundaries for each letter $a$:
\begin{description}
    \item[Fixed point] Iterate the substitution, and find a fixed point starting with $a$.
    \item[Return words] Find a return word to $a$ which is a prefix of the fixed point. Compute its image by the substitution $\sigma$.
    \item[Decompose] Decompose these images into return words. 
    \item[Iterate] Iterate this process until it stabilizes.
    \item[Abelianizations] List their abelianizations in a matrix (by rows), and find the \emph{left} kernel. 
\end{description}

\paragraph*{\bf Mercat's algorithm: eigenvalues} There is an associated algorithm to find eigenvalues:
\begin{description}
\item[Perron eigenvectors] Compute $V_1$, the right eigenvector of the substitution matrix $M_\sigma$ for the Perron eigenvalue with normalization $(1\dots 1)M_\sigma V_1=1$. 

\item[Generalized expanding eigenvectors] Let $V$ be the set of generalized eigenvectors of $M$ for eigenvalues of modulus $\geq 1$.
\item[Coboundary matrix] Let $C$ be the \emph{coboundary matrix}, with first row $(1\dots 1)M_{\sigma}$ and other rows forming a basis of coboundary.
\item[Integral sublattice] Let $W=\{w\in \mathbb Z^d, \exists x\in \mathbb R^k, xCV=wV\}$,  be the sublattice of $\Z^d$, viewed as row vectors, where $x,w$ are row vectors and $k$ is the number of rows of the coboundary matrix $C$.
\item[Conclusion] The set of eigenvalues is $WV_1$.
\end{description}

\paragraph*{\bf Suspensions} Moreover if we consider the suspension of this subshift by a positive vector $H$, the algorithm can be modified in order to find the eigenvalues of the vertical flow on this suspension, by normalizing $H$ so that $^tH V_1=1$ and replacing the first row of the matrix $C$ by $(H_1\dots H_d)$.

\subsection{Proof of Proposition \ref{prop:eigenvalues-subst}}\label{sec:mercat:eigenvalues}
We now apply Mercat's algorithm to compute the eigenvalues of our interval exchange and the associated linear flow on the double heptagon. Since the interval exchange is conjugated to the subshift defined by the substitution $\sigma_7$, we will prove that this subshift is weak mixing. We use Lemma \ref{lemma:7gonIET}. The incidence matrix of the substitution is
$$M_{\sigma_7} =\begin{pmatrix}
2&1&1&1&3&2\\ 0&2&2&1&0&0\\ 0&1&2&0&0&0\\ 2&2&2&2&3&2\\0&0&0&0&2&1\\ 1&0&0&0&1&1
\end{pmatrix}$$

\paragraph*{\bf Eigenvalues} Recall that $a=2\cos(\pi/7)$, it has minimal polynomial $X^3-(X^2+2X-1)$.
The eigenvalues of $M$ which are greater than $1$ are $a+a^2,a^2, 3+a-a^2$. thus $a^2+a$ is the maximal eigenvalue, and the matrix has determinant $1$. We now compute coboundaries for the IET described in Lemma~\ref{lemma:7gonIET} following the algorithm in \S\ref{sec:mercat}: the return words to $1$ in $X_{\sigma_7}$ are 
$$\{14,154, 164, 1424, 142324, 1423324\}.$$ It gives a matrix with zero left kernel.
Thus $$C=\begin{pmatrix}1& 1& 1&1&1&1
\end{pmatrix}M_{\sigma_7}=\begin{pmatrix}
    5&6&7&4&9&6
\end{pmatrix}.$$
\noindent We find $$V=\begin{pmatrix}
a^2+a+3/2&1&1\\ 3a^2-4a-5/2&2-a^2&a^2-a-1\\-2a^2+5a/2+2&-1&-1\\-a^2/2+3a/2-1&a^2-a-1&a\\ 5a/2-9/2&2+a-a^2&1-a\\ 13/2-9a/2+a^2/2&2a^2-a-4&-a^2+2a\end{pmatrix}.$$

\noindent and the space $W$ is the set of $w\in \mathbb Z^6$ such that there exists $r$ real number such that
$$\begin{cases}
r=wV_1\\ (a^2-a)r=wV_2\\ (a+1)r=wV_3
\end{cases} $$
Thus we obtain $$\begin{cases}
    w[(a^2-a)V_1-V_2]=0\\ w[(a+1)V_1-V_3]=0
\end{cases}$$


This can be reduced to the matrix equation

$$w\begin{pmatrix}
    -a^2/2+3a/2-2&-a^2+a/2+3/2\\
    9a^2/2-17a/2+2&a^2+a/2-9/2\\
    -2a^2+5a-3/2&-3a^2/2+a/2+5\\
    -3a^2+11a/2-1/2&a^2/2-3a/2-1/2\\
    -7a^2/2+17a/2-9/2&5a^2/2-a-11/2\\
    11a^2/2-15a+17/2&-5a^2/2+a+6
\end{pmatrix}=\begin{pmatrix}
0&0
\end{pmatrix}
$$


\paragraph*{\bf Galois conjugates} Note that $w$ is a rational (in fact integral) vector orthogonal to two vectors. By Galois property, it is also orthogonal to the algebraic conjugates of these vectors. Thus it is in the orthogonal of a space generated by $6$ vectors, and we are reduced to the system given by the matrix:

$$\begin{pmatrix}
    -2&3/2&-1/2&3/2&1/2&-1\\
    2&-17/2&9/2&-9/2&1/2&1\\
    -3/2&5&-2&5&1/2&-3/2\\
    -1/2&11/2&-3&-1/2&-3/2&1/2\\
    -9/2&17/2&-7/2&-11/2&-1&5/2\\
    17/2&-15&11/2&6&1&-5/2
\end{pmatrix}$$
\paragraph*{\bf Left kernel} A direct computation of the left kernel shows it is generated by the row vector $$\begin{pmatrix}
    5&6&7&4&9&6
\end{pmatrix}.$$
Thus the eigenvalues are all integers, so the additive group is $\mathbb Z$, and thus the substitution is weak mixing, proving Proposition~\ref{prop:eigenvalues-subst}.

\appendix

\section{Non-weak mixing directions in regular $n$-gon surfaces}
\label{appendix:ngons}

\paragraph*{\bf Billiards and $n$-gons} In this appendix, we compute pseudo-Anosov maps with Pisot expansion factors for the translation surfaces associated the regular $n$-gon, $n=7, 9, 14, 18, 20, 24$, with this list of polygons drawn from the work of Arnoux-Schmidt~\cite{ArnouxSchmidt}*{Theorem 2}. We give additional examples when $n=16, 30$ that were found as part of this project. This yields non-weak mixing directions in these surfaces, and can be used to produce non-minimal trajectories for the associated prism flows as well as interesting orbit closures of billiard trajectories in prisms over the associated $(\frac{\pi}{n}, \frac{(n-2)\pi}{2n}, \frac{\pi}{2})$ right triangles. We illustrate the eigenfunctions for the straight-line flow in eigendirections of $D\rho$. See Figure \ref{fig:heptagon eigenfunction} and the ancilliary files on the arXiv. Verification of statements made in this appendix are carried out in the notebook file {\texttt{appendix\_data.ipynb}}. This notebook uses \texttt{SageMath} \cite{sagemath} and \texttt{sage-flatsurf} \cite{sage-flatsurf} to carry out the calculations. The work of Winsor~\cite{Winsor} gives many further examples of \emph{special} pseudo-Anosov maps for these surfaces, which may be interesting to analyze in this fashion.

\subsection{Examples with an odd number of sides}
\label{sect:odd number of sides}
For odd $n \geq 3$, consider the two regular $n$-gons in $\C$ of side-length $1$ sharing the segment from $0$ to $1$. We form $S_n$ by identifying each edge of the first polygon to the parallel edge from the other polygon. (These are the surfaces depicted in Figure \ref{fig:heptagon eigenfunction} and in the ancillary files, but in pictures we have moved one polygon by translation to be adjacent to the other along a different edge.)
The holonomy field of this surface is $\Q(a_n)$ where $a_n = 2 \cos \frac{\pi}{n}$, and the ring of integers in this field is $\Z[a_n]$. Define the matrix
$$\displaystyle C_n = \begin{pmatrix}
1 & \tan \frac{\pi}{2n} \\
-1 & \tan \frac{\pi}{2n} \\
\end{pmatrix}.$$
Then $\omega_n = C_n \cdot S_n$ is a staircase surface built from $n-2$ rectangles and whose largest area rectangle is a $1 \times 1$ square. (This generalizes the surface $\omega_7$ of section \ref{sec:nonergodicdirections}.) The Veech group of $\omega_n$ contains the standard Hecke $(2, n, \infty)$-triangle group:
\begin{equation}
\label{eq:Veech odd}
\left\langle 
A_n = 
\begin{pmatrix}
    1 & a_n \\
    0 & 1 \\
\end{pmatrix},~
B = \begin{pmatrix}
    0 & -1 \\
    1 &  0 \\
\end{pmatrix}
\right\rangle.    
\end{equation}
Note that $-I=B^2$ is in the Veech group.

\paragraph*{\bf The heptagon.} Let $n=7$. This example was first found by Rosen and Towse \cite{RT}. Define
$$D_7 = B^{-1}A_7^{-1}BA_7 = \begin{pmatrix}
1 &     a_7 \\
a_7 & a_7^2+1 \\
\end{pmatrix},
\quad
E_7 = \frac{1}{7} \left(\begin{array}{rr}
-a^{2} - a + 2 & -a^{2} - 2 a \\
-a^{2} - a + 2 & -a^{2} + 2 a + 1
\end{array}\right).$$
Then $\omega_7$ admits a pseudo-Anosov with derivative
$D_7$ and expanding eigenvalue $\beta_7=a_7^2 + a_7$. The surface $E_7 \cdot \omega_7$ has the properties of Proposition \ref{prop:normalization:zbeta}: The group of eigenvalues in the horizontal and vertical directions is the ring of integers $\Z[a_7]$.

\begin{figure}
    \centering
    \includegraphics[width=\textwidth]{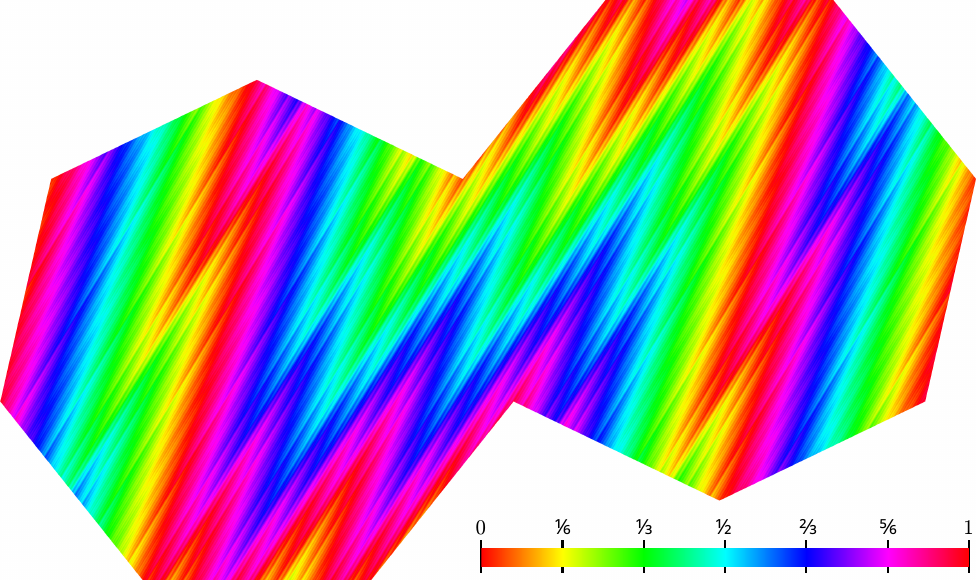}
    \caption{An eigenfunction for the straight-line flow in the direction of slope
    $$6 \, \sin\left(\tfrac{3}{7} \, \pi\right) - 2 \, \sin\left(\tfrac{2}{7} \, \pi\right) - 6 \, \sin\left(\tfrac{1}{7} \, \pi\right) \approx 1.683$$
    on the double heptagon. For more detail see Remark \ref{rem:heptagon figure}.}
    \label{fig:heptagon eigenfunction}
\end{figure}

\begin{rmk}
    \label{rem:heptagon figure}
    Figure \ref{fig:heptagon eigenfunction} depicts an eigenfunction $\psi$ on $S_7$ for the flow in the stable eigendirection of the matrix $C_7^{-1} \cdot D_7 \cdot C_7$ obtained by pulling back the eigenfunction with eigenvalue $1$ for the vertical flow on $E_7 \cdot \omega_7$. We have chosen $\psi$ to take the value zero at the singularity. As a consequence of Proposition \ref{prop:affine change}, $\psi$ has eigenvalue $ \big\|C_7^{-1} E_7^{-1} \big(\begin{smallmatrix} 0\\ 1 \end{smallmatrix}\big)\big\|^{-1}$.
\end{rmk}

\noindent A picture of the eigenfunction for the flow in the unstable eigendirection of the transformation $$C_7^{-1} \cdot D_7 \cdot C_7$$ obtained by pulling back an eigenfunction with eigenvalue $1$ for the horizontal flow on $E_7 \cdot \omega_7$ is available in the ancillary files.

\paragraph*{\bf The nonagon.} Let $n=9$. 
In \cite{HMTY}, two examples were found. The Veech group element
$$D_9 = B^{3}A_9^{3}BA_9^{-4}(BA_9)^{2} = \left(\begin{array}{rr}
-4 a_9^{2} - 1 & -9 a_9 - 4 \\
40 a_9 + 12 & 28 a_9^{2} + 12 a_9 - 1
\end{array}\right)$$
has expanding eigenvalue
$\beta_9 = 14 a_9^{2} + 26 a_9 + 7 \approx 105.31.$ Let
$$E_9 = \frac{1}{9} \left(\begin{array}{rr}
-16 a_9^{2} + 8 a_9 + 14 & -a_9^{2} - 13 a_9 - 10 \\
2 a_9^{2} + 2 a_9 - 4 & 2 a_9^{2} - a_9 - 4
\end{array}\right).$$
This conjugates the pseudo-Anosov on $\omega_9$ to a pseudo-Anosov on $E_9 \cdot \omega_9$ with diagonal derivative as in 
Proposition \ref{prop:normalization:zbeta}, and the groups of eigenfunctions for the horizontal and vertical flows on $E_9 \cdot \omega_9$ are the ring of integers $\Z[a_9]$. Pictures
of the eigenfunctions on $S_9$ obtained by pulling back the eigenfunctions with eigenvalue $1$ on $E_9 \cdot \omega_9$ in the horizontal (unstable) and vertical (stable) directions are in the ancillary files. The second hyperbolic Veech group element with Pisot expansion factor found in \cite{HMTY} has expansion factor $$\beta'_9 = 16193176 a_9^{2} + 30433216 a_9 + 8616209 \approx 1.23 \times 10^8.$$
We were not able to draw reasonable pictures of eigenfunctions in this case. 

\subsection{Examples with an even number of sides}

For even $n=2k \geq 4$, construct a regular $n$-gon in $\C$ with one edge running from $0$ to $1$ in $\C$. Gluing opposite sides of this polygon yields a translation surface $S_n$. The holonomy field of this surface is $\Q(a_k)$ where, as before, $a_k=2 \cos \left(\frac{\pi}{k}\right) = 2\cos\left(\frac{2\pi}{n}\right)$, and the ring of integers is $\Z[a_k]$. Depending on the class of $n$ modulo $4$ (that is, the parity of $k$, we define
$$C_n = 
\begin{pmatrix}
    \tan(\frac{\pi}{n}) & -1 \\
                     0 &  1
\end{pmatrix}
\quad
\text{if $n \equiv 0 \pmod{4}$, that is, $k$ even}
$$
$$C_n = 
\begin{pmatrix}
    0 & -\sec(\frac{\pi}{n}) \\
    \sin(\frac{\pi}{n}) & -\cos(\frac{\pi}{n}) \\
\end{pmatrix}
\quad 
\text{if $n \equiv 2 \pmod{4}$, that is $k$ odd.}
$$
Then the surface $\omega_n = C \cdot S_n$ has staircase presentation (partition into rectangles along saddle connections), where the dimensions of each rectangle lie in $\Q(a_k)$, and where the rectangle of largest area is a unit-area square whose vertical edges are identified, forming a cylinder of height and circumference one. The Veech group contains the $(k, \infty, \infty)$-triangle group with the following parabolic generators:
\begin{equation}
    \label{eq:Veech even}
\left\langle 
A = 
\begin{pmatrix}
    1 & 2 \\
    0 & 1 \\
\end{pmatrix},~
B_n = \begin{pmatrix}
                 1 & 0 \\
    1+\tfrac{a_k}{2} & 1 \\
\end{pmatrix}
\right\rangle.
\end{equation}

\paragraph*{\bf The tetradecagon.} Let $n=14$. This example is due to Arnoux-Schmidt \cite{ArnouxSchmidt}*{Lemma 7}. Let 
$$D_{14} = A B_{14}^{-1} A^{-1} B_{14} =
\left(\begin{array}{rr}
a_{7}^{2} + 3 a_{7} + 3 & 2 a_{7} + 4 \\
\frac{1}{2} a_{7}^{2} + 2 a_{7} + 2 & a_{7} + 3
\end{array}\right)
\quad \text{and}$$
$$E_{14} = \frac{1}{7} \left(\begin{array}{rr}
-3 a_{7}^{2} + 5 & 2 a_{7}^{2} - 2 a_{7} - 6 \\
-a_{7}^{2} - 3 a_{7} - 2 & 6 a_{7}^{2} + 2 a_{7} - 6
\end{array}\right).$$
The expanding eigenvalue of $D$ is the Pisot number $\beta_{14}=4 a_{7}^{2} + 3 a_{7} - 2 \approx 16.39.$
The surface $E_{14} \cdot \omega_{14}$ has the properties of Proposition \ref{prop:normalization:zbeta}: The group of eigenvalues in the horizontal and vertical directions is the ring of integers $\Z[a_{7}]$. Graphs of the pullbacks to $S_{14}$ of the eigenfunctions with eigenvalue $1$ on $E_{14} \cdot \omega_{14}$ in the horizontal and vertical directions are provided in the ancillary files.

\paragraph*{\bf The hexadecagon.} Let $n=16$. This example was found by searching within the Veech group. Let
\begin{align*}
D_{16} &= B_{16}^{-1}AB_{16}^{-1}(B_{16}^{-1}A)^{7}A^{2}(A^{2}B_{16}^{-1})^{2}(B_{16}^{-1}A)^{5} \\
&= \left(\begin{array}{rr}
-900 a_{8}^{3} - 1547 a_{8}^{2} + 831 a_{8} + 1141 & -772 a_{8}^{3} - 1406 a_{8}^{2} + 516 a_{8} + 890 \\
1112 a_{8}^{3} + \frac{3939}{2} a_{8}^{2} - 874 a_{8} - 1324 & 996 a_{8}^{3} + 1819 a_{8}^{2} - 639 a_{8} - 1107
\end{array}\right).
\end{align*}
Then $\beta_{16} = 184 a_{8}^{3} + 340 a_{8}^{2} - 108 a_{8} - 199
\approx 1923.07$ is the Pisot expanding eigenvalue. Let
$$E_{16} = \frac{1}{4} \left(\begin{array}{rr}
11 a_{8}^{3} + \frac{59}{2} a_{8}^{2} + \frac{35}{2} a_{1} + 1 & 7 a_{8}^{3} + 25 a_{8}^{2} + 28 a_{8} + 11 \\
\frac{23}{2} a_{8}^{3} - \frac{21}{2} a_{8}^{2} - \frac{179}{2} a_{8} - 57 & -6 a_{8}^{3} - 20 a_{8}^{2} - 18 a_{8} - 3
\end{array}\right).$$
Then $E_{16} \cdot \omega_{16}$ has the properties of Proposition \ref{prop:normalization:zbeta},
and the group of eigenvalues in the horizontal and vertical directions is the ring of integers $\Z[a_{8}]$. We again provide illustrations of the same eigenfunctions.

\paragraph*{\bf The octadecagon.} Let $n=18$. This example is due to Arnoux-Schmidt \cite{ArnouxSchmidt}*{Lemma 8}. Let
$$D_{18} = B_{18}^{-3} A^{-1} B_{18}^{4} A B_{18}^{-1} = 
\left(\begin{array}{rr}
4 a_9^{2} + 12 a_9 + 9 & -8 a_9 - 16 \\
-32 a_9^{2} - 74 a_9 - 38 & 12 a_9^{2} + 52 a_9 + 57
\end{array}\right).$$
The expanding Pisot eigenvalue is $\beta_{18}=32 a_9^{2} + 60 a_9 + 17 \approx
242.79$. The group of eigenvalues in both the horizontal and vertical directions of $E_{18} \cdot \omega_{18}$ is the ring of integers $\Z[a_9]$ when
$$E_{18} = \frac{1}{9} \left(\begin{array}{rr}
-7 a_9^{2} - 7 a_9 + 5 & 2 a_9^{2} + 6 a_9 + 4 \\
-4 a_9^{2} - a_9 + 8 & 8 a_9^{2} - 4 a_9 - 22
\end{array}\right).$$
Plots of the same eigenfunctions are provided.

\paragraph*{\bf The icosagon.} Let $n=20$. This example was found by searching within the Veech group and may be related to the example in \cite{ArnouxSchmidt}*{Lemma 9}. Let 
\begin{align*}
D_{20} &= (B_{20}^{-1}A)^{7}A^{6}B_{20}^{-1}A^{4} \\
&= \left(\begin{array}{rr}
-18 a_{10}^{3} - 31 a_{10}^{2} + 29 a_{10} + 43 & -132 a_{10}^{3} - 234 a_{10}^{2} + 208 a_{10} + 330 \\
12 a_{10}^{3} + \frac{49}{2} a_{10}^{2} - 11 a_{10} - 27 & 90 a_{10}^{3} + 183 a_{10}^{2} - 83 a_{10} - 203
\end{array}\right).
\end{align*}
The expanding Pisot eigenvalue is $\beta_{20} = 92 a_{10}^{3} + 175 a_{10}^{2} - 127 a_{10} - 242 \approx
782.72$. The matrix
$$E_{20}=\frac{1}{5} \left(\begin{array}{rr}
\frac{1}{2} a_{10}^{3} + \frac{1}{2} a_{10}^{2} - 2 a_{10} - \frac{5}{2} & -2 a_{10}^{3} - a_{10}^{2} + 5 a_{10} \\
\frac{49}{2} a_{10}^{3} + 49 a_{10}^{2} - 27 a_{10} - 60 & 8 a_{10}^{3} + 34 a_{10}^{2} + 39 a_{10} + 10
\end{array}\right)$$
sends $\omega_{20}$ to a surface whose eigenvalues for the horizontal and vertical flows are $\Z[a_{10}]$. Plots are provided as before.

\paragraph*{\bf The icositetragon.} Let $n=24$. This example is due to Arnoux and Schmidt \cite{ArnouxSchmidt}*{Lemma 10}.
\begin{align*}
D_{24} &= (B_{24}^{-1}A)^{2}B_{24}^{-6}A^{-1}B_{24}A^{-5}B_{24}^{-1}A(B_{24}A^{-1})^{2}B_{24}^{6}AB_{24}^{-1}A^{5}B_{24}A^{-1}
\\
&= \left(\begin{smallmatrix}
1384810 a_{12}^{3} + 2678580 a_{12}^{2} - 362410 a_{12} - 712559 & -2102180 a_{12}^{3} - 4065520 a_{12}^{2} + 552020 a_{12} + 1082920 \\
-1667820 a_{12}^{3} - 3220180 a_{12}^{2} + 450780 a_{12} + 864120 & 2530790 a_{12}^{3} + 4887020 a_{12}^{2} - 682790 a_{12} - 1311239
\end{smallmatrix}\right).
\end{align*}
The expanding Pisot eigenvalue is 
$$\beta_{24} = 3916750 a_{12}^{3} + 7566580 a_{12}^{2} - 1049490 a_{12} - 2027459 \approx 5.24 \times 10^7.$$ 
The matrix
$$E_{24} = 
\frac{1}{6} \left(\begin{array}{rr}
-\frac{1}{2} a_{12}^{3} - \frac{3}{2} a_{12}^{2} - \frac{1}{2} a_{12} + \frac{3}{2} & 2 a_{12}^{3} + 2 a_{12}^{2} - 4 a_{12} - 1 \\
\frac{115}{2} a_{12}^{3} + \frac{3}{2} a_{12}^{2} - \frac{383}{2} a_{12} + \frac{75}{2} & -44 a_{12}^{3} + 15 a_{12}^{2} + 187 a_{12} - 27
\end{array}\right)
$$ sends $\omega_{24}$ to a surface whose flows in the horizontal and vertical directions have eigenvalue groups $\Z[a_{12}]$. We were able to plot the eigenfunction for the flow in the unstable direction obtained by pulling back the eigenfunction with eigenvalue $1$ in the horizontal direction on $E_{24} \cdot \omega_{24}$; see the ancillary files. We were not able to plot eigenfunctions in the stable direction.

\paragraph*{\bf The triacontagon.} Let $n=30$. This example was found by searching through the Veech group. Let 
\begin{align*}
D_{30} &= (AB_{30}^{-1})^{10}A^{17}B_{30}^{-1}A(B_{30}^{-1}AB_{30}^{-1})^{2} \\
&= \left(\begin{array}{rr}
401 a_{15}^{3} + 1187 a^{2} + 706 a - 223 & -288 a^{3} - 866 a^{2} - 516 a + 190 \\
\frac{867}{2} a^{3} + 1268 a^{2} + 730 a - \frac{513}{2} & -321 a^{3} - 931 a^{2} - 498 a + 241
\end{array}\right).
\end{align*}
The expanding Pisot eigenvalue is 
$\beta_{30} = 92 a_{15}^{3} + 272 a_{15}^{2} + 164 a_{15} - 47 \approx 2003.60.$ Let
$$E_{30} = \frac{1}{15} \left(\begin{array}{rr}
-5 a_{15}^{3} + 9 a_{15}^{2} + 49 a_{15} + 32 & 2 a_{15}^{3} - 8 a_{15}^{2} - 32 a_{15} - 12 \\
23 a_{15}^{3} + 82 a_{15}^{2} + 76 a_{15} + 9 & -14 a_{15}^{3} - 70 a_{15}^{2} - 100 a_{15} - 36
\end{array}\right).
$$
Then $E_{30} \cdot \omega_{30}$ has eigenfunctions in the horizontal and vertical directions with eigenvalues given by $\Z[a_{15}]$. Plots are provided.

\section{Approximating eigenfunctions}
\label{sect:computing eigenfunctions}
\paragraph*{\bf Computing eigenfunctions} The results of section \ref{sect:Pisot} allow us to explicitly compute the group of eigenvalues $E(\rho)$ for the straight-line flow on a translation surface $S$ in the unstable eigendirection of $D \rho$. Let $c \in E(\rho)$ be an eigenvalue, let $u \in \R^2$ be a unit unstable eigenvector, and let $\phi_t:S \to S$ be the unit speed straight-line flow in direction $u$. 

\paragraph*{\bf Basepoint} Let $p_0 \in S$ be a preferred point. Consider the eigenfunction $\Psi_c:S \to \R/\Z$ satisfying
$$\Psi_c(p_0)=0 \quad \text{and} \quad \Psi_c(\phi_t x) = \Psi_c(x) + ct.$$
(The other eigenfunctions with eigenvalue $c$ can be obtained by postcomposing $\Psi_c$ with a rotation of $\R/\Z$.) Combining these equations, we see that
\begin{equation}
\label{eq:exact values}
\Psi_c(\phi_t p_0)=ct \quad \text{for all $t \in \R$}.    
\end{equation}
In the special case when $p_0$ is a singular point, this equation holds on each separatrix in the directions $\pm u$. Since the flow $\phi_t$ is minimal, \eqref{eq:exact values} determines the value of $\Psi_c$ exactly on a dense set. In practice, we can compute a long finite-length straight-line trajectory (or trajectories if $p_0$ is singular) through $p_0$ and know the value on this segment (or segments) exactly.

\paragraph*{\bf Pixelated images} To compute a pixelated image of an eigenfunction, we select a presentation of $S$ in the plane as a disjoint union $U$ of polygons with edge identifications. Then we choose an $\epsilon>0$ and cut the plane into $\epsilon \times \epsilon$ squares whose vertices lie in the grid $\epsilon \cdot \Z^2$. We consider the collection $\mathcal Q$ of those squares $Q$ whose center point lies in $U$. Given our choice of $p_0$, we compute a straight-line trajectory that is so long that it intersects each such square $Q$. Then for each $Q \in \mathcal Q$, we can select a point $x_Q \in Q$ for which we know $\Psi_c(x_Q)$ exactly using \eqref{eq:exact values}. To create our picture, we create an image where each pixel represents a $Q \in \mathcal Q$ and color the pixel according to the value of $\Psi_c(x_Q)$. Thus our pictures have the property that for each pixel in the image with color given by $v \in \R/\Z$, the corresponding square in the surface contains a point $x$ where $\Psi_c(x)=v$.

\paragraph*{\bf Flatsurf} We used sage-flatsurf \cite{sage-flatsurf} to compute our pictures of eigenfunctions (Figure \ref{fig:heptagon eigenfunction} and in the ancillary files on the arXiv) using the ideas above.

\section{IETs semiconjugate to toral rotations: An example from the heptagon}
\label{sect: example semiconjugacy}

\paragraph*{\bf Semiconjugacies} The famous Arnoux-Yoccoz~\cite{ArnouxBSMF88} IET is known to be measurably isomorphic to a (two-dimensional) toral translation by results of Arnoux-Cassaigne-Ferenczi-Hubert~\cite{ArnouxCassaigneFerencziHubert}*{Corollary 2}. We will work out
some details of a similar semiconjugacy related to the flow on the double heptagon, $S_7$, in the \emph{stable} direction of the pseudo-Ansosov described for $S_7$ in section \ref{sect:odd number of sides}. We believe that this semiconjugacy is a measurable isomorphism, but we do not prove it here.

\paragraph*{\bf Stable versus unstable directions} In \S\ref{sec:nonergodicdirections}, we consider the flow in the \emph{unstable} direction for a pseudo-Anosov on $S_7$, and the associated first return map to a transversal given by a piece of a leaf in the \emph{stable} direction, and show (Proposition~\ref{prop:eigenvalues-subst}) that the associated IET is weakly mixing. In this section we are considering the flow in the stable direction, mostly for ease of computation, and constructing a transversal for which the associated IET is semiconjugate to a toral translation, and thus not weak mixing. This emphasizes the important and subtle role of time changes in weak mixing properties of transformations and flows.

\paragraph*{\bf Level sets} Let $\Psi_c:S_7 \to \R/\Z$ denote the eigenfunction depicted in Figure \ref{fig:heptagon eigenfunction} and discussed in Remark \ref{rem:heptagon figure}. Let $p_0 \in S_7$ be the singularity. Then we have $\Psi_c(p_0)=0$. Let $L=\Psi_c^{-1}(\{0\}) \subset S_7$ be the level set of zero.

\begin{figure}[b]
    \centering
    \includegraphics[width=0.8\linewidth]{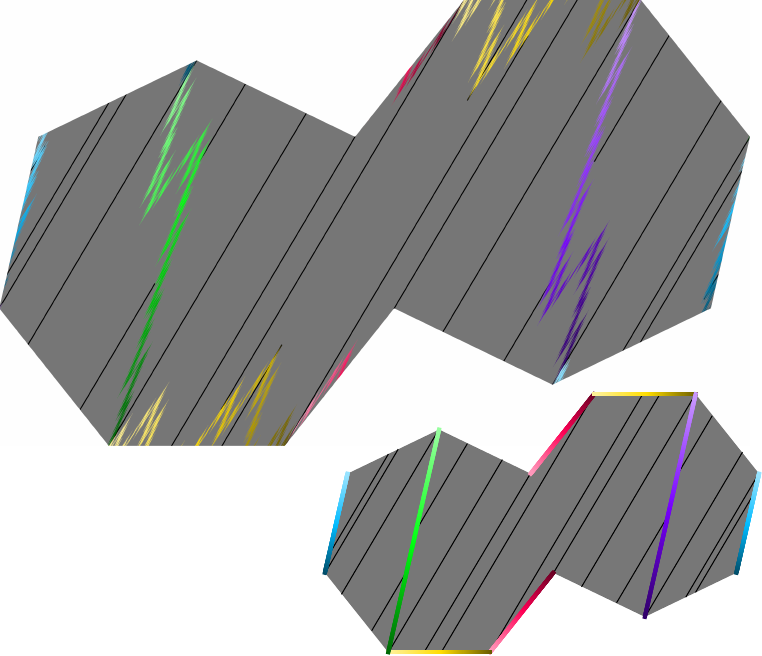}
    \caption{Top depicts the level set $L=\Psi_c^{-1}(\{0\})$, drawn in color. The bottom depicts $U$, the union of five saddle connections obtained by straightening $L$. Colors of $U$ chosen so that hue remains constant across the saddle connection and lightness varies linearly (in the HSL color space), and these colors were pushed forward to $L$ via the conjugacy of the return maps. The black lines are initial segments of separatrices in direction $\pm u$ until the first return to $L$ or $U$.}
    \label{fig:level set straightening}
\end{figure}

\paragraph*{\bf The stable direction} Let $u \approx (-0.51, -0.86)$ be the unit vector obtained by scaling $C_7^{-1} E_7^{-1}\big(\begin{smallmatrix} 0\\ 1 \end{smallmatrix}\big)$ by a positive constant.  This is a stable eigendirection for the derivative of the pseudo-Anosov on $S_7$. Let $\theta (\approx 239.27^\circ)$ be the direction of $u$ and let $\phi^\theta_t:S_7 \to S_7$ denote the unit speed flow in this direction. We define $U$ to be the union of five saddle connections depicted at the bottom of Figure \ref{fig:level set straightening}. 
We will show:

\begin{Prop}
\label{prop: conjugate L and U}
    The return maps of the flow $\phi^\theta_t$ to $U$ and to $L$ are topologically conjugate. Concretely, there is a homeomorphism $H_1:S_7 \to S_7$ that  preserves all $\phi^\theta_t$-trajectories including their orientation, and satisfies $H_1(U)=L$.
\end{Prop}

\noindent We will prove the second statement of the Proposition which implies the first.

\paragraph*{\bf Approximating the level set} An approximation
of the level set $L$ is depicted in the top part of Figure \ref{fig:level set straightening}. The colors used for $L$ correspond to the colors used for $U$ under the conjugacy. The proposition verifies the naive intuition that the arcs of $L$ are as depicted and they can be straightened to the saddle connections making up $U$ by moving points of $L$ within leaves of the straight-line flow.

\paragraph*{\bf Orienting} We orient the level set $L$ and the saddle connections making up $U$ so that as we move along the curves in the direction of orientation, the trajectories of $\phi^\theta_t$ cross from left to right. In the picture, the orientation of the saddle connections is from lighter to darker colors.

\begin{lemma}
\label{lem:homologous}
    The level set $L$ and the the union of saddle connections $U$ are homologous.
\end{lemma}
\begin{proof}
Let $\rho:S_7 \to S_7$ be the pseudo-Anosov. Theorem \ref{thm:asymptotic} tells us that the cohomology class $[\Psi_c]$ is the integer class to which $c \eta_s$ is forward asymptotic to under $\rho^{-n}$ as $n$ tends to $+\infty$. (We are considering the stable direction of $\rho$.) In the notebook \texttt{appendix\_data.ipynb} included in ancillary files on the arXiv, we compute this cohomology class. Let $e_i$ denote the $i$-th edge, $i \in \{0, \ldots, 6\}$, of the leftmost heptagon in $S_7$, as depicted in Figure \ref{fig:level set straightening}, with the horizontal edge being $e_0$ and listed in counterclockwise order. We orient these edges in counterclockwise around the boundary of this heptagon. Observe that $\{e_0, \ldots, e_6\}$ generate $H_1(S_7; \Z).$
The notebook computes that 
\begin{equation}
    \label{eq:cohomology class}
    \left([\Psi_c](e_0), \ldots, [\Psi_c](e_6)\right) = (0, 0, -1, 1, 1, 0, -1).
\end{equation}
With our orientation convention, we have $[\Psi_c](e_i) = [e_i] \cap [L]$ for all $i$, where $[L]$ and $[e_i]$ denote homology classes. The reader may check that $[e_i] \cap [U]$ also gives values matching \eqref{eq:cohomology class}.
\end{proof}

\paragraph*{\bf Cohomological functions and isotopy}
It is a classical result of Fried \cite[Theorem C]{Fried82} that homologous sections of a flow yield topologically conjugate return maps. Moveover, the conjugacy can be obtained by an isotopy which only moves points along leaves of the flow. In the setting of translation surfaces, this can't quite be true because of the singularities. We will state and prove a variant of this result which does hold in our setting. Our argument is very close to Fried's~\cite{Fried}.

\paragraph*{\bf Increasing along trajectories} Rather than working with sections, we work with continuous functions to the circle. If $S$ is a translation surface with straight-line flow $\phi_t$, we say a continuous $f:S \to \R/\Z$ is {\em strictly increasing along trajectories} if for any closed interval $J \subset \R$ containing zero and any trajectory $\gamma:J \to S$ given by $\gamma(t)=\phi_t(p_0)$, the lift of $f \circ \gamma:J \to \R/\Z$ to a function $J \to \R$ is strictly increasing. We allow the trajectory $\gamma$ to take singular values at endpoints of $J$ (meaning that $\gamma$ could be a separatrix or saddle connection).

\paragraph*{\bf Cohomology classes} In section \ref{sect:Pisot}, we discussed the cohomology class of a map $f:S \to \R/\Z$. Here we consider the relative cohomology class $[f]$, which interpret as a linear $H_1(S, \Sigma; \Z) \to \R$, computed as before. This function will take integral values on absolute homology classes.

\begin{lemma}    
\label{lem:Fried}
Let $f, g: S \to \R/\Z$ be two continuous functions on a translation surface $S$ that are strictly increasing along trajectories of the vertical straight-line flow $\phi_t$. Suppose in addition that the relative cohomology classes of $f$ and $g$ are equal, and that the functions agree at the singularities. Then there is an ambient isotopy $H_s:S \to S$ with $s \in [0, 1]$, such that the following statements hold:
\begin{itemize}
    \item Each singularity is fixed by each $H_s$.
    \item $H_0$ is the identity.
    \item $f=g\circ H_1$.
    \item Each $H_s$ acts as an orientation-preserving homeomorphism of each vertical leaf.
\end{itemize}
\end{lemma}

\paragraph*{\bf Conjugate return maps} In particular, this Lemma tells us that return maps of the flow $\phi_t$
to corresponding level sets $f^{-1}(\{\theta\})$ and $g^{-1}(\{\theta\})$ are conjugate under the restriction of $H_1$ to $f^{-1}(\{\theta\})$.

\begin{proof}
We would like to work with functions to $\R$ rather than to the circle. To this end, we construct a cover $\tilde S \to S$ and lifts $\tilde f, \tilde g:\tilde S \to \R$ such that the following diagram commutes:
$$\begin{CD}
\tilde S @>\tilde f,~\tilde g>> \R\\
@VVV @VVV\\
S @>f,~g>> \R/\Z.
\end{CD}
$$
The cover $\tilde S$ is the cover associated to the common absolute cohomology classes of $f$ and $g$. That is, a loop $\gamma$ on $S$ lifts to $\tilde S$ if and only if $[f](\gamma)=0$. Because $[f]$ comes from a map to a circle, the absolute cohomology class is an integer class, and so $\tilde S$ is a $\Z$-cover: the cover is normal and has deck group isomorphic to $\Z$. Because $f$ and $g$ coincide on singularities and the relative classes agree, we can assume (by precomposing $\tilde g$ with a deck transformation) that $\tilde f$ and $\tilde g$ coincide on the singularities of $\tilde S$.

Let $\tilde \phi_t:\tilde S \to \tilde S$ be the lifted vertical straight-line flow. 
Since $S$ is compact and $\tilde f$ is strictly increasing along trajectories, we see that for any $\epsilon>0$, there is a $c>0$ such that for any $p \in \tilde S$,
$$\tilde f \circ \tilde \phi_\epsilon(p)-\tilde f(p)>c$$
and likewise for $\tilde g$.
(The quantity $\tilde f \circ \tilde \phi_\epsilon(p)-\tilde f(p)$ is deck group-invariant, and while not well defined at $p$ such that the trajectory hits a singularity before time $\epsilon$, extends to a well-defined positive continuous function on a compact set via the standard trick of splitting trajectories hitting singularities, and so is bounded from below by a positive number.) From this we conclude that the restriction of $\tilde f$ and $\tilde g$ to a bi-infinite trajectory $\{\tilde \phi_t x:~t \in \R\}$ is an orientation-preserving homeomorphism to $\R$, while the restriction to a separatrix is a homeomorphism to an infinite closed interval with one endpoint, and the restriction to a saddle connection is a homeomorphism to a closed and bounded interval.

Because $\tilde f$ and $\tilde g$ agree on the singularities, we see that the images under $\tilde f$ and $\tilde g$ of a bi-infinite trajectory, separatrix, or saddle connection are the same. We can use this to define a map $\tilde H_1:\tilde S \to \tilde S$ that carries $\tilde f$ to $\tilde g$. We define $\tilde H_1$ to fix singular points. Any non-singular point $p$ in $\tilde S$ can be extended to a bi-infinite trajectory, separatrix, or saddle connection $\gamma(t)$ with $\gamma(0)=p$. Then there is a unique $p'$ on $\gamma$ such that $\tilde f(p) = \tilde g(p')$ and we define $\tilde H_1(p)=p'$.

We extend our definition of $\tilde H_1$ to a family of maps $\tilde H_s:\tilde S \to \tilde S$ for $s \in [0,1]$. First define $\tilde h_s=(1-s)\tilde f + s \tilde g$. Then $\tilde h_s:\tilde S \to \R$ is strictly increasing along vertical trajectories, and agrees with $\tilde f$ and $\tilde g$ on singularities. Observe that the image of any bi-infinite trajectory, separatrix, or saddle connection under each $\tilde h_s$ is the same.
We define $\tilde H_s$ as in the previous paragraph, being the identity on singular points, and satisfying $\tilde H_s(p)=p'$ where $p'$ is the point on the trajectory through $p$ such that $\tilde f(p)=\tilde h_s(p')$ whenever $p$ is non-singular.

We claim that $\tilde H_s$ is continuous when viewed as a function $[0,1]\times \tilde S \to \tilde S$. To this end suppose that $\{s_n\}$ is a sequence in $[0,1]$ that converges to $s_\ast$ and $\{p_n\}$ is a sequence in $\tilde S$ that converges to $p_\ast$. Set $q_n=\tilde H_{s_n}(p_n)$ and $q_\ast=\tilde H_s(p_\ast)$. We claim that $q_n \to q_\ast$. 
First focus on the case when $q_\ast$ is non-singular. 
Let $U$ be a neighborhood of $q_\ast$ in $\tilde S$ that contains no singular points of $\tilde S$.
Let $\tau \subset U$ be a closed horizontal segment that contains $q$ in its interior. Since $\tau$ is compact, there is an $\epsilon>0$ such that the box
\begin{equation}
    \label{eq:flowbox}
    B= \bigcup_{t \in [-\epsilon, \epsilon]} \tilde \phi_t(\tau)
\end{equation}
is contained in $U$. Since $\tilde h_{s_\ast}$ is strictly increasing along trajectories
$$\tilde h_{s_\ast}(\phi_{-\epsilon} q_\ast) < \tilde h_{s_\ast}(q_\ast) < \tilde h_{s_\ast}(\phi_{\epsilon} q_\ast).$$
By continuity of the map $(s, q) \mapsto \tilde h_s(q)$, we can shrink $\tau$ and choose a closed interval neighborhood of $s_\ast$, $I \subset [0,1]$, to arrange that
\begin{align*}
    M_- < \tilde h_{s_\ast}(q_\ast) < M_+ \quad \text{where} \quad
    M_- &= \sup\, \{\tilde h_s( \tilde \phi_{-\epsilon} q):~(s,q) \in I \times \tau\} \\
    \text{and} \quad M_+ &= \inf\,\{\tilde h_s(\tilde \phi_{\epsilon} q):~(s,q) \in I \times \tau\}.
\end{align*}
With $B$ defined as in \eqref{eq:flowbox} but with the updated values $\tau$, we see the neighborhood $B$ of $q_\ast$ has the property that if $s \in I$ and $q \in \tilde S$ satisfies both:
\begin{itemize}
    \item $\tilde h_s(q) \in (M_-, M_+)$, and
    \item the $\tilde \phi_t$-trajectory of $q$ passes through $B$,
\end{itemize}
then $q \in B \subset U$. (This just uses monotonicity of $\tilde h_s$ along trajectories.) Now return to our sequence. Since $I$ is a neighborhood of $s_\ast$, for $n$ sufficiently large, we have $s_n \in I$.
Since $q_\ast$ lies on the same trajectory as $p_\ast$, and $p_n \to p_\ast$, we see that when $n$ is sufficiently large, the trajectory through $p_n$ must pass through $B$. Since $q_n$ lies on the same trajectory as $p_n$, this also holds for $q_n$. Finally observe that
$$\lim \tilde h_{s_n}(q_n)=\lim \tilde f(p_n)=\tilde f(p)=\tilde h_{s_\ast}(q),$$
so when $n$ is sufficiently large we have $\tilde h_{s_n}(q_n) \in (M_-, M_+)$. Therefore the observation above tells us that when $n$ is sufficiently large, $q_n \in B \subset U$. Thus we've shown that $q_n \to q_\ast$ as desired.

The above argument also works when $q_\ast$ is singular: we need to replace 
$U$ by a neighborhood that only contains the singularity $q_\ast$, and replace $\tau$ by a union of initial parts of all horizontal segments emanating from $q$.

We will now show that when $s$ is fixed, $\tilde H_s:\tilde S \to \tilde S$
is a homeomorphism. The above argument shows that each $\tilde H_s$ is continuous. We can deduce that $\tilde H_s$ has a continuous inverse using the same ideas. We produced $\tilde H_s$ to interpolate between $\tilde f$ and $\tilde g$. Fixing $s$, we can produce a new family of functions $\tilde H'_{s'}$ in the same way which interpolates between $\tilde h_s$ and $\tilde f$. That is, 
$\tilde h'_{s'} \circ \tilde H'_{s'} = \tilde h_s$
for each $s' \in [0,1]$ where $h'_{s'}=(1-s') \tilde h_s + s' \tilde f$.
Then each $\tilde H'_{s'}$ is continuous by the reasoning above. Now observe that $\tilde f \circ \tilde H'_{1} = \tilde h_s.$ By the definition of the maps, we see that $\tilde H'_1$ is the inverse of $\tilde H_s$, showing that $\tilde H_s$ has a continuous inverse as desired.

We have shown that $\tilde H_s$ is an ambient isotopy.
Observe that each $\tilde H_{s}$ is deck group-equivariant, and so descends to a homeomorphism $H_{s}:S \to S$ that fixes singularities. We observe that $H_s$ satisfies the required statements.
\end{proof}

We need some basic observations about the complement $S_7 \setminus U.$

\begin{lemma}
    \label{lem:transversality}
    No saddle connection in $U$ is parallel to the flow direction $\phi_t^\theta$. The complement $S_7 \setminus U$ is a pair of annuli.
    Initial segments of separatrices in direction $\theta$, up to the first intersection with $U$ join opposite sides of the annulus containing them, and cut each annulus into trapezoids.
\end{lemma}

This situation is depicted in the bottom of figure \ref{fig:level set straightening}. Checking this is a calculation we omit. The fact that the annuli are cut into trapezoids can be deduced by observing that any initial part of a separatrix makes an angle of less than $\pi$ on each side with the boundary of the annulus containing it.

\begin{proof}[Proof of Proposition \ref{prop: conjugate L and U}]
    We need two functions to apply Lemma \ref{lem:Fried}. We use $\Psi_c$ which has $L$ as the level set of zero. We will also define a function $g:S_7 \to \R/\Z$ such that $g^{-1}(\{0\})=U$. We define $g$ to be zero on $U$, now let $\overline{pq}$ be a maximal segment of the flow whose interior is in the complement of $U$. Since it is a segment of the flow, it has a unit speed parameterization $\gamma: [a,b] \to \overline{pq}$. We define $g$ on $\overline{pq}$ by $g \circ \gamma(s)=\frac{s-a}{b-a} + \Z$; this map can be seen to be the inverse of $\gamma$ followed by an affine rescaling $[a,b] \to [0,1]$ and the covering to $\R \to \R/\Z$. From the partition into trapezoids coming from Lemma \ref{lem:transversality}, we see that $g$ is continuous; in fact, $g$ is affine-linear on each trapezoid and so continuity comes from the fact that the maps on each trapezoid agree on their boundaries.

    By construction, $U=g^{-1}(0)$ and $g$ is strictly increasing along flow lines. The absolute cohomology classes of $\Psi_c$ and $g$ are equal by Lemma \ref{lem:homologous}, because they are Poincaré dual to the homology classes of $L$ and $U$. Since $S_7$ only has one singularity, their relative classes coincide as well, and by construction both $\Psi_c$ and $g$ are zero at the singularity. Thus Lemma \ref{lem:Fried} gives us an an ambient isotopy $H_s$ that preserves flowlines and satisfies $\Psi_c \circ H_1=g$. Thus $H_1(U)=L$ and $H_1$ conjugates the first return maps.
\end{proof}

\begin{rmk}[Generality of the approach]
    Let $S$ be a translation surface and let $\psi:S \to \R/\Z$ be an eigenfunction for the straight-line flow. Call a level set {\em critical} if it contains a singularity. Let $L_1, \ldots, L_n$ be the
    critical level sets, written in increasing cyclic order. The absolute homology class $[L_i]$ is Poincar\'e dual to the class $[\psi]$ and we assume this class is known (via methods in this paper for example). 
    For each $L_i$ choose a graph $U_i$ whose vertices are the singularities in $L_i$, and whose edges are smooth curves transverse to the flow. We insist that each $[U_i]$ be Poincar\'e dual to $[\psi]$ and that the $U_i$ are pairwise disjoint.
    Let $L=\bigcup L_i$ and $U=\bigcup U_i$. We require that the complement $S \setminus U$ is a collection of annuli, and that trajectories enter an annulus through one boundary component and exit through the other, inducing a homomorphism between the components. We futher insist that for each annulus there is an index $i$ so that the boundary component through which trajectories enter comes from points in $U_i$ and the other boundary component comes from points in $U_{i+1}$.
    Let $g:S \to \R/\Z$ be the function which assigns to each point $U_i$ the value in $\psi(L_i)$ and which affine-linearly interpolates between boundary values on segments of the flow passing through annuli. 
    We have ensured that $[\psi]=[g]$ as absolute cohomology classes and
    that the functions $\psi$ and $g$ agree on the singularities. 
    To apply Lemma \ref{lem:Fried}, we need to know that $[\psi]=[g]$ as relative classes, which we get for free in our context because we only have one singularity. Then Lemma \ref{lem:Fried} gives a conjugacy between the return maps to $L$ and to $U$.
\end{rmk}

\paragraph*{\bf A conjugate IET}
The return map to $L$ is given by $\phi^\theta_{1/c}|_L:L \to L$. Let $r:U \to U$ denote the first return map of the flow to $U$. Note that both the sections $L$ and $U$ have natural invariant measures coming from the Lebesgue-transverse measure to the flow $\phi^\theta_t$.

\paragraph*{\bf Normalizing} After normalizing to ensure the total transverse measure of the saddle connections in $U$ is one, we compute the transverse measures of the saddle connections to be given by:

\begin{center}
    \begin{tabular}{cc}
        Color & Normalized transverse measure \\
        \hline
        red & $\frac{1}{13}(-16 a_7^{2} + 7 a_7 + 40) \approx 0.0509$\\
        yellow & $\frac{1}{13}(21 a_7^{2} - 10 a_7 - 46) \approx 0.3206$\\
        green, purple & $\frac{1}{13}(-10 a_7^{2} + 6 a_7 + 25) \approx 0.2571$\\
        blue & $\frac{1}{13}(15 a_7^{2} - 9 a_7 - 31) \approx 0.1144$\\
    \end{tabular}
\end{center}

\paragraph*{\bf Conjugating} Choose a measure-preserving map $h:U \to [0,1)$ that continuously carries the interior of each saddle connection to a subinterval of length equal to the normalized transverse measure of the saddle connection, with the saddle connections images being disjoint. Then $h$ conjugates $r:U \to U$ to an IET $\tau:[0,1) \to [0,1)$. Specifically, we arrange the intervals in decreasing rainbow order. This gives the IET depicted in Figure \ref{fig:iet}, whose interval lengths are:
$$\begin{array}{ll}
|A| = \tfrac{1}{13}(-10 a_7^{2} + 6 a_7 + 25) \approx 0.2571 & 
|B|= \tfrac{1}{13}(24 a_7^2 - 17 a_7 - 47) \approx 0.0227, \\
|C| =  \tfrac{1}{13}(-19 a_7^2 + 14 a_7 + 41) \approx 0.3488, &
|D| =  \tfrac{1}{13}(6 a_7^2 - a_7 - 15) \approx 0.2061, \\
|E| =  \tfrac{1}{13}(15 a_7^2 - 9 a_7 - 31) \approx 0.1144, &
|F| =  \tfrac{1}{13}(-16 a_7^2 + 7 a_7 + 40) \approx 0.0509.
\end{array}$$

\begin{figure}
    \centering
    \includegraphics[width=0.8\linewidth]{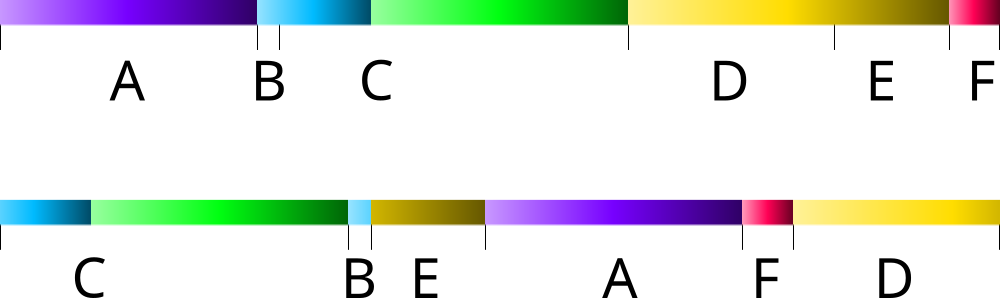}
    \caption{The IET $\tau$ carries the top arrangement to the bottom arrangement by piecewise translation. Colors come from $U$ via $h:U \to [0, 1)$ which conjugates $r:U \to U$ to $\tau$.}
    \label{fig:iet}
\end{figure}

\paragraph*{\bf Eigenvalues} From work in section \ref{sect:odd number of sides}, we know that the entries of $\mathbf c = (c, \frac{-c}{a_7}, -a_7 c)$ generates the group of eigenvalues in the stable direction on $S_7$. (These eigenvalues come from the triple $(1, \frac{-1}{a_7}, -a_7)$ of eigenvalues on $E_7 \cdot \omega_7$, and they generate the full group $\Z[a_7]$ of eigenvalues on $E_7 \cdot \omega_7$.) Let 
$$\bPsi_{\mathbf c} = \Psi_c \times \Psi_{-c/a_7} \times \Psi_{-a_7 c}:S_7 \to \TT^3$$ 
be the semiconjugacy defined as in Theorem \ref{theorem:pisotsemiconjugacy}. 
The map $\bPsi_{\mathbf c}$ is surjective and the preimage of $\{0\} \times \TT^2$ is the level set $L$ above. This means that the restriction $\bPsi_{\mathbf c}|_L:L \to \{0\} \times \TT^2$ is a surjection from $L$ to the $2$-torus. 
We have drawn an approximaton of this map $L \to \TT^2$ in Figure \ref{fig:torus}. From Theorem \ref{theorem:pisotsemiconjugacy}, we see that the first return map $\phi^\theta_{1/c}|_L:L \to L$ is semiconjugated to the translation $R:\TT^2 \to \TT^2$ of this $2$-torus by the vector $(\frac{-1}{a_7}, -a_7) \approx (0.4450, 0.1981)$
modulo $\Z^2$. Using the conjugacy between $\phi^\theta_{1/c}|_L$ and the IET $\tau$, we see that:

\begin{Cor}
    The (multiplicative) group of eigenvalues of $\tau$ is $$\{e^{2 \pi x i}:~x \in \Z[a_7]\}.$$ All eigenfunctions are continuous.
\end{Cor}

\paragraph*{\bf The Cantor topology} We remark that all the proof shows that all eigenfunctions continuously factor through $L$, which is topologically a bouquet graph with $5$ edges. This is a stronger condition than continuity in the Cantor topology.

\begin{proof}
    Since $\tau$ is semi-conjugate to a rotation of $\TT^2$ by $(\frac{-1}{a_7}, -a_7)$, postcomposing this map with projection to $\R/\Z$ gives (additive) eigenfunctions with (multiplicative) eigenvalues $e^\frac{-2 \pi i}{a_7}$ and $e^{-2 \pi i a_7}$, respectively. 
    These two generate the claimed group. Conversely, suppose that $e^{2 \pi i x}$ is a (multiplicative) eigenvalue for $\tau$. Then this is also an eigenvalue for $\phi^\theta_{1/c}|_L:L \to L$. Let $f:L \to \R/\Z$ be an associated (additive) measurable eigenfunction. Then we have 
    $$f \circ \phi^\theta_{1/c}(p) = x + f(p) \quad \text{for $p \in L$}.$$
    Extend $f$ to all of $S_7$ by defining
    $$\bar f\big(\phi^\theta_{t}(p)\big)=c x t + f(p) \quad \text{for $p \in L$ and $t \in [0, \frac{1}{c})$.}$$
    Observing that $\bar f$ is an eigenfunction for $\phi^\theta_t$ with additive eigenvalue $c x$, we conclude that $\bar f$ is continuous (from Solomyak's results \cite{Solomyak97ETDS}) and that $c x \in c \Z[a_7]$ from the classification of eigenvalues for $\phi^\theta_t$ obtained from appendix \ref{sect:odd number of sides} and Proposition \ref{prop:affine change}. We conclude that $x \in \Z[a_7]$.
\end{proof}

\begin{figure}
    \centering
    \includegraphics[width=0.9\linewidth]{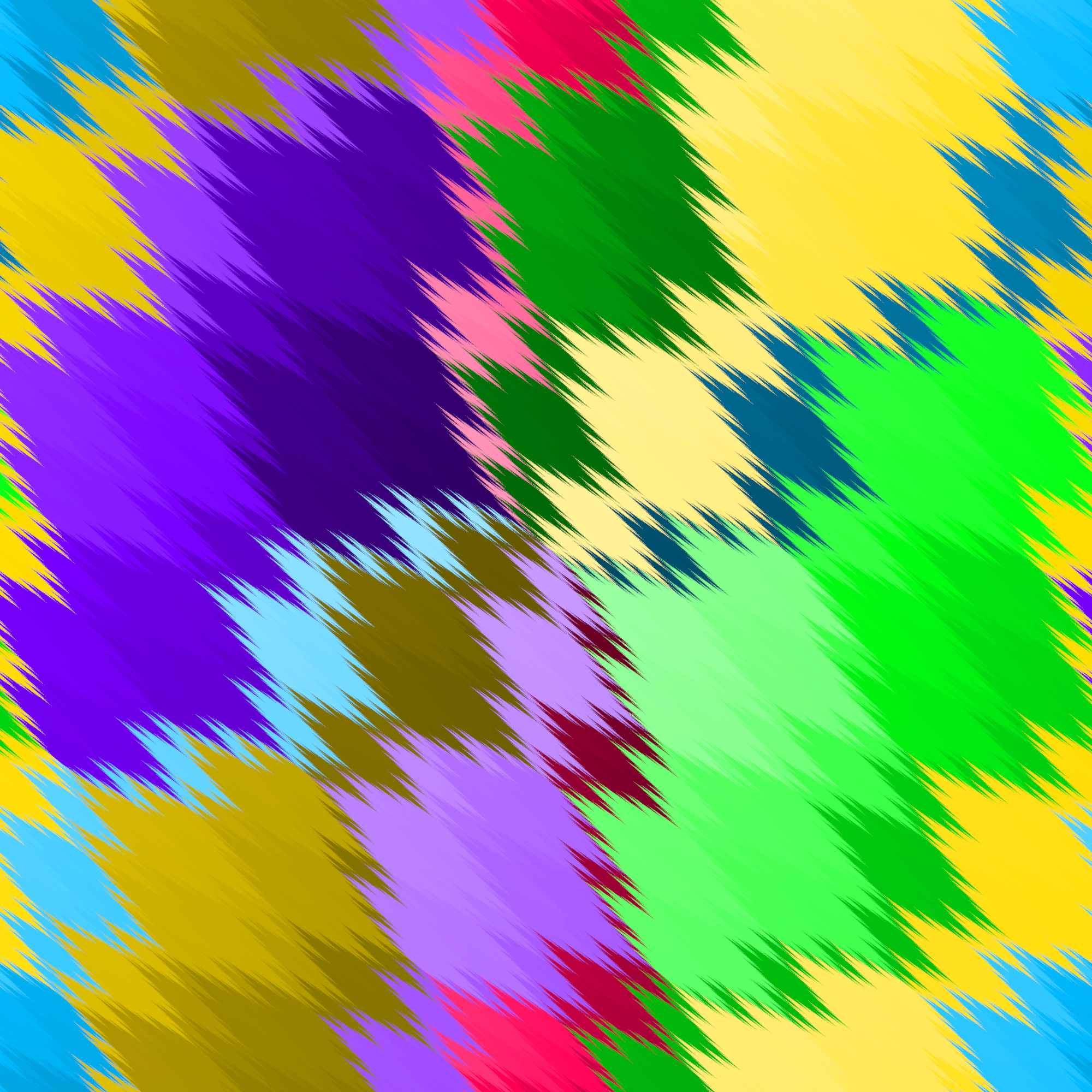}
    \caption{The square represents the $2$-torus $\{0\} \times \TT^2$, which is the image of $L$ under $\bPsi_{\mathbf c}$. We use the square fundamental domain with $\mathbf 0$ placed at the center.
    Points $x \in L$ are colored in Figure \ref{fig:level set straightening}. In this figure, we color $\bPsi_{\mathbf c}(x)$ the same color as $x \in L$.}
    \label{fig:torus}
\end{figure}

\begin{rmk}[Drawing Figure \ref{fig:torus}]
We computed the picture of $\bPsi_{\mathbf c}|_L:L \to \R^2$ in Figure \ref{fig:torus} using the results above and the ideas from section \ref{sect:computing eigenfunctions}. Using the semi-conjugacy from $\phi_{1/c}|_L:L \to L$ and the IET $\tau:[0,1) \to [0,1)$, observe we can draw the image of a semi-conjugacy from $\tau$ to the toral translation $R$. Let $h:[0,1) \to \TT^2$ denote this semi-conjugacy. We know that the boundary points of the colored intervals (which correspond to the singularity on $S$) should be sent to $\mathbf 0 \in \TT^2$. If $x_0 \in [0,1)$ is one of these boundary points, then we have $R^n(\mathbf 0) = h \circ \tau^n(x_0)$ for all $n \in \Z$. Since orbits of $R$ are dense, we can divide the torus into small squares, and find for each square $Q \subset \TT^2$ an $n_Q \in \Z$ such that $R^n(\mathbf 0) \in Q$. We then color the associated pixel to match the color of $\tau^n(x_0)$ in Figure \ref{fig:iet}.
\end{rmk}
\paragraph*{\bf Measurable isomorphisms}
We believe it should be the case that the semiconjugacy from $\bPsi_{\mathbf c}|_L$ to a toral automorphism is in fact a measurable isomorphism; see the discussion of the Pisot conjecture for flows in the introduction. We hope to address question in both this special case and in greater generality in a future work. Specifically, we wonder:

\begin{ques}
    \label{conj:conjugacy}
    In statement (3) of Theorem~\ref{theorem:pisotsemiconjugacy} where $\mathbf c=(c_1, \ldots, c_d)$ is a generating collection of eigenvalues, is the resulting map $\Psi_c$ always a measurable isomorphism from $S$ to $\TT^d$?
\end{ques}

\paragraph*{\bf Restriction to a level set} From this, it follows that the map from a level set to $\TT^{d-1}$ is also a measurable isomorphism. If this holds, then in this special case the following four systems would all be measurably isomorphic:
\begin{itemize}
    \item The map $\psi_{1/c}: L \to L$, equipped with normalized transverse measure.
    \item The first return map $r:U \to U$, equipped with normalized transverse measure.
    \item The IET $\tau:[0,1) \to [0, 1)$, with Lebesgue measure.
    \item The toral translation $R:\TT^2 \to \TT^2$, with Lebesgue measure.
\end{itemize}

\paragraph*{\bf Mercat's semi-algorithm} Mercat~\cite{Mercat:semialgorithm}, has implemented a \emph{semi-algorithm} to determine whether a substitution system has pure discrete spectrum, based on ideas of Rauzy~\cite{Rauzy}, and explained further in~\cite{Fogg}. There is also a more combinatorial algorithm based on ideas of Livshits~\cites{Livshits1, Livshits2}, as explained by Sirvent-Solomyak~\cite{SirventSolomyak}. In either case, these semi-algorithms terminate if the subsitution system does indeed have discrete spectrum. While $\tau$ is not periodic under Rauzy induction, it is pre-periodic, so the algorithm must be modified before applying it to the IET $\tau$. Mercat has done this, and indeed, the semi-algorithm verifies that $\tau$  has pure discrete spectrum, and thus $\tau$ is conjugate to the toral translation $R$. This implies then that the linear flow on the surface and the linear flow on the torus are conjugate, as are the pseudo-Anosov map and hyperbolic toral automorphism. In principle, we can apply Mercat's semi-algorithm (with appropriate adjustments to deal with pre-periodicity versus periodicity) to any of the semi-conjugacies implied by Theorem~\ref{theorem:pisotsemiconjugacy}, to try and answer Conjecture~\ref{conj:conjugacy}, but there is no guarantee that the algorithm terminates in any particular case.

\end{document}